\documentclass[aop,MSNbibl,dvips]{arximspdf}
\usepackage{subeqn}
\usepackage{graphicx}

\doi{10.1214/12-AOP781} %kopijuoti is PTS
\volume{41}
\issue{6}
\pubyear{2013}
\firstpage{4359}
\lastpage{4406}

\makeatletter
\def\underline{\textit}

\newcommand{\rright}{\right}
\newcommand{\lleft}{\left}
\newcommand{\angler}{\rangle}
\newcommand{\anglel}{\langle}
\newcommand{\rrVert}{\Vert}
\newcommand{\llVert}{\Vert}
\newcommand{\eqref}[1]{(\ref{#1})}

\newcommand{\mathrmm}{}
\newproclaim{defn}{Definition}[section]
\newtheorem{lemma}{Lemma}[section]
\newtheorem{thmm}{Theorem}[section]
\newproclaim{rmk}{Remark}[section]
\newtheorem{prop}{Proposition}[section]
\newtheorem{cor}{Corollary}[section]
\newtheorem{rhp}{Riemann--Hilbert Problem}[section]

\newcommand{\R}{\mathbb{R}}
\newcommand{\C}{\mathbb{C}}
\newcommand{\Z}{\mathbb{Z}}
\newcommand{\Prob}{\mathbb{P}}

\newcommand{\eg}{e.g.}
\newcommand{\M}{\mathcal{M}}
\newcommand{\N}{\mathcal{N}}
\newcommand{\NN}{\mathcal{N}_1}
\newcommand{\U}{\mathcal{U}}
\newcommand{\bY}{\mathbf{Y}}
\newcommand{\bQ}{\mathbf{Q}}
\newcommand{\bR}{\mathbf{R}}
\newcommand{\bS}{\mathbf{S}}
\newcommand{\bT}{\mathbf{T}}
\newcommand{\bA}{\mathbf{A}}
\newcommand{\bpsi}{\bolds{\Psi}}

\newcommand{\frakA}{\mathfrak{A}}
\newcommand{\frakB}{\mathfrak{B}}
\newcommand{\frakC}{\mathfrak{C}}
\newcommand{\frakD}{\mathfrak{D}}

\newcommand{\dd}{\,d}

\newcommand{\sig}{ {\sigma_3} }
\newcommand{\tbyt}[4]{\pmatrix{ #1 & #2 \vspace*{2pt}\cr #3 & #4 }}
\newcommand{\triu}[2][1]{\pmatrix{ #1 & #2  \vspace*{2pt}\cr0 & #1 }}
\newcommand{\tril}[2][1]{\pmatrix{ #1 & 0  \vspace*{2pt}\cr #2 & #1 }}
\newcommand{\diag}[2]{\pmatrix{ #1 & 0  \vspace*{2pt}\cr0 & #2 }}
\newcommand{\offdiag}[2]{\pmatrix{ 0 & #1  \vspace*{2pt}\cr #2 & 0 }}

\newcommand{\eps}{\varepsilon}
\newcommand{\tg}{\tilde\gamma}
\newcommand{\ty}{\tilde y}
\newcommand{\ts}{\tilde s}
\newcommand{\CC}{\mathcal{K}}

\newcommand{\Next}{E}

\newcommand{\FGUE}{F_{\mathrm{GUE}}}
\newcommand{\FGOE}{F}
\newcommand{\bfY}{\mathbf{Y}}

\newcommand{\UU}{\mathcal{U}}
\newcommand{\QQ}{\mathcal{Q}}
\newcommand{\RR}{\mathcal{R}}
\newcommand{\SSS}{\mathcal{S}}
\newcommand{\Ical}{\mathcal{I}}
\newcommand{\sector}{\mathcal{S}}

\newcommand{\re}{\operatorname{Re}}
\newcommand{\Ai}{\operatorname{Ai}}

\newcommand{\cro}{\mathrm{cr}}
\newcommand{\Cro}{\mathrm{CR}}
\newcommand{\Nes}{\mathrm{NE}}
\newcommand{\res}{\operatorname{Res}}
\newcommand{\Inv}{\operatorname{Inv}}
\newcommand{\Cov}{\operatorname{Cov}}
\newcommand{\Cor}{\operatorname{Cor}}
\makeatother

\begin{document}
\begin{frontmatter}

\title{Limiting distribution of maximal crossing and nesting of
Poissonized random matchings}
\runtitle{Crossing and nesting}

\begin{aug}
\author[A]{\fnms{Jinho} \snm{Baik}\ead[label=e1]{baik@umich.edu}\thanksref{t2}}
\and
\author[A]{\fnms{Robert} \snm{Jenkins}\corref{}\ead[label=e2]{rmjenkin@umich.edu}}
\runauthor{J. Baik and R. Jenkins}
\affiliation{University of Michigan}
\address[A]{Department of Mathematics\\
University of Michigan\\
Ann Arbor, Michigan 48109\\
USA\\
\printead{e1}\\
\phantom{E-mail:\ }\printead*{e2}} %adresu isvedimo komanda gale!
\thankstext{t2}{Supported by NSF Grants DMS-075709 and DMS-10-68646.}
\end{aug}

% HISTORY:
\received{\smonth{11} \syear{2011}}

% ABSTRACT
%
\begin{abstract}
The notion of $r$-crossing and $r$-nesting of a complete
matching was introduced and a symmetry property was proved by Chen et
al. [\textit{Trans. Amer. Math. Soc.} \textbf{359} (2007) 1555--1575].
We consider random matchings of large size and study their maximal
crossing and their maximal nesting.
It is known that the marginal distribution of each of them converges to
the GOE Tracy--Widom distribution.
We show that the maximal crossing and the maximal nesting becomes
independent asymptotically, and we evaluate the joint distribution for
the Poissonized random matchings
explicitly to the first correction term.
This leads to an evaluation of the asymptotic of the covariance.
Furthermore, we compute the explicit second correction term in the
distribution function
of two objects: (a)~the length of the longest increasing subsequence of
Poissonized random permutation
and (b) the maximal crossing, and hence also the maximal nesting, of
Poissonized random matching.
%(The proof is based on a Riemann--Hilbert analysis in which we utilize
%a new construction of the Painlev\'e local %parametrix following the
%recent work of Buckingham and Miller, which is expected to be useful
%in other Riemann--Hilbert %analysis as well.)
\end{abstract}

% KEYWORDS

\end{frontmatter}

%s1 #&#
\section{Introduction} \label{sectionintroduction}

%stucture.}

%of GUE, GOE; see Tracy's student's papers.}

Let $\M_n$ be the set of complete matchings of $[2n]$. The size of $\M_n$ is $(2n-1)!!$.
It is well known that the number of complete matchings of $[2n]$ with
no crossings equals the $n$th Catalan number $C_n$, as is the number of
complete matchings with no nestings.
In \cite{5}, a notation of $r$-crossing and $r$-nesting was introduced:
given a complete matching $M=\{(i_1,j_1),\ldots,(i_n,j_n) \} \in\M_n$,
$\{ (i_{s_1}, j_{s_1}), \ldots,(i_{s_r},j_{s_r}) \}$ is called an
\underline{$r$-crossing}
if $i_{s_1} < i_{s_2} < \cdots< i_{s_r} < j_{s_1} < \cdots< j_{s_r}$ and
an \underline{$r$-nesting} if $i_{s_1} < i_{s_2} < \cdots< i_{s_r} <
j_{s_r} < \cdots< j_{s_2} < j_{s_1}$.
Let $\mathrm{cr}_n(M)$ be the largest number $k$ such that $M$ has a
$k$-crossing (maximal crossing) and $\mathrm{ne}_n(M)$ denote the largest
number $j$ such that $M$ has a $j$-nesting (maximal nesting).
See Figure~\ref{figmatching} for an example.
Various combinatorial properties of $\mathrm{cr}_n$ and $\mathrm
{ne}_n$ were studied
by Chen et al. in \cite{5}.
This paper subsequently generated a flurry of research concerning
crossings and nestings of many combinatorial objects; see, for example,
\cite{13} and also the survey article \cite{15}.

%
%f1 #&#
\begin{figure}

\includegraphics{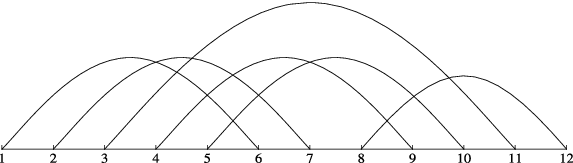}

\caption{A complete matching $M$ of [12]. In this sample $\mathrm
{cr}_6(M) =
4$, achieved by
$\{(1,6), (2,7),  (4,9), (5,10)\}$, and $\mathrm{ne}_6(M) = 2$,
achieved by $\{
(3,11), (4,9)\}$.}\label{figmatching}
\end{figure}

We may equip $\M_n$ with the uniform probability and regard $\mathrm{cr}_n$
and $\mathrm{ne}_n$ as random variables.
Let $\N$ be a Poisson random variable with parameter $t^2/2$ and
consider matchings of random size distributed as $2\N$. Let $\Cro_t$
and $\Nes_t$ denote $\mathrm{cr}_\N$ and $\mathrm{ne}_\N$,
respectively. The object
of this paper is to study the asymptotics of $\Cro_t$ and $\Nes_t$ as
$t\to\infty$.

%later.}
One of the main results of \cite{5} is that the joint distribution of
$\mathrm{cr}_n$ and $\mathrm{ne}_n$ are symmetric. Hence $\Cro_t$
and $\Nes_t$ are
symmetrically distributed.
The limit of the marginal distribution of $\Nes_t$ can be obtained by
noting a bijection
between matchings and fixed-point-free involutions.
Let $\Inv_n$ be the set of permutations of size $2n$ consisting of only
2-cycles.
%the size}
%, \ie involutions with no fixed points.
%There is a natural bijection between $\Inv_n$ and $\M_n$.
To $\sigma\in\Inv_n$ whose cycles are $(i_1,j_1),\ldots,(i_n,j_n)$,
associate the complete matching $\{(i_1,j_1),\ldots,(i_n,j_n)\}$.
This gives a natural bijection $\varphi$ from $\Inv_n$ onto $\M_n$.
Moreover, if we define $\tilde\ell_n(\sigma)$ as the length of the
longest \emph{decreasing} subsequence of $\sigma\in\Inv_n$, it is easy
to check that $\tilde\ell_n(\sigma)/2=\mathrm{ne}_n(\varphi(\sigma))$.
The limiting distribution of $\tilde{\ell}_n$, and also of $\tilde
{\ell
}_\N$ were obtained obtained earlier in~\mbox{\cite{2,3}}.
From this and the symmetry of $\mathrm{cr}_n$ and $\mathrm{ne}_n$,
Chen et al. \cite
{5} concluded that
%we obtain (see Section 5 of \cite{5})
for each $x\in\R$,
%
%e1 #&#
\begin{equation}
\label{eq1-0} \lim_{n \to\infty} \Prob \biggl\{ \frac{ \mathrm
{cr}_n - \sqrt {2n} }{2^{-1}(2n)^{1/6} } \leq x
\biggr\} = \lim_{n \to\infty} \Prob \biggl\{ \frac{ \mathrm
{ne}_n - \sqrt{2n}
}{2^{-1}(2n)^{1/6} } \leq x \biggr\} =
\FGOE(x),
\end{equation}
where $\FGOE(x)$ is the GOE Tracy--Widom distribution function from
random matrix theory \cite{16} defined in~\eqref{FGOE} below.
We also find a similar result for the Poissonized version,
%
%e2 #&#
\begin{equation}
\label{eq1} \lim_{t \to\infty} \Prob \biggl\{ \frac{\Cro_t -
t}{2^{-1}t^{1/3}} \leq x \biggr
\} = \lim_{t \to\infty} \Prob \biggl\{ \frac{\Nes_t -
t}{2^{-1}t^{1/3}} \leq x \biggr\} =
\FGOE(x) .
\end{equation}

We note that the length $\ell_n(\sigma)$ of the longest \emph
{increasing} subsequence of $\sigma\in\Inv_n$ has a different
distribution from $\tilde\ell_n$. For example, while $\tilde\ell_n(\sigma)$
is always an even integer, $\ell_n(\sigma)$ can be both
even or odd integers.
Moreover, it was shown in \cite{3} that $\frac{ \ell_n/2 - \sqrt{2n}
}{2^{-1}(2n)^{1/6}}$ converges to a random variable whose distribution
function is different from $\FGOE$; it is given by the so-called GSE
Tracy--Widom distribution.
Hence the joint distribution of $\mathrm{cr}_n$ and $\mathrm{ne}_n$
cannot be the
joint distribution of $\ell_n/2$ and $\tilde\ell_n/2$.

%
%f2 #&#
\begin{figure}

\includegraphics{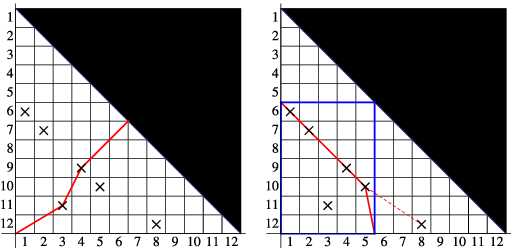}

\caption{The permutation matrix of the permutation
$\sigma$ %with cycles $(1,6), (2,7), (3,11), (4,9), (5,10), (8,12)$
corresponding to the matching
in Figure~\protect\ref{figmatching}. Since the matrix is symmetric,
only the lower triangular part is shown and the entries with element
$1$ are marked by $\times$.
\emph{On the left}:
The maximal up/right path (of length $2$) corresponding to $\mathrm{ne}_n(\varphi(\sigma))$.
%for the matching of $[12]$ depicted in Figure~\ref{figmatching}.
\emph{On the right}:
The maximizing down/right path (of length $4$) corresponding to
$\mathrm{cr}_n(\varphi(\sigma))$
% for the matching depicted in Figure~\ref{figmatching}
is realized by $\ell_{6}^{  6}$. Note that the longer down/right path
indicated by the dashed line is not allowed as it does not fit inside
the rectangles bounding the paths $\ell^{k}_6$ for any $k=1,\ldots,
12$.} \label{figpercolationpaths}
\end{figure}

A geometric meaning of $\mathrm{cr}_n(\varphi(\sigma))$ and $\mathrm
{ne}_n(\varphi
(\sigma))$ is the following.
Represent~$\sigma$ as a permutation matrix.
Geometrically we imagine the square of size $2n$ with $(1,1)$ entry at
the top left corner.
The condition that $\sigma$ consists of only 2-cycles implies that
the matrix is symmetric and the diagonal entries are zeros.
Then it is easy to see that $\mathrm{ne}_n(\varphi(\sigma))= \tilde
\ell_n(\sigma)/2$ equals the length of the longest up/right chain
consisting of $1$'s in the lower-triangle $\{(i,j)\dvtx 1\le j<i\le2n\}$.
On the other hand, for each $k=1,\ldots, 2n$, let $\ell_n^k(\sigma)$
denotes the length of the longest down/right chain consisting of $1$'s
in the rectangle with two opposite corners $(2n,1)$, $(k,k)$.
Then $\mathrm{cr}_n(\varphi(\sigma))$ equals the maximum of $\ell_n^k(\sigma)$
over $k=1, \ldots, 2n$ \cite{13}; see Figure~\ref{figpercolationpaths}.

%s1.1 #&#
\subsection{Joint distribution}

The first main result of this paper is the following result for the
joint distribution of $\Cro_t$ and $\Nes_t$.
Let $F(x)$ denote the GOE Tracy--Widom distribution defined by
%
%e3 #&#
\begin{equation} \label{FGOE}
\FGOE(x) := \exp \biggl[ \frac{1}{2} \int_x^\infty
\bigl( u(s)-q(s) \bigr) \,ds \biggr],\qquad u(x) := \int
_{\infty}^x q(s)^2 \,ds , %\label{udef}
\end{equation}
where $q(s)$ is the unique solution of Painlev\'e II, $q''(s) = s q(s)
+ q(s)^3$, such that $q(s) \sim\Ai(s)$ as $s \to\infty$ (where
$\Ai$ denotes the Airy function).
The solution $q(s)$ is called the Hastings--McLeod solution \cite{HM};
see also \cite{FIK}.

%
%th1.1 #&#
\begin{thmm}\label{thm3}
Set
%
%e4 #&#
\begin{equation}
\label{eq9} \tilde\Cro_t := \frac{\Cro_t - t}{2^{-1}t^{1/3}},\qquad \tilde
\Nes_t = \frac{\Nes_t - t}{2^{-1}t^{1/3}}.
\end{equation}
%
%For each $x, x' \in\R$,
% \lim_{t \to\infty} \Prob\left\{ \tilde\Cro_t \leq x, \tilde\Nes_t
%Moreover,
We have
%
%e5 #&#
\begin{eqnarray}
\label{eq11} &&\Prob \bigl\{ \tilde\Cro_t \leq x, \tilde
\Nes_t \leq x' \bigr\}
\nonumber
\\[-8pt]
\\[-8pt]
\nonumber
&&\qquad=\Prob\{ \tilde\Cro_t < x \} \Prob \bigl\{ \tilde
\Nes_t < x' \bigr\} %\\
%&= \FGOE(x) \FGOE(x')
+
\frac{\FGOE'(x)\FGOE'(x')}{t^{2/3}} + \mathcal{O} \bigl( t^{-1} \bigr).
\end{eqnarray}
\end{thmm}

This, together with a tail estimate, implies the asymptotics of the covariance.

%
%co1.1 #&#
\begin{cor}\label{cor1}
% %
% &\Cov\left( \tilde\Cro_t , \tilde\Nes_t \right)
% = \frac{1}{t^{2/3}} + \bigo{t^{-1}}.
% %
The covariance of $\Cro_t$ and $\Nes_t$ satisfies
%
%e6 #&#
\begin{equation}
\label{eq11-1-2} \Cov( \Cro_t , \Nes_t ) =
\tfrac{1}{4} + \mathcal{O} \bigl( t^{-1/3} \bigr).
\end{equation}
Hence, the correlation is asymptotically
%
%e7 #&#
\begin{equation}
\label{eq11-1-3} \rho( \Cro_t , \Nes_t ) =
\frac{1}{\sigma^2t^{2/3}} + \mathcal{O} \bigl( t^{-1} \bigr),
\end{equation}
where $\sigma^2= 1.6077810345\ldots$ is the variance of $F(x)$; cf.
page~862 of \cite{Bornemann}.
% %
% \sigma^2= \int_{-\infty}^\infty x^2 dF(x) - \bigg[ \int_{-\infty}^
% %
\end{cor}

We can also interpret $\Cro_t$ and $\Nes_t$ as ``height'' and ``depth''
of certain nonintersecting random walks. See Section~\ref{secpaths} below.

%
%t1 #&#
\begin{table}[b]
\caption{The exact correlation and covariance of $\mathrm{cr}_n$ and
$\mathrm{ne}_n$
for complete matchings of $[2n]$ for the first few nontrivial $n$'s.
Note that both statistics are strictly negative}\label{tabexact}
\begin{tabular*}{\textwidth}{@{\extracolsep{\fill}}lccc@{}}
\hline
$\bolds{[2n]}$ & $\bolds{\#\mathcal{M}_n}$ & $\bolds{\Cov( \mathrm{cr}_n, \mathrm{ne}_n )}$
& $\bolds{\Cor( \mathrm{cr}_n,
\mathrm{ne}_n )}$ \\
\hline
\phantom{0}4 & \phantom{00000}3 & $-1/9$ & $-1/$2 \\
\phantom{0}6 & \phantom{0000}15 & $-0.137777777$ & $-0.418918919$ \\
\phantom{0}8 & \phantom{000}105 & $-0.129614512$ & $-0.362983698$ \\
10 & \phantom{000}945 & $-0.132998516$ & $-0.331342276$ \\
12 &\phantom{0}10395 & $-0.143259767$ & $-0.309871555$ \\
14 & 135135 & $-0.151180948$ & $-0.29369603$2 \\
\hline
\end{tabular*}
\end{table}

% \begin{center}
% \includegraphics[width=.8\textwidth]{Fixedsample} \\
% \includegraphics[width=.8\textwidth]{Poissonsample}
% \caption{Approximation of the marginal distributions of the scaled
%crossing and nesting numbers for (top) fixed sized matchings with
%$n=2500$ and (bottom) Poissonized matchings with $t=100.$ The solid
%line indicates the large $n$ limit, $F(x)$ (see \eqref{eq1-0}-
%$\Cov( \tilde{\Cro}_{100}, \tilde{\Nes}_{100} ) = 0.0143428$ in the
%Poissonized case.
% \label{figSamples}}
% \end{center}

We may apply the de-Poissonization argument \cite{10} to~\eqref{eq11}
to find a result for the joint distribution of $\mathrm{cr}_n$ and
$\mathrm{ne}_n$.
However, intuitively, for fixed $n$ and $M\in\M_n$, any $(i,j)\in M$
that is used to form the maximal crossing of $M$ cannot be used for
the maximal nesting of $M$. This indicates a negative correlation of
$\mathrm{cr}_n$ and $\mathrm{ne}_n$ for a fixed $n$, contrary to the positive
correlation of $\Cro_t$ and $\Nes_t$ found in the above corollary.
This is verified for small $n$ by direct computation: Table \ref{tabexact} shows
exact calculation of the covariance and correlation of $\mathrm{cr}_n$ and
$\mathrm{ne}_n$ for small values of~$n$.
%and we see explicitly that each is negative.
For large $n$, a sampling of 5000 pseudo-random matchings of $[5000]$
%%and again for 5000 samples of Poissonized matchings with $t=100$ (so
%that the expected size of the matching is 10,000)
yielded the sample covariance of $\tilde{\mathrm{cr}}_{2500}$ and
$\tilde{\mathrm{ne}
}_{2500}$ equal to $-0.0420258\ldots.$
%show that the covariance of $\cro_{n}$ and $\nes_{n}$ (for n=2500) is
%negative, while $\Cro_t$ and $\Nes_t$ (for t=100) is positive from~
%of each marginal distribution to the theoretical limit $F(x)$ are
%summarized in Figures \ref{figSamples}.
%For large $n$ the number of matchings makes exact calculation
%intractable. However, by sampling we can estimate the covariance and
%correlation. Numerical sampling of 5000 pseudo-random matchings of
%$[5000]$ and again for 5000 samples of Poissonized matchings with
%$t=100$ (so that the expected size of the matching is 10,000) show
%that the covariance of $\cro_{n}$ and $\nes_{n}$ (for n=2500) is
%negative, while $\Cro_t$ and $\Nes_t$ (for t=100) is positive from~
%of each marginal distribution to the theoretical limit $F(x)$ are
%summarized in Figures \ref{figSamples}.
Therefore, a~naive substitution of $t$ by $\sqrt{2n}$ in~\eqref{eq11}
only yields the following weaker result. A~further analysis is needed
to obtain the correction terms in the asymptotic behavior of $\mathrm{cr}_n$
and $\mathrm{ne}_n$.
A heuristic explanation for the positive correlation of the Poissonized
random matchings is that when $\Cro_t$ is large, it most likely due to
fact that the size of the matching is large, and hence the maximal
nesting of the matching is also likely to be large.

%A further analysis is needed to obtain a result similar to~

%implies the following weaker result for $\cro_n$ and $\nes_n$.

%
%co1.2 #&#
\begin{cor}\label{thm1}
Set
%
%e8 #&#
\begin{equation}
\label{eq2} \tilde \mathrm{cr}_n := \frac{\mathrm{cr}_n - \sqrt {2n}}{2^{-1}(2n)^{1/6}},\qquad \tilde
\mathrm{ne}_n = \frac{\mathrm{ne}_n - \sqrt{2n}}{2^{-1}(2n)^{1/6}}.
\end{equation}
For each $x, x' \in\R$,
% \lim_{n \to\infty} \Prob\left\{ \tilde\cro_n \leq x, \tilde\nes_n
%Moreover, we have
%
%e9 #&#
\begin{equation}
\label{eq4} \Prob \bigl\{ \tilde \mathrm{cr}_n \leq x, \tilde
\mathrm{ne}_n \leq x' \bigr\} = \Prob\{ \tilde
\mathrm{cr}_n < x \} \Prob \bigl\{ \tilde\mathrm{ne}_n <
x' \bigr\} %\\
%&= \FGOE(x)
+ \mathcal{O} \biggl(
\frac{\sqrt{\log n}}{n^{1/6}} \biggr).
\end{equation}
\end{cor}

%We may apply the de-Poissonization argument \cite{10} to~\eqref{eq11}
%obtain a result for the joint distribution of $\cro_n$ and $\nes_n$.
%However, a straightforward application of the de-Poissonization
%argument only implies the following result.
%distribution of $\cro_n$ and $\nes_n$.

We compare Theorem~\ref{thm3} with the result of %Bornemann
\cite{Bornemann10} on the joint distribution of the extreme eigenvalues
of Gaussian unitary ensemble (GUE). Let $\lambda_{\max}^{(n)}$ and
$\lambda_{\min}^{(n)}$ denote the largest and the smallest eigenvalues
of $n\times n$ GUE.
Setting
%
%e10 #&#
\begin{equation}
\label{eqG-1}\quad  \tilde{\lambda}_{\max}^{(n)}:=
2^{1/2}n^{1/6} \bigl(\lambda_{\max}^{(n)}-
\sqrt{2n} \bigr),\qquad \tilde{\lambda}_{\min}^{(n)}:=
2^{1/2}n^{1/6} \bigl(\lambda_{\min}^{(n)}+
\sqrt{2n} \bigr),
\end{equation}
it was shown in \cite{Bornemann10} that
%
%e11 #&#
\begin{eqnarray}
\label{eqG-2} &&\Prob \bigl\{ \tilde{\lambda}_{\max}^{(n)} \leq
x, \tilde{\lambda}_{\min}^{(n)} \leq x' \bigr\}
\nonumber
\\[-8pt]
\\[-8pt]
\nonumber
&&\qquad= \Prob \bigl\{ \tilde{\lambda}_{\max}^{(n)} < x \bigr\}
\Prob \bigl\{ \tilde{\lambda}_{\min}^{(n)} < x'
\bigr\} %\\
%&= \FGOE(x)
+ \frac{\FGUE'(x)\FGUE'(x')}{4n^{2/3}} + \mathcal{O} \bigl(
n^{-4/3} \bigr),
\end{eqnarray}
where $\FGUE$ is the \emph{GUE} Tracy--Widom distribution function
defined by
%
%e12 #&#
\begin{equation}
\FGUE(x) := \exp \biggl[ \int_x^\infty u(s) \,ds
\biggr]. \label{FGUE}
\end{equation}
%
%This indicates that i
It is interesting to study the joint distribution of the extreme
eigenvalues of Gaussian orthogonal ensemble (GOE)
and compare the result with~\eqref{eq11}. This will be done in a
separate paper.
%Also
It might also be interesting to see if the error term of~\eqref{eq11}
can be improved to $\mathcal{O}  ( t^{-4/3}  )$ as
in~\eqref{eqG-2}, but we do
not pursue this in this paper.
%But since this requires a further detailed analysis, we do not pursue
%this direction in this paper.
%where $u$ is defined in~\eqref{udef}.

%s1.2 #&#
\subsection{Marginal distribution}

We also evaluate the second order term in the asymptotics expansion of
the marginal distributions of $\Cro_t$ and $\Nes_t$ explicitly.
Let $[a]$ denote the largest integer less than or equal to $a$.\eject

%
%th1.2 #&#
\begin{thmm}\label{thm4}
For $x\in\R$ and $t>0$, define $x_t$ by
%
%e13 #&#
\begin{equation}
\label{eqm-1} x_t:= \frac{[t+2^{-1}xt^{1/3}]-t}{2^{-1}t^{1/3}} + \frac 1{t^{1/3}}.
\end{equation}
For each $x \in\R$,
%
%e14 #&#
\begin{eqnarray}
\Prob\{ \tilde\Cro_t \leq x \} &=& \Prob\{ \tilde
\Nes_t \leq x \}
\nonumber
\\[-8pt]
\\[-8pt]
\nonumber
&=& \FGOE(x_t) - \frac{1}{20 t^{2/3}} \biggl[ 4F''(x)+
\frac13 x^2 F'(x) \biggr] + \mathcal{O} \bigl(
t^{-1} \bigr).
\end{eqnarray}
%
%%
% &\log\Prob\left\{ \tilde\Cro_t \leq x \right\} = \log\Prob\left
% = \log\FGOE(x_t) + \frac{ \Next(x)}{t^{2/3}} + \bigo{t^{-1}}
%%
%where $\Next=\Next(x)$ equals
% \Next:= \frac{1}{20} \left[ -(u(x)-q(x))^2 +2 \lp u'(x) - q'(x) \rp+
\end{thmm}

%It is easy to see that $\Next(x) = \frac1{20 F(x)} ( - 4F''(x)-\frac13
%x^2 F'(x)).$

Note that since $\Prob\{ \Cro_t\le x \}$ has the same value for $x\in
[\ell, \ell+1)$ for a given integer~$\ell$, it is natural that the
leading term $F(x_t)$ of~\eqref{eq12} is expressed in terms of~$x_t$,
which contains $[t+2^{-1}xt^{1/3}]$.

In addition to this integral part correction, there is an additional
shift by $t^{-1/3}$ from $x$ in the definition of $x_t$. This is
responsible for the absence of the term of order $t^{-1/3}$ in the
expansion~\eqref{eq12}.
For classical ensembles in random matrix theory, there are several
papers that showed that a fine scaling can remove such a term (which
looks like a natural term to be present.) See \cite{ElKaroui} for the
Laguerre unitary ensemble, \cite{Johnstone} for Jacobi unitary and
orthogonal ensembles, \cite{Ma} for the Laguerre orthogonal ensemble
and \cite{JohnstoneMa} for Gaussian unitary and orthogonal ensembles.
A~similar result was obtained recently for random growth models and
intersecting particle systems in \cite{FerrariFrings}, including the
height of the so-called PNG model with flat initial condition.
It is well known that this is precisely the length of the longest
decreasing subsequence of random fixed-point-free involution and hence
$\Nes_t$. The result of \cite{FerrariFrings} in the context of this
paper is that $\Prob\{ \tilde\Nes_t \leq x \} = \FGOE(x_t) +
O(t^{-2/3})$. The above result finds the term of order $O(t^{-2/3})$ explicitly.

As in the joint distribution,
the evaluation of the second order term of $\Prob\{ \tilde\mathrm
{cr}_n \le
x \}$ does not immediately follow from the de-Poissonization argument
in \cite{10}.
It remains an open problem to evaluate the the error terms of $\Prob\{
\tilde\mathrm{cr}_n \le x \}$ asymptotically.

%s1.3 #&#
\subsection{Toeplitz minus Hankel with a discrete symbol}\label{secdet}

Set
%
%e15 #&#
\begin{equation}
\label{eq7} G_{k,j}(t) := \sum_{n=0}^\infty
g_{k,j}(n) \frac{t^{2n}}{(2n)!},
\end{equation}
where $g_{k,j}(n) := \# \{ M \in\M_n  \dvtx \mathrm{cr}_n(M) \leq k,
\mathrm{ne}_n(M)
\leq j \}$
so that
%
%e16 #&#
\begin{eqnarray}
\label{eq8} %\label{eq6}
\Prob\{ \Cro_t \leq k, \Nes_t
\leq j \} &=& \sum_{n=0}^\infty\Prob\{
\cro_{\N} \leq k, \mathrm{ne}_{\N} \leq j | \N=n \} \Prob\{
\N= n \}
\nonumber
\\[-8pt]
\\[-8pt]
\nonumber
% & = \sum_{n=0}^\infty\frac{g_{k,j}(n)}{(2n-1)!!}
&=& e^{-t^2/2} G_{k,j}(t).
\end{eqnarray}
%
% \Prob\left\{ \Cro_t \leq k, \Nes_t \leq j \right\} = e^{-t^2/2}
%G_{k,j}(t).
%Note that $G_{k,j}(t)$ is a generating function of $\{ g_{k,j}(n)
%A determinantal formula for $G_{k,j}(t)$ %was %the joint distribution
%of $\Cro_t$ and $\Nes_t$
%was obtained in \cite{5}.
%Theorem~\ref{thm3} is proved by using
An explicit determinantal formula of $G_{k,j}(t)$ %for the joint
%distribution
was obtained in \cite{5} which we describe now.

Stanley had shown earlier that matchings are in bijection with
oscillating tableaux of empty shape and of length $2n$; %obtained
%previously by Stanley
see Section 5 of \cite{5}.
This was further generalized to a bijection between partitions of a set
and so-called vacillating tableaux in \cite{5}.
In the same paper, it was shown that the maximal crossing (resp.,
nesting) of a partition
equals the maximal number
of rows (resp., columns) in any partitions appearing in the
corresponding vascillating tableau.

Since an oscillating tableau can be thought of as a walk in the chamber
of the affine Weyl group $\tilde C_n$,
$g_{k,j}(n)$ % := \# \left\{ M \in\M_n  :  \cro_n(M) \leq k,
equals the number of walks with $n$ steps from $(j, j-1,\ldots,2,1)$ to
itself in the chamber $0<x_j<\cdots< x_2<x_1<j+k+1$ where each step is
a unit coordinate vector or its negative in $\Z^j$.
The number of such walks was evaluated by Grabnier in \cite{9} using
the Gessel--Viennot method of evaluation of nonintersecting paths.
This result implies (see the displayed equation before (5.3) in \cite
{5}) that %$G_{k,j}(t)$ defined in~\eqref{eq7} is given by
%
%e17 #&#
\begin{equation}
\label{eq13} G_{k,j}(t) = \det \Biggl[ \frac{1}{m} \sum
_{r=0}^{2m-1} \sin \biggl( \frac{\pi r a}{m}
\biggr) \sin \biggl( \frac{\pi r b}{m} \biggr) e^{2t \cos(\pi r/m)}
\Biggr]_{a,b=1}^j,
\end{equation}
where
%
%e18 #&#
\begin{equation}
\label{eq14} m:= j+k+1.
\end{equation}
%
%See also Section~\ref{secpaths} below.

We prove Theorem~\ref{thm3} by analyzing the determinant~\eqref{eq13}
asymptotically.
For this purpose, we first re-formulate the determinant slightly.
By writing the product of the sine functions in terms of a sum of two
cosine functions and noting the realness of the entries,
we find that %the above determinant equals %is a determinant of a
%Toeplitz matrix minus a Hankel matrix:
%
%e19 #&#
\begin{equation}
\label{eq15} G_{k,j}(t) = \det[ h_{a-b} - h_{a+b}
]_{a,b=1}^j,
\end{equation}
where
%
%e20 #&#
\begin{equation}
\label{eq16} h_\ell:= %\frac{1}{2m} \sum_{r=0}^{2m-1} \cos\left(
\frac{1}{2m}\sum
_{r=0}^{2m-1} e^{-i\pi r \ell/m}
e^{2t \cos(\pi r /m)}.
\end{equation}
%
%Here, the last equality follows form the fact that $\sum_{r=0}^{2m-1}
%This formula is the starting point of the proof of the Theorems above.
This is the determinant of a Toeplitz matrix minus a Hankel matrix.
This structure is important in the asymptotic analysis.
An interesting feature of the above determinant is that the measure for
the Toeplitz determinant is not an absolutely continuous measure but a
discrete measure.
%Since the parameter $m$ tends to infinity in our analysis, we take
%In such a case, the analysis of the determinant is naturally related
%to discrete orthogonal polynomials.

Let $\omega:= e^{\pi i/m}$ be the primitive $2m$th root of unity.
Define the discrete measure
%
%e21 #&#
\begin{equation}
\label{eqD-1} d\mu_m(z):= \frac{1}{2m}\sum
_{r=0}^{2m-1} e^{t(z+z^{-1})} \delta_{\omega^r}(z)
\end{equation}
on the circle.
Let $\pi_{n,m}(z)$ be the monic orthogonal polynomial of degree $n$
with respect to $d\mu_m$, defined by the conditions
%
%e22 #&#
\begin{equation}
\oint_{|z|=1} z^{-\ell}
\pi_{n,m}(z) \,d\mu_m(z) = 0,\qquad 0 \leq\ell< n.
\end{equation}
We emphasize the dependence on $m$ since later we will use the notation
$\pi_{n,\infty}$ to denote the case when ``$m=\infty$;'' the orthogonal
polynomials with respect to the absolutely continuous measure
$e^{t(z+z^{-1})} \frac{dz}{2\pi iz}$.
%Let $N_n$ be the positive constant defined by
% N_n := \oint_{|z|=1} \overline{\pi_n(z)} \pi_n(z) \,d \mu_m(z) =
Note that $d\mu_m$ depends on the parameter $t$ and hence $\pi_{n,m}(z)$ also depends on $t$.
When we wish to emphasize this dependence on $t$, we write $\pi_{n,m}(z;t)$.

The fact that the $t$-dependence of the measure is from the factor
$e^{t(z+z^{-1})}$ implies the following basic formula, which is proved
in Section~\ref{secproofprop1} below.
Recall from~\eqref{eq14} that $m:=j+k+1$.

%
%pr1.1 #&#
\begin{prop}\label{prop1}
We have
%
%e23 #&#
\begin{eqnarray}
\label{eqprop11}&& \log\Prob\{ \Cro_t \leq k, \Nes_t \leq
j \}
\nonumber
\\[-8pt]
\\[-8pt]
\nonumber
%= \log\lp e^{-t^2/2} G_{k,j}(t) \rp\\
&&\qquad= \int_0^t \pi_{2j+1,m}(0;
\tau) \,d \tau+ \int_0^t \int
_0^s \QQ_j^m(\tau) \,d
\tau \,ds,
\end{eqnarray}
where
%
%e24 #&#
\begin{eqnarray}
\label{eqT-1} \QQ_j^m(\tau) &:=& - \bigl(
\pi_{2j,m}(0;\tau)\pi_{2j+2,m}(0;\tau) +\bigl|\pi_{2j+1,m}(0;
\tau)\bigr|^2 \bigr)
\nonumber
\\[-8pt]
\\[-8pt]
\nonumber
&&{}+ \pi_{2j,m}(0;\tau)\pi_{2j+2,m}(0;\tau)\bigl|\pi_{2j+1,m}(0;
\tau)\bigr|^2 .
\end{eqnarray}
\end{prop}

We obtain the asymptotics %of $\Prob\left\{ \Cro_t \leq k,  \Nes_t
$\pi_{2j+\ell, m}(0,\tau)$ for $\ell=0,1,2$ %Note that since $\tau
%(0, t)$, we need to evaluate the asymptotics for a
by using the associated discrete version of the Riemann--Hilbert
problem; see, for example, \cite{BKMM}.
% as explained in Section~\ref{secRHP-start}.
See Sections~\ref{secRHP-start},~\ref{secRHPanalysis} and~\ref
{secRHPpainleveregime} below.

We compare the analysis of this paper based on the formula~\eqref
{eqprop11} with the analysis of
the determinant of a similar Toeplitz minus Hankel matrix in \cite{3}.
Even though
the determinant in \cite{3} was for continuous measure (which is
precisely the one for the marginal distribution of $\Nes_t$; see
Section~\ref{secproofprop1} below),
the basic structure of the matrix is the same; a Toeplitz minus a
Hankel matrix.
Denoting the matrix by $D_j$, the approach of \cite{3} was to write
$D_j= D_{\infty} \prod_{n=j}^\infty\frac{D_n}{D_{n+1}}$ where
$D_\infty$ is the strong Szeg\"o limit, which exists in that particular
case, and analyze $D_n/D_{n+1}$, which can be evaluated from the
Riemann--Hilbert problem for the $n$th orthogonal polynomial.
For our case, since the measure is discrete, the strong Szeg\"o limit
does not apply.
% and $D_{\infty}$ does not exist.
Indeed $D_n=0$ for all large enough $n$.
Then alternatively one can still analyze $D_j$ by expressing $D_j=
D_0\prod_{n=1}^j \frac{D_n}{D_{n-1}}$ as was done in
\cite{BBD}.
However, this expression is more subtle to analyze since $\log
(D_n/D_{n+1})$ is not small when~$n$ is small (indeed it grows as $n$
decreases when $t$ is proportional to $j$) and this requires careful
cancellations of the terms in the product.
Though this was done for the leading term in \cite{BBD}, the evaluation
of the lower terms in the asymptotic expansion in this method becomes
more complicated.
%(In \cite{BBD}, only the leading term was obtained from this approach.)
A particularly useful point in using formula~\eqref{eqprop11} is that
we only need to consider the so-called full band case (and the
transitional case when a gap and a saturated region are about to open
up) in the Riemann--Hilbert analysis.
This makes the analysis much simpler, and it becomes easier to evaluate
the lower order terms.
On the contrary, if we use the expression $D_j= D_0\prod_{n=1}^j \frac
{D_n}{D_{n-1}}$, then we need to consider
both the so-called void-band case and the saturation-band case,
including the transitional cases, in the Riemann--Hilbert analysis (and
this is the reason for the need of cancellations mentioned above.)

The continuous Riemann--Hilbert problem for $\pi_{n, \infty}(z;t)$ was
analyzed asymptotically to the leading term in \cite{BDJ,2,BBD}.
We expand this work to the discrete counterpart and moreover, we
improve the analysis so that we compute explicit formulae for the first
three terms in the expansion of the solution in both the discrete and
continuous cases.
As a technical note, we remark that we use a different local map for
the so-called Painlev\'e parametrix related to the local problem for
the Riemann--Hilbert problem from the previous cases \cite{BDJ,CK}.
We adapt the map used in the recent paper \cite{BMiller} for a
different parametrix,
which seems to be useful for further analysis in other Riemann--Hilbert
problems.
For the purpose of this paper, we only
analyze the full band case (and the transitional case) of the discrete
Riemann--Hilbert problem.
The analysis for the full parameter set of the discrete
Riemann--Hilbert problem will be discussed somewhere else in the
context of Ablowitz--Ladik equations and Schur flows in integrable
systems.\looseness=1

A determinantal formula of the marginal distribution $\Prob\{ \Nes_t
\leq j \}$ can be obtained from the joint distribution by taking $k\to
\infty$ while keeping $j$ fixed.
Then we find a Toeplitz minus a Hankel determinant with symbol
$e^{t(z+z^{-1})}$.
Here too, the factor of $e^{t(z+z^{-1})}$ in the limiting measure
implies a formula for the marginal distribution analogous to~\eqref
{eqprop11}. See Section~\ref{secproofprop1} below.

The Toeplitz determinant with symbol $e^{t(z+z^{-1})}$ is known to be
describe the distribution of the length of the longest increasing
subsequence of a random permutation \cite{Gessel}.
By using a formula similar to~\eqref{eqprop11}, the analysis of this
paper implies the following result.

%s1.4 #&#
\subsection{Longest increasing subsequence of random
permutation}\label{secellt}

%$2j+1$. What about $2j$?}

Consider the symmetric group $S_n$ of permutations of size $n$ and
equip it with the uniform probability.
Let $l_n(\pi)$ denote the length of the longest increasing subsequence
of $\pi\in S_n$.
Let $\NN$ be a Poisson random variable with parameter $t^2$ and
let $L_t$ denotes $l_{\NN}$.
It was shown in \cite{BDJ} that
$ \frac{ L_t - 2t}{t^{1/3}}$ converges to the \emph{GUE} Tracy--Widom
distribution~\eqref{FGUE}.
%Since the distribution function of $L_t$ is known to be expressible in
%terms of a Toeplitz determinant of symbol $e^{t(z+z^{-1})}$, the
%analysis
%The analysis of this paper allows us to
We evaluate the next term of the asymptotic expansion explicitly.

%
%th1.3 #&#
\begin{thmm}\label{thm5}
For each $x \in\R$,
%
%e25 #&#
\begin{eqnarray}
\label{eqthm5-2} &&\Prob \biggl\{ \frac{ L_t - 2t }{t^{1/3}}\leq x \biggr\} %\\
\nonumber
\\[-8pt]
\\[-8pt]
\nonumber
&&\qquad=
\FGUE \bigl(x^{(t)} \bigr)
- \frac{1}{10t^{2/3}} \biggl[ \FGUE''(x)+
\frac{1}{6} x^2 \FGUE'(x) \biggr] + \mathcal{O}
\bigl( t^{-1} \bigr),
\end{eqnarray}
%
% \log\Prob\left\{ \frac{ L_t - 2t }{t^{1/3}}\leq x \right\} %\\
% = \log\FGUE(x^{(t)}) \\
% - \frac{1}{10t^{2/3}} \left[ u(x)^2 - q(x)^2 - \frac{1}{6} x^2 u(x)
where
%
%e26 #&#
\begin{equation}
\label{eqtm5-2-1} x^{(t)}:= \frac{[2t+xt^{1/3}]-2t}{t^{1/3}}.
\end{equation}
\end{thmm}

The study in \cite{FerrariFrings} also considered the height of the
so-called PNG model with the droplet initial condition, which is
distributed precisely as $L_t$, and showed that
the above distribution function is $\FGUE(x^{(t)}) + O(t^{-2/3})$.
The above theorem evaluates the error term explicitly.

For the Gaussian unitary ensemble, Choup \cite{Choup2006,Choup2008}
evaluated the distribution of the largest eigenvalue explicitly up to
the term of order $O(n^{-2/3})$ which corresponds to
the term of order $t^{-2/3}$ in the above expansion.
It would be interesting to compare the term in the above theorem with
the formula of \cite{Choup2006,Choup2008}.

%s1.5 #&#
\subsection{Organization of paper}

In Section~\ref{secpaths}, we consider a nonintersecting random
process that gives rise to
$\Cro_t$ and $\Nes_t$. Proof of Proposition~\ref{secproofprop1} is
given in Section~\ref{secproofprop1}. The Riemann--Hilbert problem is
introduced in Section~\ref{secRHP-start}, and is analyzed
asymptotically in Sections~\ref{secRHPanalysis} and~\ref
{secRHPpainleveregime}.
Theorem~\ref{thm3} and Corollary~\ref{cor1} are proved in
Secton~\ref
{secGjkt}, and Theorems~\ref{thm4} and~\ref{thm5} are
proved in
Section~\ref{secmarginalproofs}.
We prove
Corollary~\ref{thm1} in Section \ref{secde-Poissonization} using a
de-Poissonization argument.
Finally, the Riemann--Hilbert problem for the Painlev\'e II equations
that are needed to model the local parametrix of the Riemann--Hilbert
problem for orthogonal polynomials are discussed in Section \ref{secpainleve}.

%s2 #&#
\section{Height and depth of nonintersecting continuous-time simple
random walks}\label{secpaths}

In Section~\ref{secdet} we discussed a relation between $\mathrm
{cr}_n$ and
$\mathrm{ne}_n$ and a walk in the chamber $\{0<x_j< \cdots<x_2
<x_1<j+k+1\}$
of the affine Weyl group $\tilde{C}_n$. In this section, we give an
interpretation of $\Cro_t$ and $\Nes_t$ in terms of the ``height'' and
``depth'' of continuous-time simple random walks.

Let $N^+(\tau)$ and $N^-(\tau)$ be two independent Poisson processes of
rate 1 and let $Z(\tau) := N^+(\tau) - N^-(\tau)$
be a continuous-time simple random walk. Then $Z(\tau)$ is an $\mathbb
{Z}$-valued Markov process with the transition probability $p_s(a,b) =
p_s(a-b)$ where
$p_t(a) = e^{-2t}\sum_{n \in\Z} \frac{t^{2n+a}}{n!(n+a)!}= p_t(-a)$
for $a\in\Z$.
% p_t(a) = e^{-2t}\sum_{n \in\Z} \frac{t^{2n+a}}{n!(n+a)!}= p_t(-a),
%  a\in\Z.
Here we used the convention that $1/ n! \equiv0$ if $n<0$. Set
%
%e27 #&#
\begin{equation}
\label{eq20} \phi(z) := \sum_{a \in\Z} \bigl(
e^{-2t} p_t(a) \bigr) z^{-a} = e^{t(z+z^{-1})}.
\end{equation}
Then we have
%
%e28 #&#
\begin{equation}
\label{eq21} p_t(a) = e^{-2t} \phi_{-a}=
e^{-2t} \phi_a, \qquad \phi_a:=
\oint_{|z|=1} z^{-a} \phi(z) \frac{dz}{2\pi i z} .
\end{equation}

Let $Z_i(\tau), i=0,1,2,\ldots,$ be independent copies of $Z(\tau)$,
and consider the infinite system of processes $X_i(\tau) = Z_i(\tau) -
i, i = 0,1,2,\ldots.$ Fix a number $t>0$.
We will consider the process conditioned on the event that (a) $X_i(t)
= X_i(0)$ for all $i$ and (b) $X_i(\tau)$ do not intersect in time
$[0,t]$, that is, $X_0(\tau) > X_1(\tau) > \cdots$ for all $\tau\in
[0,t]$. A precise interpretation will be given below.
Such nonintersecting continuous-time simple random walks have been
studied, for example, in \cite{OConnell02,1,AdlerFvM}.

Define the ``height'' $K := \max_{\tau\in[0,t]} X_0(\tau)$
and define the ``depth'' $J$ as the smallest index such that $X_i(\tau)
= i$ for all $\tau\in[0,t]$ and for all $i = J, J+1, \ldots,$ in other
words, only the top $J$ processes moved in the interval $[0,t]$. We are
interested in the joint distribution of $J$ and $K$ conditional of the
above event satisfying (a) and (b).
%that (a) $X_i(t) = X_i(0)$ for all $i$ and (b) $X_i(\tau)$ do not
%intersect in time $[0,t]$, \ie, $X_0(\tau) > X_1(\tau) > \dots$ for
%all $\tau\in[0,t]$, interpreted in the following way.

Precisely, fix $N \in\mathbb{N}$ and let $\frakA_N$ and $\frakB_N$ be
the events defined as
%
%e29 #&#
%e30 #&#
\begin{eqnarray}
\label{eq22} \frakA_N&:=& \bigl\{ X_i(t) =
X_i(0) = -i, i=0,1,\ldots, N-1 \bigr\},\\
%
%and let $\frakB_N$ be the event %that the $N$ processes are
%nonintersecting and stay above $-N$ during the time $t$, \ie
%
\label{eq23} \frakB_N&:=& \bigl\{ X_0(\tau)
>X_1(\tau) > \cdots> X_{N-1}(\tau) \geq-N+1, \tau \in[0,t]
\bigr\}.
\end{eqnarray}
The condition that $X_{N-1}(\tau) \geq-N+1$ for all $\tau\in[0,t]$
is natural %to add
because $J$ is likely to be a finite number and by definition of $J$,
$X_{J-1}(\tau) \geq X_{J-1}(0)$ for all $\tau\in[0,t]$. % and $J$ is
%likely to be finite.
The joint distribution of $K$ and $J$ is interpreted as
%
%e31 #&#
\begin{equation}
\label{eq24} P(k, j) := \lim_{N \to\infty} \Prob( K \leq k, J \leq j |
\frakA_N \cap\frakB_N ).
\end{equation}

%
%le2.1 #&#
\begin{lemma}
Let $K$ and $J$ be the ``height'' and ``depth,'' respectively, defined
above. Then
%The joint distribution of $J$ and $K$ is interpreted as
%
%e32 #&#
\begin{equation}
\label{eq24-1} P(k,j) = e^{-t^2/2} G_{k,j}(t),
\end{equation}
where $G_{k,j}(t)$ is given in~\eqref{eq15}.
\end{lemma}

\begin{pf}
We first evaluate $\Prob(\frakA_N \cap\frakB_N)$. The condition that
$X_i(\tau) > -N$, $i=0, \ldots, N-1$, implies that $X_i(\tau)$ has an
absorbing boundary at $-N$. Since the transition probability of $X_i$
with an absorbing boundary at $-N$ is $p_t(a,b) - p_t(-2N-a,b)$, the
Karlin--McGregor formula \cite{11} of nonintersecting probability
applied to continuous-time simple random walks (see, \eg, \cite
{AdlerFvM,1}) implies then that
%
%e33 #&#
\begin{eqnarray}
\label{eq25} %
\Prob( \frakA_N \cap \frakB_N
) &= &\det \bigl[ p_t(-a,-b)-p_t(-2N+a,-b)
\bigr]_{a,b=0}^{N-1}
\nonumber
\\[-8pt]
\\[-8pt]
\nonumber
% &= e^{-2tN} \det\left[ \phi_{a-b}-\phi_{2N-a-b}
&=& e^{-2tN} \det[ \phi_{a-b}-
\phi_{a+b} ]_{a,b=1}^{N}. %
\end{eqnarray}
%
%The last equality is obtained by setting $a \mapsto N-a$ and $b

Second, we evaluate $\Prob( \{ K \leq k, J \leq j \} \cap\frakA_N
\cap\frakB_N )$. We assume that $N$ is large so that $N \geq j$. By
the definition of $K$ and $J$,\vadjust{\goodbreak} the desired probability equals $\Prob
(\frakC\cap\frakD)$ where $\frakC$ and $\frakD$ are independent
events defined as follows.
$\frakC$ is the event that the top $j$ processes, $X_0(\tau), \ldots
,X_{j-1}(\tau)$, satisfy the two conditions (a) $X_i(t) = X_i(0)$ for
all $i = 0,\ldots,j-1$ and (b) $-j+1 \leq X_{j-1}(\tau) < \cdots<
X_0(\tau) \leq k$ for all $\tau\in[0,t]$, that is, the $j$
nonintersecting paths are not absorbed at the boundaries $-j$ and $k+1$.
$\frakD$ is the event that $X_i(\tau) = -i$ for all $i = j,j+1,\ldots
,N-1$ and for all $\tau\in[0,t]$, that is, the bottom $N-j$ processes
stay put during the interval $[0,t]$.
Clearly, $\Prob( \frakD) = ( e^{-2t} )^{N-j}$. On the other hand, from
the Karlin--McGregor formula again, $\Prob(\frakC) = \det[ \hat
p_t(-a,-b) ]_{a,b=0}^{j-1}$ where $\hat p_t(a,b)$ is the transition
probability of $Z(\tau)$ in the presence of the absorbing walls at $-j$
and $k+1$ in time $t$. It is easy to see that
%
%e34 #&#
\begin{equation}
\label{eq26} \hat p_t(a,b) = \sum_{n \in\Z}
\bigl[ p_t(a+2nm,b) - p_t(-2j-a+2nm,b) \bigr],
\end{equation}
where $m:=j+k+1$.
% m:=j+k+1.
Now consider the identity $z^{-a} \phi(z)= \sum_{n\in\Z} \phi_{a+n}z^n$. Set $\omega:= e^{\pi i/m}$.
By inserting $z= \omega^r$, $r=0,1,\ldots, 2m-1$ and summing over $r$,
we find that
%
%e35 #&#
\begin{equation}
\label{eq28} \sum_{r=0}^{2m-1} \bigl(
\omega^r \bigr)^{-a} \phi \bigl(\omega^r \bigr)
= 2m \sum_{n\in\Z} \phi_{a+2mn},\qquad \omega:=
e^{\pi i/m}.
\end{equation}
Hence from~\eqref{eq21},~\eqref{eq26} becomes
%
%e36 #&#
\begin{equation}
\label{eq29} \hat p_t(a,b) %=\frac{ e^{-t} }{2m} \sum_{r=0}^{2m-1}
=
e^{-2t} ( h_{a-b} - h_{-a-b+2j} ),
\end{equation}
where
%
%e37 #&#
\begin{equation}
\label{eq30} h_a := \oint_{|z|=1} z^{-a} \,d
\mu_m(z),\qquad  d \mu_m(z) := \frac{1}{2m} \sum
_{r=0}^{2m-1} \phi(z) \delta_{\omega^r}(z).
\end{equation}
Hence, for $N\ge j$,
%
%e38 #&#
\begin{equation}
\label{eq30-1} \Prob \bigl( \{ K \leq k, J \leq j \} \cap\frakA_N
\cap\frakB_N \bigr) = e^{-2tN} \det[ h_{a-b}-h_{a+b}
]_{a,b=1}^{j}.
\end{equation}

The strong Szeg\"o limit theorem for Toeplitz minus Hankel determinants
(see, e.g., \cite{4}) implies that for the function $\phi(z)$
in~\eqref
{eq20}, $\det[ \phi_{a-b}-\break \phi_{a+b} ]_{a,b=1}^N \to e^{t^2/2}$ as
$N \to\infty$.
Therefore, from~\eqref{eq25} and~\eqref{eq30-1} we find that
%
%e39 #&#
\begin{equation}
\label{eq32}\qquad  P(j,k) = \lim_{N \to\infty} \frac{ \det[
h_{a-b}-h_{a+b} ]_{a,b=1}^j }{ \det[ \phi_{a-b}- \phi_{a+b}
]_{a,b=1}^N } = e^{-t^2/2}
\det[ h_{a-b} - h_{a+b} ]_{a,b=1}^j .
\end{equation}
This is~\eqref{eq24-1}.
\end{pf}

Hence $K$ and $J$ have the same joint distribution as $\Cro_t$ and
$\Nes_t$.
This nonintersecting process interpretation of $\Cro_t$ and $\Nes_t$
provides some useful information.
As an example, note that the process considered\vadjust{\goodbreak} above has a natural
dual process; see Figure~\ref{figpaths}.
In the dual process the roles of $K$ and $J$ are reversed: the depth is
$K$ and height is $J$ in the dual process.
It follows that $K$ and $J$, and hence $\Cro_t$ and $\Nes_t$, are
symmetrically distributed.

%
%f3 #&#
\begin{figure}

\includegraphics{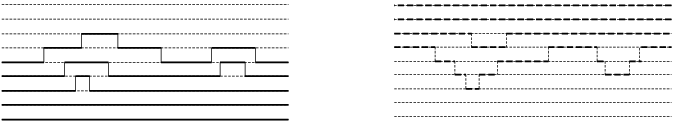}

\caption{A nonintersecting continuous-time simple random walks (left)
and its dual walk (right).}
\label{figpaths}
\end{figure}

In various nonintersecting processes, including the above model, the
top curve is shown to converge, after appropriate scaling, to the Airy
process in the long-time, many-walker limit; see, for example, \cite
{Johansson02,Johansson05}.
%This is the exact the same set-up for the above model, and it is
%showFor the above model, this is indeed the case, and hence
%When $t\to\infty$, the top curve $X_0(\tau)$ converges, after
%appropriate scaling, to the Airy process. ,
Then it is natural to think that the leading fluctuation term of $K$ is
given by the maximum of the Airy process.
It is a well-known fact that the maximum of the Airy process is
distributed as the GOE Tracy--Widom distribution.
This was first proved indirectly in \cite{Johansson03}. A direct proof
was only recently obtained in \cite{CorwinQR}.
(See also \cite{MFQuastelR} for the distribution of the location of
the maxima.)
Hence the leading term $F(x)$ in~\eqref{eq12} is as expected.
Moreover, when $t$ becomes large, it is plausible to expect that the
fluctuation of the top curve of the original process (whose max is $K$)
and the fluctuation of
the bottom curve of the dual process (whose min is $-J$) become
independent at least to the leading order. The leading term of
Theorem~\ref{thm3} is natural from this.
Theorem~\ref{thm3} evaluates the second term of the asymptotic
expansion of their joint distribution.

For a family of finitely many nonintersecting walks,
it is interesting to consider the maximum of the top curve and the
minimum of the bottom curve.
It is curious to check if the joint distribution of them would have the
same expansion as in Theorem~\ref{thm3}.
This will be considered elsewhere.
Finally, we mention that the asymptotics of the distribution of the
width of nonintersecting processes was studied recently in \cite{BaikLiu}.

\section{\texorpdfstring{Proof of Proposition~\protect\ref{prop1}}{Proof of Proposition 1.1}}\label{secproofprop1}

In this section, we give a proof of Proposition~\ref{prop1}.
We also obtain similar formulas for the marginal distributions of $\Cro_t$ and $\Nes_t$, and for the distribution of $L_t$. They are stated at
the end of this section.

Let $d\rho$ be a (either continuous or discrete) measure on the unit
circle %with infinite support
and define a new measure $d\rho(z;t)$ which depends on a parameter $t$ as
%
%e40 #&#
\begin{equation}
\label{eq33-1} d\rho(z;t) := e^{t(z+z^{-1})} \,d \rho(z).
\end{equation}
Measure~\eqref{eqD-1}, associated to the joint distribution of $\Cro_t, \Nes_t$, is certainly of this form, but the following algebraic
steps apply to general $d\rho$.\vadjust{\goodbreak}

Let
%
%e41 #&#
\begin{equation}
\label{eq33} h_\ell(t) := \oint_{|z| = 1} z^{-\ell} \,d
\rho(z;t).
\end{equation}
We are interested in finding a simple formula for the second derivative
of the Toeplitz determinant $T_n(t)$ and the Toeplitz--Hankel
determinant $H_n(t)$ [see~\eqref{eq15}] associated to the measure
$d\rho(z;t)$,
%
%e42 #&#
\begin{equation}
\label{eq33-2} \qquad T_n(t):= \det \bigl[ h_{a-b}(t)
\bigr]_{a,b=1}^n,\qquad  H_n(t)=\det \bigl[
h_{a-b}(t)-h_{a+b} (t) \bigr]_{a,b=1}^n.
\end{equation}
We assume that when $d\rho$ is a discrete measure, $n$ is smaller than
the number of points in the support of $d\rho$.

Let $\pi_n(z;t)=z^n+ \cdots,$ $n=0, 1, 2,\ldots,$ be the monic
orthogonal polynomials defined by the conditions
%
%e43 #&#
\begin{equation}
\label{OPconditions} \bigl\langle\pi_n ,z^\ell\bigr
\rangle := \oint_{|z|=1} \pi_n(z;t) \overline{z^{\ell}}
\,d \rho(z;t) = 0,\qquad  0 \leq\ell< n.
\end{equation}
Set
%
%e44 #&#
\begin{equation}
\label{Nndef} N_n(t) := \langle \pi_n, \pi_n\rangle  =
\bigl\langle \pi_n, z^n\bigr\rangle. %\oint_{|z|=1} \overline{\pi_n(z)} \pi_n(z) \,d \mu_m(z) =
\end{equation}
Then it is well known that (see, e.g., Sections 2 and 3 of \cite{2} for
the second identity)
%
%e45 #&#
\begin{equation}
\label{GkjOPproduct} T_j(t)= \prod_{n=0}^{j-1}
N_n(t),\qquad H_j(t) = \prod_{n=1}^{j}
N_{2n}(t) \bigl(1- \pi_{2n}(0;t) \bigr)^{-1}.
\end{equation}

Define (see \cite{Szego}) %Recall that if $\pi_n(z) = z^n +
%a_{n-1}z^{n-1}+\dots+a_1(z)+\pi_n(0)$, then we define
%
%e46 #&#
\begin{equation}
\label{reverseOP} \pi^*_n(z;t) := z^n\overline{
\pi_n \bigl(z^{-1};t \bigr)} = 1 + \overline{a_{n-1}}
z + \cdots+ \overline{a_1} z^{n-1} + \overline{
\pi_n(0;t)} z^n.
\end{equation}
This polynomial satisfies the orthogonality properties
%
%e47 #&#
\begin{equation}
\bigl\anglel{\pi_n^*}, {z^k} \bigr\angler =
N_n \delta_{k,0},\qquad  k=0,1,\ldots,n.
\end{equation}
Recall the Szeg\"o recurrence relations \cite{Szego},
%
%e48 #&#
\begin{eqnarray}
\label{Szegorecurrence} %
\pi_{n+1}(z) &=& z
\pi_n(z) + \pi_{n+1}(0) \pi^*_n(z),
\nonumber
\\[-8pt]
\\[-8pt]
\nonumber
z \pi_n(z) &=& \frac{N_n}{N_{n+1}} \bigl( \pi_{n+1}(z) -
\pi_{n+1}(0) \pi_{n+1}^*(z) \bigr). %
\end{eqnarray}
The second relation, when we compare the coefficients of $z^{n+1}$,
gives rise to the relation
%
%e49 #&#
\begin{eqnarray}
\label{eqd2} \frac{N_{n+1}}{N_n} = 1- \bigl\vert\pi_{n+1}(0)
\bigr\vert^2.
\end{eqnarray}

We now derive differential equations for $\pi_n(0;t)$ and $N_n(t)$.
All the differentiations are with respect to $t$, and we use the
notation $f'$ for $\frac{d}{dt} f$.
By differentiating the formula $\anglel{\pi_n},  {z^k} \angler =
0, k=0,\ldots,n-1$,
we obtain, by noting $\frac{\mathrmm{d}^{ } }{\mathrmm{d} t^{ } }
e^{t(z+z^{-1})} =
(z+z^{-1})e^{t(z+z^{-1})}$, that
$\anglel{\pi_n'},  {z^k} \angler +\anglel{\pi_n},
{z^{k+1}+z^{k-1}} \angler=0$. Then by
using the orthogonality conditions, we find that
%
%e50 #&#
\begin{eqnarray}
\label{piinnerproducts} %
\bigl\anglel{\pi_n'},
{z^k} \bigr\angler &=& 0,\qquad k=1,\ldots,n-2,
\nonumber\\
\bigl\anglel{\pi_n'}, {1} \bigr\angler &=& -\bigl
\anglel{\pi_n}, {z^{-1}} \bigr\angler = -\anglel{z
\pi_n}, {1} \angler = \pi_{n+1}(0) N_n,
\\
\bigl\anglel{\pi_n'}, {z^{n-1}} \bigr\angler
&=& -\bigl\anglel{\pi_n}, {z^n} \bigr\angler =
-N_n,\nonumber  %
\end{eqnarray}
where the last equality in the second condition above follows from the
first recurrence in \eqref{Szegorecurrence}.
From these relations, we conclude that, for $n \ge1$,
%
%e51 #&#
\begin{equation}
\label{piz} %
\pi_n'(z;t)=
\frac{N_n(t)}{N_{n-1}(t)} \bigl(\pi_{n+1}(0;t) \pi_{n-1}^*(z;t)-
\pi_{n-1}(z;t) \bigr). %
\end{equation}
This can be checked by taking the difference and noting that the
difference is a polynomial of degree at most $n-1$ and is orthogonal to
$z^k$, $k=0,1, \ldots, n-1$.
%The case when $n=1$ is slightly different but can be checked directly.
Evaluating~\eqref{piz} at $z=0$, we obtain, using~\eqref{eqd2},
for $n\ge1$,
%
%e52 #&#
\begin{equation}
\label{pi0} %
\pi_n'(0;t) % &= (\pi_{n+1}(0;t)-\pi_{n-1}(0;t))\frac{N_n(t)}{N_{n-1}(t)} \\
=
\bigl(\pi_{n+1}(0;t)-\pi_{n-1}(0;t) \bigr) \bigl(1-\bigl|
\pi_n(0;t)\bigr|^2 \bigr). %
\end{equation}
%
%where we used the general relation~\eqref{eqd2}.
This equation is related to the Ablowitz--Ladik equations and the Schur
flows; see, for example, \cite{Nenciu,Golinskii}.

We also differentiate $N_n(t)= \anglel{\pi_n},  {\pi_n } \angler$
and obtain
%
%e53 #&#
\begin{eqnarray}
\label{eqd10} %
N_n' % &=\inprod{\pi_n'}{ \pi_n}+ \inprod{\pi_n}{\pi_n'}+ \inprod{\pi_n
%(z+z^{-1})}{\pi_n} \\
&= 2 \bigl\anglel{\pi_n'}, {
\pi_n} \bigr\angler+2 \anglel{z \pi_n}, {
\pi_n} \angler = \langle 2z\pi_n, \pi_n
\rangle. %
\end{eqnarray}
%
%The first inner product in the second line is $0$ since $\pi_n'$ is of
%degree $<n$.
Using the first recurrence of~\eqref{Szegorecurrence},
%
%e54 #&#
\begin{eqnarray}
\anglel{z\pi_n}, {\pi_n} \angler= \anglel{
\pi_{n+1}}, {\pi_n} \angler - \pi_{n+1}(0) \bigl
\anglel{\pi_n}, {\pi_n^*} \bigr\angler = -
\pi_{n+1}(0)\pi_n(0) \bigl\anglel{\pi_n},
{z^n} \bigr\angler.
\end{eqnarray}
Hence, we obtain, for $n\ge0$,
%
%e55 #&#
\begin{equation}
\label{N} %
N_n'(t) % &= 2<\pi_{n+1}, \pi_n>- 2\pi_{n+1}(0)<\pi_n, \pi_n^*>= -2 \pi_{n+1}(0)
= - 2\pi_{n+1}(0;t)\pi_n(0;t)N_n(t).
\end{equation}

We now evaluate the logarithmic derivatives of $T_j$ and $H_j$.
From~\eqref{GkjOPproduct} and~\eqref{N}, we find that
%
%e56 #&#
\begin{equation}
\label{eqD4Q} %
\bigl( \log T_{j}(t) \bigr)'
= \sum_{n=0}^{j-1} \frac{N_n(t)}{N_n(t)} = -2
\sum_{n=0}^{j-1} \pi_{n}(0;t)
\pi_{n+1}(0;t) . %
\end{equation}
We take one more derivative.
%By taking one more derivative,
% %
% \left( \log T_{j}(t) \right)''
% = \sum_{n=0}^{j-1} \big( \log N_n(t) \big)'
% = -2\sum_{n=0}^{j-1} (\pi_{n}(0;t) \pi_{n+1}(0;t))' .
% %
By using~\eqref{pi0}, for $n\ge1$,
%
%e57 #&#
\begin{equation}
\label{eqD4} %
 \bigl(\pi_{n}(0) \pi_{n+1}(0)
\bigr)' = P_{n+1}-P_n,
%
% &- \left[ n+1 \mapsto n \right].
%
\end{equation}
where $P_n:= |\pi_{n}(0)|^2 + \pi_{n-1}(0)\pi_{n+1}(0) (1-|\pi_{n}(0)|^2 )$.
For $n=0$, $(\pi_{0}(0) \pi_{1}(0))' = \pi_1'(0) = (\pi_2(0)-1)(1-|\pi_1(0)|^2)=P_1-1$.
%%
% (\pi_{0}(0) \pi_{1}(0))' = \pi_1'(0)
% = (\pi_2(0)-1)(1-|\pi_1(0)|^2).
%%
%which is equal to~\eqref{eqD4} when $n=0$ (if we set $\pi_{-1}(0)$
%to be a finite number since $\pi_0(0)=1$).
Hence from a telescoping sum, we obtain
%
%e58 #&#
\begin{eqnarray}
\label{eqD50} &&\tfrac12 \bigl( \log \bigl(e^{-t^2}T_{j}(t)
\bigr) \bigr)'' %= \sum_{n=0}^{j-1} (\pi_{n}(0) \pi_{n+1}(0))' \\
\nonumber
\\[-8pt]
\\[-8pt]
\nonumber
&&\qquad= - \bigl(
\pi_{j-1}(0)\pi_{j+1}(0) +\bigl|\pi_{j}(0)\bigl|^2
\bigr) + \pi_{j-1}(0)\pi_{j+1}(0)\bigl|\pi_{j}(0)\bigr|^2
.
\end{eqnarray}

We now consider $H_j(t)$ in~\eqref{GkjOPproduct}.
By taking the log derivative and using \eqref{pi0}, \eqref{N} and
$\pi_0(z)=1$,
%
%e59 #&#
\begin{equation}
\qquad
\bigl( \log H_j(t) \bigr)' = \sum
_{n=1}^j \biggl[ \frac{N_{2n}'}{N_{2n}} +
\frac{\pi_{2n}'(0)}{1-\pi_{2n}(0)} \biggr] % &= \sum_{n=1}^j
% &= \pi_{2j+1}(0) - \pi_1(0) - \sum_{n=1}^j \pi_{2n}(0)\left(
= \pi_{2j+1}(0) - \sum_{n=0}^{2j}
\pi_{n}(0) \pi_{n+1}(0). %
\end{equation}
From~\eqref{eqD4Q}, we find that
%
%e60 #&#
\begin{equation}
\label{eqHj} \bigl( \log H_j(t) \bigr)' =
\pi_{2j+1}(0) + \tfrac{1}{2} \bigl( \log T_{2j+1}(t)
\bigr)' . % \nonumber\intertext{where} T_{2j}(t) &:=
\end{equation}

Proposition~\ref{prop1} is proven from~\eqref{eq8},~\eqref
{eq15},~\eqref{eqD50} and~\eqref{eqHj}
by noting that $\pi_n(0;0)=0$ for all $n\ge1$, and $T_{j}(0) = 1$ and
$H_j(0)=1$ for all $j\ge1$.

%In the remaining of this section, we discuss the analogue formula for $
The marginal distribution of $\Nes_t$ is obtained from~\eqref{eq8} by
taking the limit \mbox{$k\to\infty$}. Then by taking $m\to\infty$
in~\eqref{eq15},
we find that $\Prob\{\Nes_t\le j\}= e^{-t^2/2} G_{\infty, j}$ where
$G_{\infty, j}(t)$ is same as~\eqref{eq15} where the measure $\mu_m$
in~\eqref{eqD-1} is replaced by %$d \mu_{\infty}= e^{t(z+z^{-1})}
%
%e61 #&#
\begin{equation}
\label{eqC-1} d \mu_\infty(z) :=e^{t(z+z^{-1})} \frac{d z}{2\pi i z}.
\end{equation}
Then the above computation applies that %, and we obtain
%
%e62 #&#
\begin{equation}
\label{eqA56} %
\log\Prob\{ \Nes_t \leq j \} = \int
_0^t \pi_{2j+1, \infty}(0; \tau) \,d \tau+
\int_0^t \int_0^s
\QQ_j^\infty(\tau) \,d\tau \,ds, %\\
% &\log\Prob\left\{ \Cro_t \leq k \right\}
% = \int_0^t \pi^\infty_{2k+1}(0)   \,d \tau
% + \int_0^t \int_0^s \frac{1}{2} \left( \log(e^{-\tau^2} T_{2k}^\infty(
%
\end{equation}
where $\pi_{n,\infty}(z;t))$ is the monic orthogonal polynomial of
degree $n$ with respect to the measure~\eqref{eqC-1},
and $\QQ_j^\infty(\tau)$ is same as~\eqref{OPconditions} with $\pi_{n,m}(z;\tau)$ replaced by $\pi_{n,\infty}(z;\tau)$.
Due to the symmetry, $\Prob\{\Cro_t\le j\}= \Prob\{\Nes_t\le j\}$.

Finally, it is well known \cite{Gessel,Rains} that for the length
$L_t$ of the Poissonized random permutation defined in Section~\ref{secellt},
$\Prob\{L_t\le\ell\} = e^{-t^2} T_\ell(t)$, where $T_j(t)$ is the
determinant of the $\ell\times\ell$ Toeplitz matrix~\eqref{eq33-2}
with respect to measure~\eqref{eqC-1}. Hence we have
%
%e63 #&#
\begin{equation}
\label{eqellt-1} %
\log\Prob\{ L_t \leq\ell\} = 2 \int
_0^t \int_0^s
\QQ_{{(\ell-1)}/{2}}^\infty(\tau) \,d\tau \,ds. %
\end{equation}

%%%%%%%%%%%%%%%%%%%%%%%%%%%%%%%%%%%%%%%%%%%%%%%%%
%%%% RHP %%%%
%%%%%%%%%%%%%%%%%%%%%%%%%%%%%%%%%%%%%%%%%%%%%%%%%

%s4 #&#
\section{Orthogonal polynomial Riemann--Hilbert problems}\label{secRHP-start}

We prove Theorems \ref{thm3} and \ref{thm4} by deriving asymptotic
expansions of
$\pi_{n,m}(0; \tau)$ and $\pi_{n,\infty}(0;\tau)$, $n=2j, 2j+1, 2j+2$,
$\tau\in(0,t)$, in the joint limit $t,j,m \to\infty$ such that given
any fixed $x, x' \in\R$,
%
%e64 #&#
\begin{eqnarray}
\label{jkmcriticalscalings} j = t + \frac{x}{2} t^{1/3} ,\qquad k = t
+ \frac{x'}{2} t^{1/3},\qquad  m= j + k +1 .
\end{eqnarray}
The jumping off point for our analysis is the fact that $\pi_{n,m}(z;t)$ and $\pi_{n,\infty}(z;t)$ can be recovered from the
solution of the following discrete and continuous measure
Riemann--Hilbert problems, respectively.

% start discrete RHP
%
%ri4.1 #&#
\begin{rhp}[\textsc{for discrete OPs}]\label{rhpdiscrete}
Find a $2\times2$ matrix $\bY(z;t,n,m)$ with the following properties:
\begin{longlist}[(1)]
\item[(1)] $\bY(z;t,n,m)$ is an analytic function of $z$ for $z \in
\C
\setminus\{ \omega_r \}_{r=0}^{2m-1}$ where $\omega_r := \omega^r$
and $\omega:=e^{i\pi/m}$.
\item[(2)] $\bY(z;t,n,m) = [ I + \mathcal{O}  ( 1/z  ) ]
z^{n \sigma_3}$ as $z
\to
\infty$.
\item[(3)] At each $\omega_r$, $\bY(z;t,n,m)$ has a simple pole
satisfying the residue relation
%
%e65 #&#
\begin{eqnarray}
\label{Yresidues} \res_{z = \omega_r} \bY(z;t,n,m) = \lim_{z \to
\omega_r}
\bY(z;t,n) \triu[0]{\displaystyle -\frac{z}{2m} z^{-n}e^{t(z+z^{-1})} }.
\end{eqnarray}
\end{longlist}
\end{rhp}
%
% end discrete RHP
As is well known (see, e.g., \cite{FIK}, \cite{BKMM}),
and may be verified directly, the solution $\bfY(z;t;n,m)$ is given by
%
%e66 #&#
\begin{eqnarray}
\label{RHPsolution} \bY(z;t;n,m) = \tbyt{ \pi_{n,m}(z;t)} {*} {-
\pi^*_{n-1,m}(z;t)/N_{n-1,m}} {*},
\end{eqnarray}
where we recall that $\pi^*_{n,m}$ is the reverse polynomial defined by
\eqref{reverseOP} and
\begin{eqnarray*}
\bY_{12}(z;t,n,m) &=& -\frac{1}{2m} \sum
_{r=0}^{2m-1} \frac{ \pi_{n,m}(\omega_r;t) \omega_r^{-n+1}
e^{t(\omega
^r + \omega^{-r})} }{z- \omega_r},
\\
\bY_{22}(z;t,n,m) &=& \frac{1}{2m} \sum
_{r=0}^{2m-1} \frac{ N_{n-1,m}^{-1} \pi_{n-1,m}^*(\omega_r; t)
\omega_r^{-n+1} e^{t(\omega^r + \omega^{-r})} }{z- \omega_r}.
\end{eqnarray*}
Hence, using the OP properties listed in \eqref{OPconditions}--\eqref
{Szegorecurrence} we can easily check that
%
%e67 #&#
\begin{equation}
\label{pinrecovery} \bY(0;t,n,m) = \tbyt{ \pi_{n,m}(0) } {
N_{n,m} } { -1/N_{n-1,m} } {\pi_{n,m}(0)}.
\end{equation}
Note that the generic $(2,2)$-entry would be $\overline{ \pi_{n,
m}(0)}$ but as our weight $e^{t(z+z^{-1})}$ is real $\overline{ \pi_{n
,m}(0)} = \pi_{n ,m}(0)$.

% start cont RHP
The continuous RHP can be thought of as a limit of the discrete case
when $m$, the number of points in the support of the measure, goes to infinity.
%
%ri4.2 #&#
\begin{rhp}[\textsc{for continuous OPs}]\label{rhpcontinuous}
Find a $2 \times2$ matrix $\bY^\infty(z;t,n)$ with the following properties:
\begin{longlist}[(1)]
\item[(1)] $\bY^\infty(z;t,n)$ is an analytic function of $z$ for $z
\in\C\setminus\Sigma$, $\Sigma:=\{z \dvtx |z| = 1\}$ oriented
counterclockwise.
\item[(2)] $\bY^\infty(z;t,n) = [ I +\mathcal{O}  ( 1/z
) ] z^{n\sig}$ as $z
\to
\infty$.
\item[(3)] $\bY^\infty$ takes continuous boundary values $\bY_+^\infty$
and $\bY_-^\infty$ as $z \to\Sigma$ from the left/right, respectively,
satisfying the relation
%
%e68 #&#
\begin{equation}
\label{continuousjump} \bY^\infty_+(z;t,n) = \bY^\infty_-(z;t,n)
\triu{z^{-n} e^{t(z+z^{-1})} },\qquad z\in\Sigma.
\end{equation}
\end{longlist}
\end{rhp}
%
% end Cont RHP

The solution $\bY^\infty$ is related to the orthogonal polynomials
$\pi_{n,\infty}$ with respect to the measure $\mu_\infty$~\eqref{eqC-1},
and we have
%Just as in the discrete case, the solution of this problem takes the
%form of \eqref{RHPsolution} with $\pi_n^\infty$ replacing $\pi_n$
%and, in particular
%
%e69 #&#
\begin{equation}
\label{continuouspinrecovery} \bY^\infty(0;t,n,m) = \tbyt{
\pi_{n,\infty}(0) } { N_{n, \infty} } { -1/N_{n-1,\infty} } {
\pi_{n,\infty}(0)}.
\end{equation}

Precisely, this continuous Riemann--Hilbert problem was analyzed
asymptotically in \cite{BDJ,2,BKMM}. The steepest-descent analysis
for discrete Riemann--Hilbert problem was studied for general discrete
measure on the real line in \cite{BKMM}. Both works expand upon the
continuous weight case studied in \cite{DKMVZa,DKMVZb}.
In the course of proving Theorems~\ref{thm3} and~\ref{thm4} we
improve these results as follows: we expand the analysis of \cite{BKMM}
to the case when a gap and saturated region of the equilibrium measure
(see the discussion below) are about to open up, and we compute
explicit formulas for the first three terms in the expansion of the
solution in both the discrete and continuous cases extending the
results of \cite{BDJ,2,BKMM} where only leading terms were
calculated.\looseness=1

%For our case, we need to improve the analysis of \cite{BKMM} to the
%case when a gap and a saturated region of the equilibrium measure are
%about to open up.
%We also need the explicit formula of the first three terms in the
%asymptotic expansion of $\pi_{n,m}(0)$.

%for the regime that we are interested in, but we need further lower
%terms in the asymptotic expansion.
%The discrete Riemann--Hilbert problem for general discrete measure on
%the line was analyzed in \cite{BKMM} but for our case we need the
%situation when a gap and a saturated region are about to open up, and
%up to the three terms in the expansion.

%two previous works: [BDJ] and [BKMM]. A new improvements are ...}

%We analyze both Riemann--Hilbert problems simultaneously.
%The analysis of both Riemann--Hilbert problems have strong
%similarities.
%The second Riemann--Hilbert problem was analyzed asymptotically in
%for the regime that we are interested in, but we need further lower
%terms in the asymptotic expansion.

%We study the asymptotic behavior of these RHP's using the steepest
%descent method developed first by Deift and Zhou and extended by many
%others References... .
One of the key steps in the steepest-descent analysis of
Riemann--Hilbert problems %for orthogonal polynomials %, one of the key
%steps of in this construction
is the introduction of the so-called $g$-function.
For the Riemann--Hilbert problem~\ref{rhpdiscrete} for discrete
orthogonal polynomials, the $g$-function is
given by $g(z)=\int_{|s|=1} \log(z-s) \,d\mu(s)$ where $d\mu(s)$ is the
so-called equilibrium measure
satisfying $0\le d\mu(s) \le\frac{2m}{n} \frac{ds}{2\pi is}$; see, for
example, \cite{BKMM}. The upper-constraint $d\mu(s) \le\frac{2m}{n}
\frac{ds}{2\pi is}$
is due to the fact that the weight is discrete.
The support of $d\mu$ consists of three types of intervals, voids
(where $d\mu=0$), bands [where $0< d\mu(s) < \frac{2m}{n} \frac
{ds}{2\pi is}$] and saturations [where $ d\mu(s) = \frac{2m}{n} \frac
{ds}{2\pi is}$].

For the continuous Riemann--Hilbert problem, the upper-constraint for
the equilibrium is not present, and there are no saturations.
For the Riemann--Hilbert Problem~\ref{rhpcontinuous},
it was shown in \cite{BDJ} that with $\gamma= \frac{n}{2t}$,\footnote
{This is actually the inverse of the parameter appearing in \cite{BDJ}
which we find more convenient to work with presently.}
the support of the equilibrium measure consists of the entire unit circle when
$\gamma> 1$, and consists of single void and band intervals, with the void set
centered about $z =-1$, when $\gamma< 1$.
%the equilibrium
%measure is supported on the entire unit circle; that is, $\Sigma$ is a
%single band when $\gamma> 1$, and $\Sigma$ consist of single void and
%band intervals with the void set centered about $z=-1$ when $\gamma<
%1$.

In the discrete Problem~\ref{rhpdiscrete} the solution $\bY$ now
depends on the three parameters $(t,n,m)$ and as we shall see in
Section \ref{secRHPanalysis}, the equilibrium measure's support
depends critically on the two parameters
%
%e70 #&#
\begin{equation}
\label{gamma} \gamma= \frac{n}{2t} \quad\mbox{and} \quad\tg= \frac{2m-n}{2t}.
\end{equation}
As each of these parameters passes through the critical value $\gamma_{\mathrm{crit}} = 1$ a
transition occurs in the support of the equilibrium measure.\eject

It turns out that to prove Theorems~\ref{thm3}--\ref{thm5}, we only
need to evaluate $\bY(0; t,n,m)$ in two regimes:
the ``exponentially small regime''
%
%e71 #&#
\begin{equation}
\label{eqregime1} n\ge2t(1+\delta),\qquad 2m-n\ge2t(1+\delta)
\end{equation}
for a fixed $\delta>0$, and the ``Painlev\'e regime''
%
%e72 #&#
\begin{equation}
\label{eqregime2} 2t-Lt^{1/3}\le n\le2t(1+\delta),\qquad 2t-L
t^{1/3}\le2m-n\le2t(1+\delta)
\end{equation}
for fixed $L>0$ and $\delta>0$. In the ``exponential'' case $\gamma,
\tg\geq1+\delta$ and the equilibrium measure is supported on the
whole of $\Sigma$, while in the ``Painlev\'e'' case $\gamma, \tg\in[1
- \frac{L}{2} t^{-2/3}, 1+\delta]$ and the equilibrium measure is in
the transitional region where a void and saturation region are
beginning to open at $z=-1$ and $z=1$, respectively.
As such we never need to consider cases in which either a void or
saturation have fully opened,
and we restrict our attention to the full band (and the transitional)
case only, focusing on obtaining the three lower-order terms of the
asymptotic expansion explicitly. In this case the $g$-function is
explicit, and the transformations of the Riemann--Hilbert problem will
be all stated explicitly without mentioning the $g$-function in the
subsequent sections.

There are many interesting related problems in which one needs an
asymptotic description of the $\pi_{n,m}$ for a whole range of degrees
$n$; one such example which we plan to study in the future is the
Ablowtiz--Ladik equations. There we will fully describe the structure
of the equilibrium measure in the full range of parameter space.

%We analyze both Riemann--Hilbert problems simultaneously.
%The analysis of both Riemann--Hilbert problems have strong
%similarities.

The analyses of the discrete and continuous Riemann--Hilbert problems
have strong similarities, and we analyze them simultaneously. The
important fact, which we clarify in Sections \ref{seccontinuousweight}--\ref{secR0},
is that in the discrete Riemann--Hilbert
problem we can partition the solution into terms that come from (two)
continuous Riemann--Hilbert problems which correspond to the marginal
distributions and the remaining ``joint'' terms which contribute only
to the joint distribution.

%%%%%%%%%%%%%%%%%%%%%%%%%%%%%%%%%%%%%%%
% %
% RHP analysis %
% %
%%%%%%%%%%%%%%%%%%%%%%%%%%%%%%%%%%%%%%%

%s5 #&#
\section{The exponentially small regime}\label{secRHPanalysis}

The first steps of the steepest-descent analysis are the same for both
the exponentially small regime and the Painlev\'e regime. We begin by
first considering parameters $(n,m,t)$ in the ``exponentially small
regime'' \eqref{eqregime1},
\[
n\ge2t(1+\delta), \qquad 2m-n\ge2t(1+\delta)
\]
for fixed $\delta>0$. We assume that $\delta< 1/2$; see the discussion
before~\eqref{CFUcoordinates}.

We begin our analysis of RHP~\ref{rhpdiscrete} by first introducing a
transformation $\bY\mapsto\bQ$ such that the new unknown $\bQ$ has
no poles.
%Observing that $z^{2m} - 1 = \prod_{\ell=0}^{2m-1} (z- \omega_\ell)$
%we can write the factor $\omega_r/(2m)$ appearing in \eqref{Y%residues} as $\omega_r/(2m)= \left[ \prod_{\ell= 0, \neq r}^{2m-1} (
%achieves unity precisely at the $2m$ roots of unity $\omega_r$. These
%facts motivate the following pole removing factorization.
Let $\Sigma$ denote the unit circle and let $\Sigma_{\mathrm{in}}$ and $\Sigma_{\mathrm{out}}$ denote positively oriented simple closed contours enclosing the
origin such that $\Sigma_{\mathrm{in}} \subset\{z \dvtx |z| < 1\}$ and $\Sigma_{\mathrm{out}} \subset\{ z \dvtx |z|>1 \}$; let $\Omega_+$ and $\Omega_-$ denote
the nonempty open sets enclosed between $\Sigma$ and $\Sigma_{\mathrm{in}}$ and
$\Sigma$ and $\Sigma_{\mathrm{out}}$, respectively; see Figure \ref{figQcontours}.
Define
%
%e73 #&#
\begin{eqnarray}
\label{Qdef} \bQ(z) := \cases{ \bY(z) \triu{\displaystyle \frac
{z^{2m}}{z^{2m} - 1} z^{-n}
e^{t(z+z^{-1})} }, &\quad $z \in\Omega_+$,
\cr
 \bY(z) \triu{\displaystyle \frac{1}{z^{2m} - 1}
z^{-n} e^{t(z+z^{-1}) } } , & \quad $z \in\Omega_- .$}
\end{eqnarray}
%
%
%f4 #&#
\begin{figure}

\includegraphics{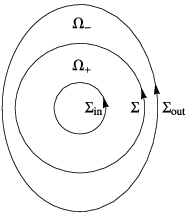}

\caption{The contours and regions used to define the map $\bY\mapsto
\bQ$. The contours $\Sigma_{\mathrm{in}}$ and
$\Sigma_{\mathrm{out}}$ can be deformed as necessary provided they do not intersect
$\Sigma$.}
\label{figQcontours}
\end{figure}

The triangular factors introduced in the above definition have poles at
each $\omega_r$, and the residues are such that the new unknown $\bQ
(z)$ has no poles, but is now piecewise holomorphic.
Note that the residue of each triangular factor at each $z=\omega_r$ is
the same since $z^{2m}=1$ at $z=\omega_r$.
Two different extensions of $\bQ$ as above were introduced in \cite
{KMM}; see also \cite{BKMM}.
By explicit computation $\bQ(z)$ satisfies
%
%ri5.1 #&#
\begin{rhp}[\textsc{for Q(z)}]\label{rhpQ}
Find a $2 \times2$ matrix $\bQ(z)$ such that:
\begin{longlist}[(1)]
\item[(1)] $\bQ(z)$ is analytic in $\C\setminus(\Sigma\cup\Sigma_{\mathrm{in}} \cup\Sigma_{\mathrm{out}})$.
\item[(2)] $\bQ(z) = [ I + \mathcal{O}  ( 1/z  ) ] z^{n
\sigma_3}$ as $z \to
\infty$.
\item[(3)] Along each jump contour $\bQ_+(z) = \bQ_-(z) V_Q(z)$ where
%
%e74 #&#
\begin{eqnarray}
\label{Qjumps} V_Q(z) = \cases{\triu{ \displaystyle z^{-n}
e^{t(z+z^{-1})} },  & \quad $z \in\Sigma$,\vspace*{3pt}
\cr
\triu{\displaystyle \frac{-z^{2m}}{z^{2m}-1} z^{-n}
e^{t(z+z^{-1})} } , & \quad $z \in\Sigma_{\mathrm{in}}$,\vspace*{3pt}
\cr
\triu{\displaystyle \frac{1}{z^{2m}-1}
z^{-n} e^{t(z+z^{-1})} }, & \quad $z \in \Sigma_{\mathrm{out}}.$}
\end{eqnarray}
\end{longlist}
\end{rhp}

Once we transforms a RHP with poles to a ``continuous'' RHP as $\bQ$,
the next step is to introduce a ``$g$-function.''
However, for the above RHP, when the parameters are in the
regimes~\eqref{eqregime1} and~\eqref{eqregime2}, it turns out that
the $g$-function is simple and explicit. We proceed by explicitly defining
%
%e75 #&#
\begin{eqnarray}
\label{Sdef} \bS(z) := \cases{ \bQ(z) \diag{e^{t z}} {e^{-tz}
} \offdiag{-1} {1}, &\quad  $|z| < 1$,\vspace*{3pt}
\cr
\bQ(z) \diag{z^{-n} e^{tz^{-1}}
} { z^n e^{-t z^{-1}} }, & \quad $|z| > 1$.}
\end{eqnarray}
%
%This is the first transformation which affects the values of the
%solution at $z=0$ and near infinity.
Clearly $\bY(0) = \bS(0) ({ 0 \enskip1 \atop-1 \enskip0 })$ and
%%because near infinity $\bS(z) = \bQ(z) z^{-n\sigma_3} [ I +
$\bS(z)=I+ \mathcal{O}  ( z^{-1}  )$ for large $z$.
Calculating the new jump matrices, we arrive at the following problem
for $\bS(z)$.
%
%ri5.2 #&#
\begin{rhp}[\textsc{for $S(z)$}]
Find a $2 \times2$ matrix-valued function $\bS(z)$ such that:
\begin{longlist}[(1)]
\item[(1)] $\bS(z)$ is analytic for $z\in\C\setminus(\Sigma\cup
\Sigma_{\mathrm{in}} \cup\Sigma_{\mathrm{out}})$.
\item[(2)] $\bS(z) = I + \mathcal{O}  ( 1/z  )$ as $z \to
\infty$.
\item[(3)] The boundary values of $\bS(z)$ satisfy the jump relation
$\bS_+(z) =\break \bS_-(z) V_S(z)$ where
%
%e76 #&#
\begin{eqnarray}
\label{Sjump} V_S(z) = \cases{ \displaystyle\tril{(-1)^n
e^{-2t\theta} } \triu{\displaystyle -(-1)^n e^{2t \theta} }, &\quad  $z \in
\Sigma$,\vspace*{3pt}
\cr
\tril{ \displaystyle\frac{-1}{1-z^{2m}} e^{-2t\phi} }, &\quad  $z \in
\Sigma_{\mathrm{in}}$,\vspace*{3pt}
\cr
\triu{\displaystyle\frac{1}{1-z^{-2m}} e^{2t\phi} }, & \quad $z \in
\Sigma_{\mathrm{out}}$,}
\end{eqnarray}
where
%
%e77 #&#
\begin{eqnarray}
\label{allbandphases} %
\theta(z; \gamma) &:=& \frac{1}{2}
\bigl(z-z^{-1} \bigr) + \gamma\log(-z), \qquad\gamma:= \frac{n}{2t},
\nonumber
\\[-8pt]
\\[-8pt]
\nonumber
\phi(z; \tilde\gamma)& :=& \frac{1}{2} \bigl(z-z^{-1} \bigr) -
\tilde\gamma\log z,\qquad  \tilde\gamma:=\frac{2m-n}{2t}. %
\end{eqnarray}
Here the log is defined on the principal branch. %$\gamma$ and $\tg$ as
%defined in \eqref{gamma}.
\end{longlist}
\end{rhp}

Now we assume that the parameters are in regime~\eqref{eqregime1}.
Note that for any $e^{i\alpha} \in\Sigma$, $\theta(e^{i\alpha})
\in i\R$.
Also note that writing $z = r e^{i \alpha}$, we have $\frac{\mathrmm{d}^{ } }{\mathrmm{d} r^{ } } [ \re
\theta(r e^{i \alpha};\break\gamma) ]_{r=1}= \cos\alpha+ \gamma\ge
-1+\gamma\ge\delta>0$ and $\frac{d^2}{dr^2} [ \re\theta(r e^{i
\alpha
};\gamma) ]_{r=1}= -r^3\cos\alpha-\gamma r^{-2}
\le r^{-3}-\gamma r^{-2}<0$ if $r>\gamma^{-1}$.
Hence
$\re\theta(r e^{i \alpha};\gamma)\le(-1+\gamma)(r-1)$ for $r\in
(\gamma^{-1},1)$ and for all $\alpha\in(-\pi, \pi]$.
Therefore, for a given $\delta>0$, there exist $0<r_1<r_2<1$ and $c>0$
such that
$\re[ \frac1{\gamma}\theta(r e^{i \alpha};\gamma) ] \le-c$ for all
$r\in[r_1, r_2]$, $\alpha\in(-\pi, \pi]$ and for
the parameters $(n,m,t)$ in the regime~\eqref{eqregime1}.
Note that this implies that
%
%e78 #&#
\begin{equation}
\label{eqcase1-1} \bigl|e^{2t\theta(z; \gamma)}\bigr| = e^{n \re[ (1/{\gamma)
}\theta(r e^{i \alpha};\gamma) ] } \le e^{-cn},\qquad
r_1\le|z|\le r_2
\end{equation}
for parameters $(n,m,t)$ in the regime~\eqref{eqregime1}.

Similarly, $\re[ \frac1{\tilde\gamma}\phi(\frac1{r} e^{i \alpha
};\tilde\gamma) ] \le-c$ for all $r\in[r_1, r_2]$, $\alpha\in
(-\pi,
\pi]$ and for
the parameters $(n,m,t)$ in the regime~\eqref{eqregime1}.
This can be easily seen by noting that $\phi(z;\gamma)= \theta(-z^{-1};
\gamma)$. Hence
%
%e79 #&#
\begin{eqnarray}
\label{eqcase1-2} \bigl|e^{2t\phi(z; \tilde\gamma)}\bigr| = e^{(2m-n) \re[
(1/{\gamma})\theta(r e^{i \alpha};\gamma) ] } \le e^{-c(2m-n)},\qquad
\frac1{r_1}\le|z|\le\frac1{r_2}
\end{eqnarray}
for parameters $(n,m,t)$ in the regime~\eqref{eqregime1}.

%
%f5 #&#
\begin{figure}

\includegraphics{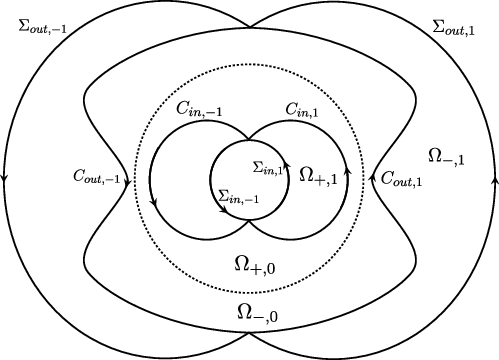}

\caption{Lens contours and regions in the definition of $\bT
(z)$.}\label
{figTcontours}
\end{figure}

%To $\Sigma_{\mathrm{in}}$ and $\Sigma_{\mathrm{out}}$ we add new contours denoted by

Let $C_{\mathrm{in},-1}, C_{\mathrm{in},1}, C_{\mathrm{out},1}$ and $C_{\mathrm{out}, -1}$ be the
contours as depicted in Figure~\ref{figTcontours} such that
$C_{\mathrm{in},-1}$ and $C_{\mathrm{in},1}$ lie in the annulus $r_1<|z|<r_2$ and
$C_{\mathrm{out},1}$ and $C_{\mathrm{out}, -1}$ lie in the annulus $\frac
1{r_1}<|z|<\frac1{r_2}$.
Make now the following change of variables which moves the oscillations
on $\Sigma$ into regions of exponential decay.
%
%e80 #&#
\begin{eqnarray}
\label{Tdef} \bT(z) = \cases{ \displaystyle\bS(z) \triu{(-1)^n e^{2t
\theta}},
&\quad $z \in\Omega_{+,0}$,\vspace*{3pt}
\cr
\bS(z) \triu{(-1)^n
e^{2t \theta}} \tril{ - e^{-2t \phi} }, &\quad $ z \in\Omega_{+,1}$,\vspace*{3pt}
\cr
\bS(z) \tril{(-1)^n e^{-2t\theta} }, & \quad $z \in
\Omega_{-,0}$,\vspace*{3pt}
\cr
\bS(z) \tril{ (-1)^n e^{-2t\theta} }
\triu{-e^{2t \phi} }, &\quad  $z \in\Omega_{-,1}$,
\cr
\bS(z), &\quad  $\mbox{elsewhere.}$}
\end{eqnarray}
Note that $\bY(0)=\bT(0)({ 0 \enskip1 \atop-1 \enskip0 })$.
Explicitly calculating the new jumps, the new unknown $\bT(z)$
satisfies the following problem:
%
%ri5.3 #&#
\begin{rhp}[\textsc{for T(z)}]
Find a $2 \times2$ matrix-valued function $\bT(z)$ satisfying the
following properties:
\begin{enumerate}[(1)]
\item[(1)] $\bT(z)$ is analytic in $\C\setminus(\Sigma_{\mathrm{in}} \cup
\Sigma_{\mathrm{out}} \cup C_{\mathrm{in},\pm1} \cup C_{\mathrm{out},\pm1})$.
\item[(2)] $\bT(z) = I + \mathcal{O}  ( 1/z  )$ as $z \to
\infty$.
\item[(3)] The boundary values of $\bT(z)$ satisfy the jump relation
$\bT_+(z) =\break \bT_-(z) V_T(z)$ where
%
%e81 #&#
\begin{eqnarray}
\label{Tjumps} %
V_T(z)
& =&\cases{\displaystyle \triu{-(-1)^n e^{2t\theta} } ,& \quad $z \in
C_{\mathrm{in},-1}$,\vspace*{3pt}
\cr
\displaystyle\tril{ (-1)^n e^{-2t\theta} }, &\quad $ z \in
C_{\mathrm{out},-1} $,\vspace*{3pt}
\cr
\tril{\displaystyle-e^{-2t\phi} }, &\quad $z \in C_{\mathrm{in},1}$,\vspace*{3pt}
\cr
\triu{\displaystyle e^{2t \phi} }, & \quad $z \in
C_{\mathrm{out},1}$,\vspace*{3pt}\cr
%}
%&&{}  \cases{
\tril{\displaystyle
\frac{-e^{-2t \phi} }{1-z^{2m}}}, & \quad$z
\in\Sigma_{\mathrm{in},-1}$,\vspace*{3pt}
\cr
\triu{\displaystyle\frac{- (-1)^n e^{-2t\theta} }{ 1-z^{2m} }}
\bigl(1-z^{2m} \bigr)^{-\sigma_3}, &\quad  $z \in\Sigma_{\mathrm{in},1}
$,\vspace*{3pt}
\cr
\displaystyle\bigl(1- z^{-2m} \bigr)^{-\sigma_3} \tril{\displaystyle \frac{(-1)^n e^{-2t
\theta}
}{1-z^{-2m} }} ,&\quad $
z \in\Sigma_{\mathrm{out},1} $,\vspace*{3pt}
\cr
\triu{\displaystyle\frac{1}{1-z^{-2m}} e^{2t\phi} }, &\quad $
z \in\Sigma_{\mathrm{out},-1}.$ } %
\end{eqnarray}
\end{enumerate}
\end{rhp}

Then from~\eqref{eqcase1-1} and~\eqref{eqcase1-2}, we find that
$V_T(z)=I+\mathcal{O}  ( e^{-c\max\{n, 2m-n\}}  )$
uniformly for $z$ on the contour.
Hence we obtain the following result.

%
%pr5.1 #&#
\begin{prop}\label{propexpsmall}
Let $\bY(z;t,n,m)$ be the solution to the RHP~\eqref{rhpdiscrete}.
For any $\delta>0$, there exists a constant $c>0$ such that, if
%
%e82 #&#
\begin{equation}
\label{eqregime1-1} n\ge2t(1+\delta),\qquad 2m-n\ge2t(1+\delta),
\end{equation}
then
%
%e83 #&#
\begin{equation}
\label{eqregime1-2} \bY(0;t,n,m) \pmatrix{ 0 & -1
\cr
1 & 0 } = I+ \mathcal{O}
\bigl( e^{-c\max\{n, 2m-n\}} \bigr).
\end{equation}
In particular,
%
%e84 #&#
\begin{equation}
\label{eqregime1-3} \pi_{n,m}(0;t)=\mathcal{O} \bigl( e^{-c\max\{
n, 2m-n\}}
\bigr).
\end{equation}
\end{prop}

%s6 #&#
\section{Painlev\'e regime}\label{secRHPpainleveregime}

$\!\!$We now consider the parameters $(n,m,t)$ in regime~\eqref{eqregime2},
%
%e85 #&#
\begin{equation}
\label{eqregime2-1} 2t-Lt^{1/3}\le n\le2t(1+\delta), \qquad 2t-L
t^{1/3}\le2m-n\le2t(1+\delta)
\end{equation}
for fixed $L>0$ and $\delta>0$.
We assume that $\delta<1/2$; see the discussion before~\eqref{CFUcoordinates}.

%
%f6 #&#
\begin{figure}[b]

\includegraphics{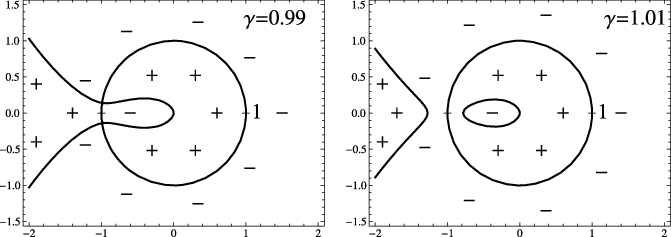}

\caption{The sign of $\re\theta(z;\gamma)$ for values of $\gamma$ near
$\gamma_{\mathrm{crit}} = 1$.
Note the sign change near $z=-1$ on either side of the transition.
}\label{figsigns}
\end{figure}

Let $\bS(z)$ be same as in the previous section.
When $\gamma\in[1-\delta, 1+\delta]$, estimate~\eqref{eqcase1-1}
does not hold any more.
However, it is easy to check using a similar calculation as before that
the exponential decay still holds in an annular sector away from the
point $z=-1$. [Note that the (double) critical point of $\theta(z;1)$
is $z=-1$.]
More precisely, one can check that given $\delta\in(0, 1)$, there
exist positive constants
%$\alpha_1>\cos^{-1}(1-\delta)=O(\delta^{1/2})$
$\alpha_1\ge O(\delta^{1/2})>0$
and $\rho_1 \ge O(\delta^{1/2})$ such that
if $\gamma\in[1-\delta, 1+\delta]$, then
$|e^{2\theta(z;\gamma)}|<1$ for $z$ in the annular sector
$\sector_{\mathrm{in},-1}:=\{ z=re^{i\alpha}  \dvtx \rho_1 < r < 1,   |\alpha
| <
\pi-\alpha_1\}$.
Moreover, if $z$ is in a compact subset of $\sector_{\mathrm{in},-1}$, then
there exists $c>0$ such that
$|e^{2\theta(z;\gamma)}|\le e^{-c}$ uniformly in $\gamma\in
[1-\delta,
1+\delta]$; see Figure~\ref{figsigns}.

Similarly, from the symmetry $\phi(z;\gamma)=\theta(-z^{-1}; \gamma)$,
under the same assumptions,
$|e^{2\phi(z;\tilde\gamma)}|<1$ for $z$ in the annular sector
$\sector_{\mathrm{out},1}:=\{ z=\frac1{r}e^{i\alpha}  \dvtx \rho_1 < r < 1,   \alpha_1\le|\alpha| \le\pi\}$.
Note the change of the condition on the angle from $\sector_{\mathrm{in},-1}$;
the (double) critical point of $\phi(z;1)$ is $z=1$.
As before, if $z$ is in a compact subset of $\sector_{\mathrm{out},1}$, then
there exists $c>0$ such that
$|e^{2\phi(z;\gamma)}|\le e^{-c}$ uniformly in $\gamma\in[1-\delta,
1+\delta]$.

Now define $\bT$ by~\eqref{Tdef} as before. In doing so,
we take $C_{\mathrm{in},1}$ and $C_{\mathrm{in}, -1}$ to lie in the annulus $\rho_1<|z|<1$, and take
$C_{\mathrm{out}, 1}$ and $C_{\mathrm{out}, -1}$ to lie in the annulus $1<|z|<1/\rho_1$.
Then the jump matrix in~\eqref{Tjumps} satisfies
%
%e86 #&#
\begin{equation}
\label{eqVTexp-1} V_T(z)= I + \mathcal{O} \bigl( e^{-ct}
\bigr)
\end{equation}
uniformly for $\gamma, \tilde\gamma\in[1-\delta, 1+\delta]$
and for $z$ in all the contours except for
$(C_{\mathrm{in}, -1}\cup C_{\mathrm{out}, -1})\cap\{|\operatorname{arg}(z)|>\pi-\alpha_1\}$ and
$(C_{\mathrm{in}, 1}\cup C_{\mathrm{out}, 1})\cap\{|\operatorname{arg}(z)|<\alpha_1\}$.
The parts of the contour where~\eqref{eqVTexp-1} is not valid are
handled by introducing local parametrix that can be solved by the RHP
for the Painlev\'e II equation; see Section \ref{secpainleve}.
Such a ``Painlev\'e parametrix'' was introduced in the analysis of
\cite
{BDJ} on a similar orthogonal polynomials but with a continuous weight.
A drawback of the analysis of \cite{BDJ} was that the parametrix was
solved asymptotically rather than exactly as in other cases such as
\cite{DKMVZa,DKMVZb}.
The exactly matching Painlev\'e parametrix was constructed later in
\cite{CK}. The construction of \cite{CK} requires, in the context of
this paper, that $\gamma\in[1-Lt^{-2/3}, 1+Lt^{-2/3}]$. In a recent
paper \cite{BMiller}, a different approach to the exact construction of
the Painlev\'e parametrix was introduced.
This construction has the advantage that it works for all $\gamma$ (and
$\tilde\gamma$) in regime~\eqref{eqregime2}.

%The result of the transformation to $\bT(z)$ is an unknown whose jumps
%are now all well controlled. For any fixed $z$, the jump $V_T(z) \to
%I$ as $m \to\infty$, If $\gamma, \tilde\gamma> 1+\varepsilon$ then the
%convergence is uniform, and the solution $\bT(z)$ is uniformly well
%approximated by identity to all orders. However, for $\gamma, \tilde
%$C_{\mathrm{in}/out, 1}$ and $C_{\mathrm{in}/out, -1}$ collapse onto the unit circle at $
%the lack of uniformity near these points.

We seek a global parametrix in the form
%
%e87 #&#
\begin{equation}
\label{parametrix} \bA(z) = \cases{ \bA_1(z), &\quad $z \in
\U_1$,
\cr
\bA_{-1}(z), &\quad $z \in\U_{-1}$,
\cr
I, & \quad $\mbox{elsewhere},$ }
\end{equation}
where $\U_{\pm1}$ are sufficiently small, fixed size, neighborhoods of
$\pm1$.
Later we will fix the size of $\U_{\pm1}$ first and then choose
$\delta
$ small enough so that
$\U_{-1}$ contains
$(C_{\mathrm{in}, -1}\cup C_{\mathrm{out}, -1})\cap\{|\operatorname{arg}(z)|>\pi-\alpha_1\}$
and $\U_1$ contains
$(C_{\mathrm{in}, 1}\cup C_{\mathrm{out}, 1})\cap\{|\operatorname{arg}(z)|<\alpha_1\}$
so that~\eqref{eqVTexp-1}
is valid for all $z$ in the contour of $\bT$ except for in $\U_{\pm1}$.

%s6.1 #&#
\subsection{Local models near $1$ and $-1$}
In order to construct exactly matching parametrices $\bA_{\pm1}$, we
need to introduce Langer transformations which map the local phase
functions $\theta$ and $\phi$ to the Painlev\'e phase \eqref{PIIphase}
in $\U_{-1}$ and $\U_{1}$, respectively.

The phase $\theta(z; \gamma)$ is analytic in $z$ in the neighborhood
$|z+1|<1$ (and entire in $\gamma$) and admits the expansion
%
%e88 #&#
\begin{eqnarray}
\label{phaseexpansion} \theta(z;\gamma) &=& (1-\gamma) (z+1) +
\frac{1-\gamma}{2}(z+1)^2 + \frac{3-2\gamma}{6}(z+1)^3
\nonumber
\\[-8pt]
\\[-8pt]
\nonumber
&&{} +
\mathcal{O} \bigl( (z+1)^4 \bigr).
\end{eqnarray}
At the critical value $\gamma=1$ the expansion degenerates to a cubic
at leading order; for values of $\gamma$ near $1$ the cubic unfolds
either into three real or one real and two complex roots near $z = -1$.
The double critical point--double root of $\theta'(z;1)$--unfolds into
a pair of simple critical points near $z = -1$,
%
%e89 #&#
\begin{eqnarray}
\label{allbandcriticalpoints} \frac{\mathrmm{d}^{ } \theta
}{\mathrmm{d} z^{ } } = 0 \quad\Rightarrow\quad
z_\pm= -\gamma\pm\sqrt{\gamma^2 -1}.
\end{eqnarray}
Note that the relation $\phi(z; \tg) = - \theta(-z;\tg)$ implies that
$\phi$ admits a similar expansion about $z=1$ with the same structure.

As the cubic coefficient in \eqref{phaseexpansion} is bounded away
from zero
(note that $\gamma\le1+\delta< 3/2$) we make use of a classical result
of \cite{CFU} to introduce new parameters $a(\gamma)$ and $b(\gamma)$
such that the relation
%
%e90 #&#
\begin{equation}
\label{CFUcoordinates} \tfrac{4}{3} f(z;\gamma)^3 + a(\gamma)
f(z; \gamma) + b(\gamma) = - i \theta(z; \gamma) , \qquad z \in\U_{-1}
\end{equation}
%
% \cases{
% \frac{4}{3} f(z;\gamma)^3 + a(\gamma) f(z; \gamma) + b(\gamma) = - i
%f(z; \tg) + b(\tg) =- i \phi(z; \tg) & z \in\U_1
% }
defines an invertible conformal mapping $f = f(z)$ from a sufficiently
small, $\gamma$-independent, neighborhood $\U_{-1}$ onto $f(\U_1)$ such
that the parameters $a$ and $b$ depend continuously on $\gamma$ near $1$.
It was shown in \cite{CFU} (see also \cite{Friedman}) that there exist
$\delta_1>0$ and a $\gamma$-independent neighborhood $\U_{-1}$ such
that the above map is conformal
in $\U_{-1}$ for all $\gamma\in[1-\delta_1, 1+\delta_1]$ if
the critical points $f_{\pm}= \pm\sqrt{-a}/2$ of the left-hand side,
seen as a function of $f$,
% $f(z_\pm;\gamma) = \pm\sqrt{-a(\gamma)}/2$
correspond to the critical points %$z_\pm$
$z_{\pm}$ of $\theta(z;\gamma)$.
This means that the left-hand side of~\eqref{CFUcoordinates} evaluated
at $f=f_{\pm}$
should equal to the right-hand side of~\eqref{CFUcoordinates}
evaluated at $z=z_{\pm}$.
These two conditions determine parameters $a$ and $b$ as
%
%e91 #&#
\begin{eqnarray}
b(\gamma)& =& \frac{-i}{2} \bigl[ \theta(z_+; \gamma) +
\theta(z_-; \gamma) \bigr],
\nonumber
\\[-8pt]
\\[-8pt]
\nonumber
 \bigl(-a(\gamma) \bigr)^{3/2} &=&\frac{3i}{2} \bigl(\theta(z_+;
\gamma) - \theta(z_-; \gamma) \bigr). %
\end{eqnarray}
Since $\theta(z_+; \gamma)= -\theta(z_-; \gamma)=\sqrt{\gamma^2
-1} -
\gamma\log( \gamma+ \sqrt{\gamma^2 -1})$,
we have
%
%e92 #&#
\begin{equation}
b(\gamma) = 0.
\end{equation}
There are three choices of branch of $a(\gamma)$. We choose the branch
so that
%By taking an appropriate branch
%It follows from \eqref{allbandphases} and \eqref{allbandcritical%points} that
%
%e93 #&#
\begin{equation}
\label{CFUparameters} %
a(\gamma) = - \bigl[3i \bigl( \sqrt{
\gamma^2 -1} - \gamma\log \bigl(\gamma+ \sqrt{\gamma^2 -1}
\bigr) \bigr) \bigr]^{2/3} %, \\ b(\gamma) &:= 0,
\end{equation}
satisfies the power series expansion
% = (\gamma-1)^{-1/2}h(\gamma)$ for locally real analytic and nonzero
%$h$
%it follows that we can take (?)
%
%e94 #&#
\begin{equation}
\label{eqagam-1} a(\gamma) = 2(\gamma-1)-\tfrac1{15} (\gamma-1)^2+
\mathcal{O} \bigl( (\gamma-1)^3 \bigr).
\end{equation}
To verify this, it is useful to note that $\frac{d^2}{d\gamma^2}
[\sqrt {\gamma^2 -1} - \gamma\log( \gamma+ \sqrt{\gamma^2 -1}) ]=
-(\gamma^2-1)^{-1/2}$.
% for $H$ nonzero and real analytic near $\gamma=1$.
With this choice of $a$, we have
%The new coordinate $f$ clearly depends on both $z$ and $\gamma$ and
%using \eqref{CFUcoordinates} one may easily calculate the first few
%terms in the expansion of $f$ at $z= -1$:
%
%e95 #&#
\begin{eqnarray}
\label{CFUexpansion} f(z; \gamma) &=& \frac{i(\gamma- 1)}{a}(z+1) + \frac{i(\gamma-1)}{2a}
(z+1)^2
\nonumber
\\[-8pt]
\\[-8pt]
\nonumber
&&{}+ \frac{1}{6i} \biggl(\frac{3-2\gamma}{a} - 8 \frac{(\gamma-1)^3}{a^4} \biggr)
(z+1)^3 + \mathcal{O} \bigl( (z+1)^4 \bigr).
\end{eqnarray}
Inserting~\eqref{eqagam-1}, we obtain
%
%e96 #&#
\begin{eqnarray}
\label{CFUcriticalexpansion} f(z;\gamma)& =& \frac{i}{2}(z+1) \biggl[ 1 +
\frac{1}{2}(z+1) + \frac{7}{20} (z+1)^2 + \mathcal{O}
\bigl( (z+1)^3 \bigr) \biggr]
\nonumber\\
&&{}+\frac{i}{60}(\gamma-1) (z+1) \biggl[ 1 + \frac{1}{2}(z+1) +
\mathcal{O} \bigl( (z+1)^2 \bigr) \biggr] \\
&&{}+ \mathcal{O} \bigl( (
\gamma-1)^2(z+1) \bigr) .\nonumber
\end{eqnarray}

Define the rescaled coordinates (Langer coordinates) $\zeta=\zeta
(z;\gamma)=\break t^{1/3} f(z;\gamma)$ for $z\in\U_{-1}$, and set
%
%e97 #&#
\begin{equation}
\label{eqsdef} %
s=s(\gamma)= t^{2/3}a(\gamma). %
\end{equation}
Then [see~\eqref{PIIphase}]%The rescaled coordinates $\zeta= t^{1/3}
%f$ and $s = t^{2/3} a$ allow us to define our Langer coordinates;
%using \eqref{CFUcoordinates} define
%
%e98 #&#
\begin{eqnarray}
\label{langervariable} t\theta(z; \gamma) =i \bigl(\tfrac{4}{3}
\zeta^3 + s \zeta \bigr)= i\theta_{\mathit{PII}}(\zeta,s),\qquad  z \in
\U_{-1}.
\end{eqnarray}
%
% i\theta_{\mathit{PII}}(\zeta,s)=i \lp\frac{4}{3} \zeta^3 + s \zeta\rp=
% t\theta(z; \gamma) & z \in\U_{-1} \\
% t\phi(z; \tilde\gamma) & z \in\U_{1},
% }
We note from~\eqref{eqagam-1} that for the parameters $(n,m,t)$ in
regime~\eqref{eqregime2},
%
%e99 #&#
\begin{equation}
 s(\gamma)\ge-2L
\end{equation}
for all large enough $t$. We also have
%
%e100 #&#
\begin{equation}
\label{sexpansion} s(\gamma) = 2t^{2/3}(\gamma-1) - \frac
{(2t^{2/3}(\gamma-1) )^2}{60}
t^{-2/3} + \mathcal{O} \bigl( t^{2/3}(\gamma-1)^3
\bigr).
\end{equation}
%
%We note that plugging \eqref{gamma} into \eqref{CFUparameters} and
%expanding we have
% s(\gamma) = y - \frac{y^2}{60} t^{-2/3} + \bigo{ t^{-4/3} }.

We introduce similar coordinates in $\U_1$. This can be easily achieved
by noting the
symmetry $\phi(z,\tg)=-\theta(-z,\tg)$.
We set $\U_1= -\U_{-1}$ and define $f(z;\tg):= -f(-z; \tg)$ for
$z\in\U_1$.
Then we find, with the same choice of $a$ and~$b$,
%
%e101 #&#
\begin{equation}
\label{CFUcoordinates-2} \tfrac{4}{3} f(z;\tg)^3 + a(\tg) f(z;
\tg) =- i \phi(z; \tg),\qquad  z \in\U_1 .
\end{equation}
Defining $\zeta=\zeta(z;\gamma)=t^{1/3} f(z;\gamma)$, $z\in\U_{-1}$
and $s=s(\gamma)= t^{2/3}a(\gamma)$ as before, we obtain
%
%e102 #&#
\begin{equation}
\label{langervariable-2} t\phi(z; \tg) =i \bigl(\tfrac{4}{3} \zeta(z;
\tg)^3 + s(\tg) \zeta(z;\tg) \bigr)= i\theta_{\mathit{PII}}(
\zeta,s),\qquad  z \in\U_{-1}.
\end{equation}
Note the symmetry
%
%e103 #&#
\begin{equation}
\label{zetasymmetry} \zeta(z ; \tg) = -\zeta(-z; \tg) ,\qquad  z\in \U_1.
\end{equation}

We take $\delta$ such that $\delta<\min\{1/2, \delta_1\}$ where
$\delta_1$
we introduced in defining $f$ in~\eqref{CFUcoordinates}.
Then consider the parameters $(n,m,t)$ satisfying~\eqref{eqregime2}.

Consider the image of $\U_{-1}$ under the map $z\mapsto\zeta
(z;\gamma)$.
From~\eqref{CFUcriticalexpansion}, we find that there exists $\delta_2>0$ such that
for $\gamma\in[1-\delta_2, 1+\delta_2]$,
$\zeta(\U_{-1};\gamma)$ contains a disk centered at $0$ and of radius
$\ge\mathcal{O}  ( t^{1/3}  )$ in the $\zeta$-plane.
The same holds for $\zeta(\U_{1}; \tg)$.
Note that from \eqref{CFUcriticalexpansion},
the image contours $\zeta(C_{-1,\mathrm{in}/\mathrm{out}})$ are oriented left-to-right and
the image contours $\zeta(C_{1,\mathrm{in}/\mathrm{out}})$ are oriented right-to-left as
depicted in Figure \ref{figzetamap}.

%
%f7 #&#
\begin{figure}

\includegraphics{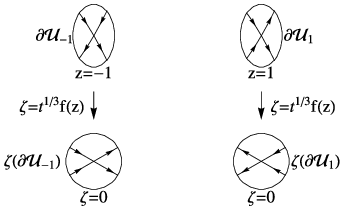}

\caption{Images of the contours near $z=\pm1$ under $\zeta$.}\label
{figzetamap}
\end{figure}

We now use $\zeta$ to map the local contours and jump matrices inside
$\U_{\pm1}$ onto the jumps of the Painlev\'e parametrix, RHP \ref
{rhpPII}.
We locally deform, if necessary, the contours $C_{\pm1, \mathrm{in}/\mathrm{out}}$
%in and our choice of $\U_{\pm1}$ we can arrange that
so that the image contours $\zeta(C_{\pm1, \mathrm{in}/\mathrm{out}}\cap\U_{\pm1})$
become %lie on
the rays $\Gamma_i$, $i=1,3,4,6$ described in \eqref{eq73}, and
we extend $C_{\pm1, \mathrm{in}/\mathrm{out}}\cap(\U_{-1}\cup\U_1)$ to the rest of
$C_{\pm1, \mathrm{in}/\mathrm{out}}$ so that estimate~\eqref{eqVTexp-1} holds for $z$
on the contour outside of $\U_{\pm1}$.
The exact shape of the contours are not important. Reorienting the
image contours, if necessary, to go from left-to-right and using \eqref
{langervariable} and \eqref{langervariable-2} the image contours and
jumps are, up to a conjugation by a constant matrix, exactly those of
the Painlev\'e parametrix, RHP \ref{rhpPII}.

% and that $\zeta(\U_{\pm1})$ is a disk in the $zeta$-plane centered
%at the origin with radius scaling as $t^{1/3}$.

Let $\bpsi(\zeta, s)$ be the solution of the Painlev\'e II model
problem, RHP \ref{rhpPII}.
Set $\sigma_2= \bigl({ 0 \enskip-i \atop i \enskip0 } \bigr)$ and recall that
$\sigma_3= \bigl( { 1 \enskip0 \atop0 \enskip1 } \bigr)$.
Taking into account the orientation of $\zeta(C_{\pm1, \mathrm{in}/\mathrm{out}}\cap\U_{\pm1})$,
we define the local models
%
%e104 #&#
\begin{eqnarray}
\label{localmodels} %
\bA_{-1}(z) &=& \bA_{-1}(z;
\gamma) := \sigma_2\sigma_3^{n} \bpsi \bigl(
\zeta(z;\gamma); s(\gamma) \bigr) \sigma_3^{n}
\sigma_2,\qquad z \in\U_{-1},
\nonumber
\\[-8pt]
\\[-8pt]
\nonumber
\bA_{1}(z) &=& \bA_{1}(z; \tg) := \sigma_2
\bpsi \bigl(\zeta(z; \tg); s(\tg) \bigr) \sigma_2,\qquad z \in
\U_{1}. %
\end{eqnarray}
Note from symmetries \eqref{psisymmetry} and~\eqref{zetasymmetry}
that these two models are related as
%
%e105 #&#
\begin{equation}
\label{eqA1A-1symm} \bA_1(z ,\tg) = \sigma_1
\sigma_3^n \bA_{-1}(-z, \tg)
\sigma_3^n \sigma_1, \qquad z\in
\U_{1}.
\end{equation}
From~\eqref{langervariable} and~\eqref{langervariable-2}, $\bA_{\pm
1}(z)$ satisfies the same jump condition as $\bT(z)$ in $\U_{\pm1}$,
respectively.

Define the ratio of the global parametrix to the exact problem $\bT(z)$,
%
%e106 #&#
\begin{equation}
\bR(z) = \bT(z) \bA^{-1}(z).
\end{equation}
Then $\bR(z)$ has no jumps inside $\U_{\pm1}$, but gains jumps on the
positively oriented boundaries $\partial\U_{\pm1}$. Let $\Sigma_R^0 =
\Sigma_{\mathrm{in}} \cup\Sigma_{\mathrm{out}} \cup C_{\mathrm{in}, \pm1} \cup C_{\mathrm{out}, \pm1}
\setminus(\U_1 \cup\U_{-1} )$; see Figure~\ref{figRcontours}. Then $\bR$ satisfies the following problem:

%
%f8 #&#
\begin{figure}

\includegraphics{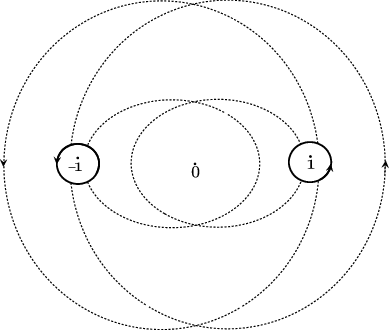}

\caption{The jump contours for the residual $\bR(z)$. The dashed lines
represent contours on which
the jumps are exponentially near identity.}
\label{figRcontours}
\end{figure}

%
%ri6.1 #&#
\begin{rhp}[\textsc{for $\bR(z)$}]\label{rhpR}
Find a $2\times2$ matrix $\bR(z)$ such that:
\begin{longlist}[(1)]
\item[(1)] $\bR(z)$ is analytic in $\C\setminus\Sigma_R$ where
$\Sigma_R = \Sigma_R^0 \cup\partial\U_1 \cup\partial\U_{-1}$.
\item[(2)] $\bR(z) \to I$ as $z \to\infty$.
\item[(3)] The boundary values of $\bR$ satisfy the jump relation
$\bR_+ = \bR_- V_R$ where
%
%e107 #&#
\begin{eqnarray}
V_R(z) = \cases{ \bA_1(z)^{-1} , &\quad  $z \in
\partial\U_1$,
\cr
\bA_{-1}(z)^{-1} , &\quad $z \in
\partial\U_{-1}$,
\cr
V_T(z), & \quad $z \in\Sigma_R^0$.
}
\end{eqnarray}
\end{longlist}
\end{rhp}

The jumps of $\bR(z)$ are now everywhere uniformly near identity. In fact,
for the parameters $(n,m,t)$ in regime~\eqref{eqregime2},
it follows from \eqref{eqVTexp-1},
%
%e108 #&#
\begin{equation}
\label{Rjumpbounds} \Vert V_R- I \Vert_{L^\infty(\Sigma_R^0 )} =
\mathcal{O} \bigl( e^{-c t} \bigr),
\end{equation}
and from \eqref{localmodels} and~\eqref{PIIexpansion} that (recall
that $\zeta(\U_{\pm1}; \gamma)$ contains a disk of radius $\ge
\mathcal{O}  ( t^{1/3}  )$ for all $\gamma\in[1-\delta,
1+\delta]$)
%
%e109 #&#
\begin{equation}
\label{Rjumpbounds-2} \Vert V_R- I \Vert_{L^\infty(\U_{\pm1})} =
\mathcal{O} \bigl( t^{-1/3} \bigr).
\end{equation}
(We will use a better estimate for the latter below.)
The above estimates establish that $\bR$ falls into the class of small
norm RHPs for any sufficiently large $t$. Let $C_-\dvtx L^2(\Sigma_R) \to
L^2(\Sigma_R)$ denote the usual Cauchy projection operator and define
%
%e110 #&#
\begin{eqnarray}
\label{CR}\qquad  C_{V_R} [ f ](z) := C_- \bigl[ f(w) (V_R - I)
\bigr] = \frac{1}{2 \pi i} \int_{\Sigma_R} \frac{f(w) (V_R(w) - I)}{(w-z)_-} \dd
w
\end{eqnarray}
and
%
%e111 #&#
\begin{equation}
\CC_R[f](z) := \frac{1}{2\pi i} \int_{\Sigma_R}
\frac{f(w) (V_R(w) - I)}{(w-z)} \dd w,
\end{equation}
which maps $f \in L^2(\Sigma_R)$ to an analytic function in $\C
\setminus\Sigma_R$.
Then as $C_-$ is a bounded $L^2$ operator whose operator norm is
uniformly bounded
(see, e.g.,~\cite{BK})
and the contours $\Sigma_R$ are finite length, it follows that $\|
C_{V_R} \|_{L^{2}\to L^2} = \mathcal{O}  ( t^{-1/3}  )$ for
large $t$ which
guarantees the existence of a unique solution to $ (1 - C_{V_R}
)
\mu= I$. Once the existence of $\mu(z)$ is established, it follows\vadjust{\goodbreak} immediately
from the general theory of RHPs that
%
%e112 #&#
\begin{eqnarray}
\label{Rsolution} \bR(z) := I + \CC_R[\mu](z)=I +
\frac{1}{2\pi i} \int_{\Sigma_R} \frac{ \mu(w) (V_R(w) - I)}{w-z} \,dw
\end{eqnarray}
is the solution of RHP \ref{rhpR}.

Unfolding the series of transformations $\bY\mapsto\bQ\mapsto\bS
\mapsto\bT\mapsto\bR$ we have $\bY(0) = \bR(0) \bigl({0\atop -1 }\enskip {1\atop 0}\bigr)$, and
from \eqref{pinrecovery} it follows that
%
%e113 #&#
\begin{equation}
\label{pin0inR} \pi_{n,m}(0;t) = - \bR_{12}(0;t,n,m) =
\bR_{21}(0;t,n,m).
\end{equation}
We now evaluate $\bR(0;t,n,m)$ explicitly for the first three terms in
the asymptotic expansion.
But we first consider the corresponding RHP for the continuous weight
in the next subsection.
We will compare the discrete weight problem to the continuous weight problem.

%s6.2 #&#
\subsection{Analysis of the continuous weight problem}\label
{seccontinuousweight}

A streamlined version of the above procedure reducing the discrete
problem, RHP \ref{rhpdiscrete}, to small-norm form can be used to
study the continuous weight problem, RHP \ref{rhpcontinuous}. \emph
{Using the same $g$-function used in the discrete case}, we define $\bY^\infty\mapsto\bS^\infty$ as in \eqref{Sdef}, replacing $\bQ$ with
$\bY^\infty$. The new RHP for $\bS^\infty$ features the single phase
$\theta(z;\gamma)$ defined by \eqref{allbandphases} which we recall
has a critical value at $z=-1$. In the ``exponentially small regime''
\eqref{eqregime1} estimate~\eqref{eqcase1-1} holds and just as in
Proposition~\ref{propexpsmall}, we have in the end
%
%e114 #&#
\begin{equation}
\label{eqpiinfexpsmall} \pi_{n,\infty}(0;t) = \mathcal{O} \bigl(
e^{-c n} \bigr)\qquad \mbox{for $(n,t)$ satisfying \eqref{eqregime1}.}
\end{equation}
In the Painlev\'e regime \eqref{eqregime2}, by introducing a
simplified version of transformation~\eqref{Tdef}, using only the
factors appearing in $\Omega_{\pm,0}$ to open lenses, one defines a
transformation $\bS^\infty\mapsto\bT^\infty$. The problem for $\bT^\infty$ is then approximated by a parametrix which is identity outside
a neighborhood $\U_{-1}$ of $z=-1$ and inside $\U_{-1}$ is approximated
by the same model as the discrete case, $\bA_{-1}(z)$ defined by~\eqref
{localmodels}. The result is a small norm problem $\bR^\infty$ for the
continuous case where
%
%e115 #&#
\begin{equation}
\label{continuousR} \bR^\infty(z) = I + \frac{1}{2\pi i} \int
_{\Sigma_{R^\infty}} \frac{\mu^\infty(w) ( V_{R^\infty
}(w)-I)}{w-z} \,dw,
\end{equation}
where
%
%e116 #&#
\begin{equation}
\label{continuouserrorjump} V_{R^\infty}(z) = \cases{ \bA_{-1}(z)^{-1},
& \quad $z \in\partial\U_{-1}$,
\cr
I+ \mathcal{O} \bigl( e^{-ct}
\bigr), &\quad  $z \in\Sigma_R^\infty \setminus \partial
\U_{-1}$. }
\end{equation}
Moreover, the continuous weight orthogonal polynomial $\pi_n^\infty(0)$
is given by
%
%e117 #&#
\begin{equation}
\label{pincontinuousrecovery} \qquad\pi_{n,\infty}(0;t) = - \bR^\infty_{12}(0;t,n)
= \bR^\infty_{21}(0;t,n) \qquad\mbox{for $(n,t)$ satisfying
\eqref{eqregime2}.}
\end{equation}

%%%%%%%%%%%%%%%%%%%%%%%%%%%%%%%%%%%%%%
%
% Expansion of R(0)
%
%%%%%%%%%%%%%%%%%%%%%%%%%%%%%%%%%%%%%%
%s6.3 #&#
\subsection{Expansion of $\bR(0)$}\label{secR0}

In this section we calculate the asymptotic expansion of
%
%e118 #&#
\begin{eqnarray}
\label{R0integralform--1} \bR(0) =I + \CC_R[\mu](0) = I +
\frac{1}{2\pi i} \int_{\Sigma_R} \frac{ \mu(w) (V_R(w)-I) }{w} \dd w
\end{eqnarray}
up to order $\mathcal{O}  ( t^{-1}  )$. We begin by
representing $\mu$ using its
Nuemann series expansion,
%
%e119 #&#
\begin{equation}
\label{museries--1} \mu(z) = I+ \sum_{k=1}^\infty
(C_R)^k[I],
\end{equation}
which, due to \eqref{Rjumpbounds} and~\eqref{Rjumpbounds-2},
convergences uniformly and absolutely.
In both~\eqref{R0integralform--1} and~\eqref{museries--1}, the
dominant contribution to the integral comes from the boundaries
$\partial\U_{\pm1}$.
In fact, denoting by $P_0$ the projection operator onto $\Sigma_R\setminus(\partial\U_{-1}\cup\partial U_1)$,
we find from \eqref{Rjumpbounds} that $\llVert C_R P_0 \rrVert_{L^2(\Sigma_R)\to L^2(\Sigma_R)} = \mathcal{O}  ( e^{-ct}
)$
and
$\llVert\CC_R P_0 \rrVert_{L^2(\Sigma_R)\to L^2(\Sigma_R)} =
\mathcal{O}  ( e^{-ct}  )$.

Denoting by $P_{\pm1}$ the projection operator onto $\partial\U_{\pm
}$, respectively,
define $C_{\pm1}:= C_RP_{\pm1}$ and $\CC_{\pm1}:= \CC_R P_{\pm1}$:
for any $f \in L^2(\Sigma_R)$,
%
%e120 #&#
\begin{eqnarray}
\label{C1} %
C_{\pm1}[ f ](z) &=& \frac{1}{2\pi i}
\oint_{\partial\U_{\pm1}} \frac{ f(w) (V_R(w) - I)
}{(w-z)_-} \dd w,\qquad z\in \Sigma_R,
\nonumber
\\[-8pt]
\\[-8pt]
\nonumber
\CC_{\pm1}[ f ](z) &=& \frac{1}{2\pi i} \oint_{\partial\U_{\pm1}}
\frac{ f(w) (V_R(w) - I) }{(w-z)} \dd w,\qquad  z\notin\Sigma_R. %
\end{eqnarray}
Then we find %Hence,
%
%e121 #&#
\begin{equation}
\label{R0integralform} \bR(0) =I + (\CC_{-1}+\CC_{1})[
\mu](0) + \mathcal{O} \bigl( e^{-ct} \bigr), %= I + \frac{1}{2\pi
%i} \int_{\partial\U_{-1}
\end{equation}
where
%
%e122 #&#
\begin{equation}
\label{museries} \mu(z) = I+ \sum_{k=1}^\infty
(C_{-1}+C_1)^k[I] (z)+ \mathcal{O} \bigl(
e^{-ct} \bigr).
\end{equation}
%
%by ignoring the contribution from $\Sigma_R \setminus( \partial\U_1

Recall $s(\gamma)$ defined in~\eqref{eqsdef}.
Introduce the shorthand $s=s(\gamma)$ and $\ts= s(\tg)$.
Using \eqref{PIIexpansion},~\eqref{eqpsiallsmall} and \eqref{localmodels} we have
\begin{subequations}\label{jumpexpansions}
%
%e123 #&#
\begin{eqnarray}\qquad
&&V_R(z)-I
\nonumber
\\[-6pt]
\\[-8pt]
\nonumber
&&\qquad = \cases{ \displaystyle\frac{\varphi_1(s)}{t^{1/3} f(z;\gamma)} + \frac
{\varphi_2(s)}{t^{2/3} f(z;\gamma)^2} +
\frac{\varphi_3(s)}{t^{-1} f(z;\gamma)^3}+ \mathcal{O} \biggl( \frac {e^{-c_0|s|^{3/2}}}{t^{4/3}} \biggr),\vspace*{3pt}\cr
 \quad z \in
\partial \U_{-1},\vspace*{3pt}
\cr
\displaystyle\frac{\phi_1(\ts)}{t^{1/3} f(z;\tg)} + \frac{\phi_2(\ts)}{t^{2/3}
f(z;\tg)^2} +
\frac{\phi_3(\ts)}{t^{-1} f(z;\tg)^3}+\mathcal{O} \biggl( \frac
{e^{-c_0|s|^{3/2}}}{t^{4/3}} \biggr), \vspace*{3pt}\cr
\quad z \in\partial
\U_{1} ,}
\end{eqnarray}
with
%
%e124 #&#
\begin{eqnarray}
\varphi_1(s) &=& \frac{1}{2i} \left[\matrix{
-u(s) & -(-1)^n q(s)
\cr
(-1)^n q(s) & u(s)} \right],
\nonumber\\
\qquad\quad\varphi_2(s) &=& \frac{1}{(2i)^2} \left[\matrix{
\displaystyle\frac{1}{2}u(s)^2 - \frac{1}{2}q(s)^2 &
(-1)^n \bigl( q(s) u(s) - q'(s) \bigr)
\cr
(-1)^n \bigl( q(s) u(s) - q'(s) \bigr) &
\displaystyle\frac{1}{2}u(s)^2 - \frac{1}{2}q(s)^2 }
\right],
\\
\varphi_3(s) &= &\frac{1}{(2i)^3} \left[ \matrix{ \alpha(s) &
(-1)^n \beta(s)
\cr
-(-1)^n \beta(s) & -\alpha(s) }
\right]\nonumber %
\end{eqnarray}
and
%
%e125 #&#
\begin{equation}
\phi_k(\ts) = \sigma_3^n
\varphi_k(\ts) \sigma_3^n,\qquad k=1,2,3,
\end{equation}
\end{subequations}
where $q$ is defined by \eqref{HMasymp} and $u$,
$\alpha$ and $\beta$ are defined in \eqref{PIIudef}--\eqref{PIIexpansion-d}.
% %
% \alpha(t) &= -\frac{u(t)^3}{6} + \frac{q(t)^2 u(t)}{2} + \log F(t)^2
%- \int_{\infty}^t q'(\xi)^2 d\xi\\
% \beta(t) &= -q(t) \lp\frac{q(s)^2}{2} + \frac{u(t)^2}{2}+t
% %
%Recall that $q$ and $u$ are defined by \eqref{udef}.

It follows from inserting the above expansions into~\eqref{R0integralform} and~\eqref{museries} that each iteration of $C_1$ or
$C_{-1}$ introduces a factor of $t^{-1/3}$; thus we are led to an
expansion of the form.
%
%e126 #&#
\begin{equation}
\label{R0expansion} \bR(0) = I + \sum_{k=1}^N
R^{(k)} t^{-k/3} + \mathcal{O} \biggl( \frac
{e^{-c_0|s|^{3/2}}}{t^{(N+1)/3}}
\biggr),
\end{equation}
where $R^{(1)} := t^{1/3}(\CC_1[I](0)+\CC_{-1}[I](0))$,
%
%e127 #&#
\begin{equation}
R^{(k)} := t^{k/3} \sum_{\vec{\tau} \in\{-1,1\}^{k-1}} (
\CC_1 + \CC_{-1})C_{\vec\tau}[I](0),\qquad  k \geq2.
\end{equation}
Here $C_{\vec\tau}$ is a multi-index understood as follows: given
$\vec
\tau=(\tau_1, \tau_2, \ldots,\tau_k) \in\{-1,1\}^k$ we define
$C_{\vec
\tau} := C_{\tau_1} C_{\tau_2} \cdots C_{\tau_k}$. Though we have
suppressed the dependence, each $R^{(k)}$ is a function of $t$.
Moreover, since both $s$ and the coefficients in the expansion \eqref
{CFUexpansion} depend on $\gamma$, each $R^{(k)} = \mathcal{O}
( 1  )$ with an
expansion in powers of $t^{-1/3}$.

At each order we can split the composition of Cauchy integrals into
three parts. Define
%
%e128 #&#
\begin{eqnarray}
\label{Rkdefs} %
R_1^{(k)} &=&
t^{k/3} \CC_1 C_1^{k-1}[I](0),\nonumber
\\
R_{-1}^{(k)} &=& t^{k/3} \CC_{-1}
C_{-1}^{k-1}[I](0),
\\
R_X^{(k)} &=& R^{(k)} -R^{(k)}_{1}
- R^{(k)}_{-1}. \nonumber %
\end{eqnarray}
Note that from definition, $R_X^{(1)}=0$.
Intuitively, the first two ``pure'' terms contain the expansions of the
continuous weight polynomials related to the marginal distributions
while the last term contains the ``cross'' terms. This can be made
concrete as follows. Let $\bR_{\pm1}(0)$ and $\bR_{X}(0)$ denote the
sum of each type of contribution to $\bR(0)$,
%
%e129 #&#
\begin{equation}
\label{eqR0expRp} \bR_{p}(0) := I + \sum
_{k=1}^\infty\frac{R^{(k)}_p}{t^{k/3}},\qquad  p=1,-1,X.
\end{equation}
%
%Then $\bR(0)= \bR_{-1}(0)+\bR_1(0)+\bR_X(0)+

Clearly, $\bR_1(0)$ and $\bR_{-1}(0)$ are the values at origin of
normalized Riemann--Hilbert problems whose jump conditions are
%
%e130 #&#
\begin{eqnarray}
(\bR_{-1})_+(z) &=& (
\bR_{-1})_-(z) \bA_{-1}(z, \gamma)^{-1},\qquad  z \in
\partial\U_{-1},
\nonumber
\\[-8pt]
\\[-8pt]
\nonumber
(\bR_1)_+(z) &=& (\bR_1)_-(z) \bA_1(z,
\tg)^{-1},\qquad z \in\partial\U_1.
\end{eqnarray}
Recalling \eqref{continuousR} and \eqref{continuouserrorjump}
we see that $\bR_{-1}(z)$ and $\bR^{\infty}(z;t,n)$ have the same jump
condition up to the exponentially small contributions from $\Sigma_{R^\infty} \setminus\partial\U_{-1}$. Hence
%
%e131 #&#
\begin{equation}
\bR^{\infty}(0;t,n) = \bigl[ I +\mathcal{O} \bigl(
e^{-ct} \bigr) \bigr]\bR_{-1}(0). %
\end{equation}
Also from~\eqref{eqA1A-1symm}, the jump of $\bR_1(z)$ is same as that
of $\sigma_1 \sigma_3^n \bR^{\infty}(0,t,2m-n)\sigma_3^n \sigma_1$, and
hence we find that
%
%e132 #&#
\begin{equation}
\sigma_1 \sigma_3^n
\bR^{\infty}(0,t,2m-n)\sigma_3^n
\sigma_1 = \bigl[I + \mathcal{O} \bigl( e^{-ct} \bigr)
\bigr] \bR_{1}(0) . %
\end{equation}
Therefore, from \eqref{pincontinuousrecovery} it follows that
%
%e133 #&#
\begin{eqnarray}
\label{eq122-1} %
\pi_{n,\infty}(0;t) &=& -(\bR_{-1})_{12}(0)
+ \mathcal{O} \bigl( e^{-ct} \bigr),
\nonumber
\\[-8pt]
\\[-8pt]
\nonumber
\pi_{2m-n,\infty}(0;t) &=& (-1)^n(\bR_{1})_{12}(0)
+ \mathcal{O} \bigl( e^{-ct} \bigr), %
\end{eqnarray}
and hence from \eqref{pin0inR}, \eqref{R0expansion} and
\eqref
{eqR0expRp}, we find that
%
%e134 #&#
\begin{eqnarray}
\label{eq122} \quad %
\pi_{n,m}(0;t) &= \pi_{n,\infty}(0) -
(-1)^n \pi_{2m-n,\infty}(0) - (\bR_X)_{12}(0)
+ \mathcal{O} \bigl( e^{-ct} \bigr) . %
\end{eqnarray}

From~\eqref{eqR0expRp}, we now need to evaluate $R_p^{(k)}, p=-1,1,X$, $k=1,2,3$.
This calculation is a straightforward but lengthy application of
residue calculus. We summarize the result of the calculations which
follow directly from the definitions~\eqref{Rkdefs},~\eqref
{C1},~\eqref{jumpexpansions}, making use of the expansions \eqref{CFUexpansion} and \eqref{sexpansion}.
It is helpful to note that
the symmetry~\eqref{eqA1A-1symm} between $\bA_1$ and $\bA_{-1}$
implies that
%
%e135 #&#
\begin{equation}
\label{eqCKsymm1} %
\CC_1 = T\CC_{-1}^{(\gamma\mapsto\tg)}T,\qquad
C_1 = TC_{-1}^{(\gamma\mapsto\tg)}T, %
\end{equation}
where $\CC_{-1}^{(\gamma\mapsto\tg)}$ and $C_{-1}^{(\gamma\mapsto
\tg
)}$ denote $\CC_{-1}$ and $C_{-1}$ with $\gamma$ replaced by $\tg$,
respectively, and $T$ is the operator defined by
%
%e136 #&#
\begin{equation}
\label{eqCKsymm2} %
Tf(z):= \sigma_1
\sigma_3^n f(-z) \sigma_3^n
\sigma_1. %
\end{equation}
In particular, note that $TI=I$, $R^{(k)}_1= TR_{-1}^{(k)}|_{\gamma\to
\tg}$.

Let $\mathrm{Err}$ and $\tilde \mathrm{Err}$ denote any terms satisfying
%
%e137 #&#
\begin{equation}
\label{eqhmm1} %
\mathrm{Err}=\mathcal{O} \bigl( e^{-c_0|s(\gamma(\tau))|^{3/2}} \bigr),\qquad
\tilde{\mathrm{Err}}= \mathcal{O} \bigl( e^{-c_0|s(\tg(\tau))|^{3/2}} \bigr) . %
\end{equation}
Denoting by $[A,B]$ and $\{A,B\}$ the commutator and anti-commutator of
matrices $A$ and $B$, respectively,
we find from an explicit evaluation that [making use of~\eqref{eqpsiallsmall}]
\begin{subequations}\label{pureR0terms}
%
%e138 #&#
\begin{eqnarray}
R^{(1)}_{-1} &=& 2i \biggl(1- \frac1{30} (
\gamma-1) \biggr) \varphi_1(s) - \frac{(2i)^3}{20 t^{2/3}}
\varphi_3(s)
\nonumber\\
&&{} + \bigl(|\gamma-1|^2+t^{-2/3}|\gamma-1|+t^{-1}
\bigr)\mathrm{Err} ,
\nonumber
\\[-8pt]
\\[-8pt]
\nonumber
R^{(1)}_{1} &=& - 2i \biggl(1- \frac1{30} (\tg-1) \biggr)
\phi_1(\ts) + \frac{(2i)^3}{20 t^{2/3}} \phi_3(\ts)
\\
&&{} + \bigl(|\tilde\gamma-1|^2+t^{-2/3}|\tilde
\gamma-1|+t^{-1} \bigr)\tilde{\mathrm{Err}} ,\nonumber
\end{eqnarray}
%
%e139 #&#
\begin{eqnarray}
R^{(2)}_{-1} &=& \frac{(2i)^2}{2}
\varphi_1(s)^2 - \frac{(2i)^3 \varphi_1(s) \varphi_2(s)}{20 t^{1/3}
} +
\frac{(2i)^3 \varphi_2(s) \varphi_1(s)}{10 t^{1/3} }
\nonumber\\
&&{} + \bigl(|\gamma-1| + t^{-2/3} \bigr)\mathrm{Err} ,
\nonumber
\\[-8pt]
\\[-8pt]
\nonumber
R^{(2)}_{1} &=& \frac{(2i)^2}{2} \phi_1(
\ts)^2 + \frac{(2i)^3 \sigma_1(\ts) \phi_2(\ts)}{20 t^{1/3} } - \frac
{(2i)^3 \phi_2(\ts) \varphi_1(\ts)}{10 t^{1/3} }
\\
&&{} + \bigl(|\tilde\gamma-1| + t^{-2/3} \bigr)\tilde{\mathrm{Err}}
,\nonumber
\end{eqnarray}
%
%e140 #&#
\begin{eqnarray}
R^{(3)}_{-1} &=& \frac{3(2i)^3}{20}
\varphi_1(s)^3 + \bigl(|\gamma-1| + t^{-1/3}
\bigr) \mathrm{Err},
\nonumber
\\[-8pt]
\\[-8pt]
\nonumber
R^{(3)}_{1} &=& -\frac{3(2i)^3}{20} \phi_1(
\ts)^3 + \bigl(|\tilde\gamma-1| + t^{-1/3} \bigr)\tilde{\mathrm{Err}} ,
\end{eqnarray}
\end{subequations}

\begin{subequations}\label{mixedR0terms}
%
%e141 #&#
\begin{eqnarray}
R^{(2)}_{X} &=& -\frac{(2i)^2}{2} \bigl\{
\varphi_1(s), \phi_1(\ts) \bigr\} \nonumber\\
&&{}-
\frac{(2i)^3}{4} \bigl( \bigl[\varphi_2(s) ,\phi_1(
\ts) \bigr] + \bigl[\varphi_1(s),\phi_2(\ts) \bigr]
\bigr) t^{-1/3}
\\
&&{}+ \bigl(|\gamma-1| + t^{-2/3} \bigr)\mathrm{Err} + \bigl(|\tilde\gamma-1| +
t^{-2/3} \bigr)\tilde{\mathrm{Err}} ,\nonumber
\end{eqnarray}
%
%
%e142 #&#
\begin{eqnarray}
R^{(3)}_{X} &=& \frac{(2i)^3}{4} \bigl\{
\varphi_1(s) \phi_1(\ts) \bigr\} \bigl(
\phi_1(\ts) - \varphi_1(s) \bigr)
\nonumber
\\[-8pt]
\\[-8pt]
\nonumber
&&{}+ \bigl(|\gamma-1| + t^{-1/3} \bigr)\mathrm{Err} + \bigl(|\tilde\gamma-1| +
t^{-1/3} \bigr)\tilde{\mathrm{Err}} .
\end{eqnarray}
\end{subequations}

Recall that $R_X^{(1)}=0$.
Note that
%
%e143 #&#
\begin{equation}
\bigl\{\varphi_1(s), \phi_1(\ts) \bigr\} = 2
\bigl(u(s)u(\ts) -(-1)^n q(s)q(\ts) \bigr)I.
\end{equation}
From \eqref{eq122-1} and \eqref{eq122} using \eqref{jumpexpansions}
and \eqref{pureR0terms}--\eqref{mixedR0terms}, we obtain the
following:

%
%pr6.1 #&#
\begin{prop}\label{propPainRegime}
Set
%
%e144 #&#
\begin{eqnarray}
\label{eqregime2-3} %
g_1(y, \ty) &:=& \tfrac{1}{2}
\bigl( u'(y)q(\ty) + u(y) q'(\ty) \bigr) ,
\nonumber
\\[-8pt]
\\[-8pt]
\nonumber
g_2(y,\ty) &:= &\tfrac{1}{2} \bigl( q(y) u'(\ty) +
q'(y) u(\ty) \bigr). %
\end{eqnarray}
Let $\pi_{n,m}(z)$ be the orthogonal polynomial given in \eqref{pinrecovery}.
Let $\pi_{n,\infty}(z)$ be the orthogonal polynomial given in \eqref
{continuouspinrecovery}.
There exists $\delta>0$ such that for any fixed $L>0$, if
%
%e145 #&#
\begin{equation}
\label{eqregime2-1} \qquad 2t-Lt^{1/3}\le n\le2t(1+\delta),\qquad
2t-Lt^{1/3} \le2m-n\le2t(1+\delta),
\end{equation}
then there exists constants $c_0>0$ and $t_0>0$ such that
%
%e146 #&#
\begin{eqnarray}
\label{eqregime2-2} %
&&\pi_{n,m}(0;t)\nonumber\\
&&\qquad =
\pi_{n,\infty}(0;t) - (-1)^n \pi_{2m-n,\infty}(0;t)
\nonumber
\\[-8pt]
\\[-8pt]
\nonumber
&&\qquad\quad{}+ \frac{g_1(s(\gamma),s(\tg)) - (-1)^n g_2(s(\gamma),s(\tg)) }{t}
\\
&&\qquad\quad{}+ \mathcal{O} \bigl( \bigl( t^{-4/3}+t^{-2/3}|
\gamma-1|+t^{-2/3}| \tg-1| \bigr)e^{-c_0(|s(\gamma)|^{3/2}+|s(\tg)|^{3/2})} \bigr) \nonumber%
\end{eqnarray}
for all $t\ge t_0$,
where
%
%e147 #&#
\begin{equation}
\label{eqregime2-5} \gamma:= \frac{n}{2t}, \qquad\tg:= \frac{2m-n}{2t}
\end{equation}
and $s(u)$ is defined in~\eqref{eqsdef} which satisfies [see~\eqref{sexpansion}]
%
%e148 #&#
\begin{eqnarray}
\label{sexpansion-1} s(u) = 2t^{2/3}(u-1) - \frac{(2t^{2/3}(u-1) )^2}{60}
t^{-2/3} + \mathcal{O} \bigl( t^{2/3}(u-1)^3 \bigr).
\end{eqnarray}
\end{prop}
%
%with
% y:= \frac{n-2t}{t^{1/3}},  \ty:= \frac{2m-n-2t}{t^{1/3}}.

We also have the following:
%
%pr6.2 #&#
\begin{prop}\label{propPainRegime2}
For $t\ge t_0$,
%
%e149 #&#
\begin{eqnarray}
\label{eqregime2-6} %
(-1)^n\pi_{n,\infty}(0;t)& =&
\frac{1}{t^{1/3}} q \bigl(s(\gamma) \bigr) \biggl(1 - \frac{\gamma-1}{30}
\biggr)+ \frac{1}{t} h \bigl(s(\gamma) \bigr)
\nonumber
\\[-8pt]
\\[-8pt]
\nonumber
&&{}+ \mathcal{O} \bigl( \bigl( t^{-4/3}+t^{-2/3}|\gamma-1|
\bigr)e^{-c_0|s(\gamma)|^{3/2}} \bigr), %
\end{eqnarray}
where
%
%e150 #&#
\begin{equation}
\label{eqh1} %h(y) := \frac{1}{5} u(y) q'(y) - y q(y) -\frac{1}{2}
%q(y) u(y)^2 - \frac{1}{2} q(y)^3 \\
h(y ):= \tfrac{1}{5} u(y)
q'(y) - \tfrac{1}{5} q^3 - \tfrac{1}{20} y
q(y).
\end{equation}
\end{prop}

%From \eqref{eq122} the continuous weight OP constant $\pi_n^{
% \pi_{n,\infty}(0) = - \frac{1}{t^{1/3}} \lp R^{(1)}_{-1} \rp_{12} -
%R^{(3)}_{-1}\rp_{12} + \bigo{t^{-4/3}}.
%Plugging in the explicit forms of these terms as described by

%%%%%%%%%%%%%%%%%%%%%%%%%%%%%%%%%%%%%%%
% %
% Calculation of G_{j,k}(t) %
% %
%%%%%%%%%%%%%%%%%%%%%%%%%%%%%%%%%%%%%%%
%s7 #&#
\section{\texorpdfstring{Proof of Theorem~\protect\ref{thm3} and Corollary~\protect\ref{cor1}}
{Proof of Theorem 1.1 and Corollary 1.1}}\label{secGjkt}

We now evaluate the asymptotics of $\Prob\{ \Cro_t \leq k, \Nes_t
\leq j \}$
when
%
%e151 #&#
\begin{equation}
\label{eqA55-6} j= \bigl[t + 2^{-1} xt^{1/3} \bigr],\qquad  k=
\bigl[ t + 2^{-1} x' t^{-1/3} \bigr],
\end{equation}
where $x, x'\in\R$ are fixed, and
$[a]$ denotes the largest integer no larger than $a$.
We define $x_t$ and $x'_t$ by
%
%e152 #&#
\begin{equation}
\label{eqA55-7} x_t:= \frac{(2j+1)-2t}{t^{1/3}},\qquad  x_t':=
\frac{(2k+1)-2t}{t^{1/3}}
\end{equation}
so that
%
%e153 #&#
\begin{equation}
\label{eqA55-8} 2j+1= 2t+x_tt^{1/3},\qquad  2k+1=
2t+x_t't^{1/3}.
\end{equation}
Then $x_t= x+O(t^{-1/3})$ and $x_t'=x'+O(t^{-1/3})$.

%Note that if we change $j$ by $O(1)$, then it can be written same as
%above with $x$ replaced by $x + O(t^{-1/3})$.
%It will be clear in the calculation below that this change of $x$ does
%not affect the final result. \marginpar{Make sure. Write better.}

From Proposition~\ref{prop1}, we have
%
%e154 #&#
\begin{eqnarray}
\label{eqA55-1}
&&\log\Prob\{ \Cro_t \leq k, \Nes_t \leq j
\}
\nonumber
\\[-8pt]
\\[-8pt]
\nonumber
%= \log\lp e^{-t^2/2} G_{k,j}(t) \rp\\
&&\qquad= \int_0^t \pi_{2j+1,m}(0;
\tau) \,d \tau+ \int_0^t \int
_0^s \QQ_j^m(\tau) \,d
\tau \,ds,
\end{eqnarray}
where
%
%e155 #&#
\begin{equation}
\label{eqT-1-1} %
\QQ_j^m(\tau) =-
\RR_j^m(\tau)-\SSS_j^m(\tau)+
\RR_j^m(\tau)\SSS_j^m(\tau)
\end{equation}
and
%
%e156 #&#
\begin{equation}
\label{eqT-1-1-1} %
\qquad
\RR_j^m(\tau) :=
\pi_{2j,m}(0;\tau)\pi_{2j+2,m}(0;\tau), \qquad \SSS_j^m(
\tau) := \bigl|\pi_{2j+1,m}(0; \tau)\bigr|^2. %
\end{equation}

From Proposition~\ref{propexpsmall} [substituting $\tau$ for $t$
in~\eqref{eqregime1-3}], we find that
the above integrals away from the interval $[(1-\varepsilon)t, t]$, for
any fixed $\varepsilon>0$, are exponentially small in $t$,
%
%e157 #&#
\begin{eqnarray}
\label{eqA55-2} \qquad &&\log\Prob\{ \Cro_t \leq k, \Nes_t \leq j
\}
\nonumber
\\[-8pt]
\\[-8pt]
\nonumber
%= \log\lp e^{-t^2/2} G_{k,j}(t) \rp\\
&&\qquad= \int_{t(1-\varepsilon)}^t
\pi_{2j+1,m}(0; \tau) \,d \tau+ \int_{t(1-\varepsilon)}^t
\int_{t(1-\varepsilon)}^s \QQ_j^m(
\tau) \,d \tau \,ds + O \bigl(e^{-ct} \bigr).
\end{eqnarray}
We can take $\varepsilon>0$ small enough so that Proposition~\ref
{propPainRegime} is applicable to $\pi_{2j+\ell,m}(0;\tau)$ for
$\ell
=0,1,2$ and $\tau\in[(1-\varepsilon)t, t]$.

Now by the same argument, we have
%
%e158 #&#
\begin{eqnarray}
\label{eqA55-3} \qquad &&\log\Prob\{ \Nes_t \leq j \}
\nonumber
\\[-8pt]
\\[-8pt]
\nonumber
%= \log\lp e^{-t^2/2} G_{k,j}(t) \rp\\
&&\qquad= \int_{t(1-\varepsilon)}^t
\pi_{2j+1,\infty}(0; \tau) \,d \tau+ \int_{t(1-\varepsilon)}^t
\int_{t(1-\varepsilon)}^s \QQ_j^{\infty}(
\tau) \,d \tau \,ds + O \bigl(e^{-ct} \bigr)
\end{eqnarray}
and
%
%e159 #&#
\begin{eqnarray}
\label{eqA55-4}\qquad
&& \log\Prob\{ \Cro_t \leq k \}
\nonumber
\\[-8pt]
\\[-8pt]
\nonumber
%= \log\lp e^{-t^2/2} G_{k,j}(t) \rp\\
&&\qquad= \int_{t(1-\varepsilon)}^t
\pi_{2k+1,\infty}(0; \tau) \,d \tau+ \int_{t(1-\varepsilon)}^t
\int_{t(1-\varepsilon)}^s \QQ_k^{\infty}(
\tau) \,d \tau \,ds + O \bigl(e^{-ct} \bigr).
\end{eqnarray}
Consider
%
%e160 #&#
\begin{eqnarray}
\label{eqA55-5} \log\Prob\{ \Cro_t \leq k, \Nes_t \leq j
\} - \log\Prob\{ \Nes_t \leq j \} -\log\Prob\{ \Cro\leq k \} .
\end{eqnarray}

We first consider the three single integrals.
From~\eqref{eqregime2-2} applied to $n=2j+1$ and $t$ replaced by
$\tau$,
we have
%
%e161 #&#
\begin{eqnarray}
\label{eq140}
&&\int_{t(1-\varepsilon)}^t \bigl[
\pi_{2j+1,m}(0;\tau) - \pi_{2j+1,\infty}(0;\tau) - \pi_{2k+1,\infty}(0;
\tau) \bigr] \,d \tau
\nonumber\\
&&\qquad= \int_{t(1-\varepsilon)}^t \frac{1}{\tau} \bigl[
g_1\bigl(s \bigl(\gamma(\tau)\bigr), s \bigl(\tg(\tau) \bigr) \bigr) +
g_2\bigl(s \bigl(\gamma(\tau) \bigr), s \bigl(\tg(\tau) \bigr)\bigr) \bigr] \,d
\tau
\\
&&\qquad\quad{}+ \mathcal{O} \biggl( \int_{t(1-\varepsilon)}^t \bigl(
\tau^{-4/3}+ \tau^{-2/3}\bigl|\gamma(\tau)-1\bigr| \bigr)e^{-c_0 |s(\gamma(\tau
))|^{3/2}}
\,d\tau \biggr),\nonumber
\end{eqnarray}
where
%
%e162 #&#
\begin{equation}
\label{eqA55-6-1} \gamma(\tau):= \frac{2j+1}{2\tau}, \qquad \tg (\tau):=
\frac{2k+1}{2\tau}.
\end{equation}
Changing the integration variable $\tau\mapsto\eta$ as
%
%e163 #&#
\begin{equation}
\label{eqA55-10} \tau= t- 2^{-1} \eta t^{1/3},
\end{equation}
the integral involving $g_1$ in~\eqref{eq140} becomes
%
%e164 #&#
\begin{eqnarray}
\label{eqA55-11} \frac{1}{2t^{2/3}} \int_0^{2\varepsilon t^{2/3}}
g_1\bigl(s \bigl(\gamma(\tau)\bigr), s \bigl(\tg(\tau) \bigr) \bigr)
\frac{d\eta}{1-2^{-1} \eta t^{-2/3}}.
\end{eqnarray}
Note that from~\eqref{sexpansion},
%
%e165 #&#
\begin{eqnarray}
\label{eqA55-12} s \bigl(\gamma(\tau) \bigr)&=& (x_t+\eta) +
\mathcal{O} \bigl( \eta^2 t^{-2/3} \bigr),
\nonumber
\\[-8pt]
\\[-8pt]
\nonumber
  s \bigl(\tg (\tau)
\bigr)&= &\bigl(x'_t+\eta \bigr) + \mathcal{O} \bigl(
\eta^2 t^{-2/3} \bigr).
\end{eqnarray}
Also note that from its definition, $g_1(x_0+\eta, x_0'+\eta)$ is
integrable for $\eta\in[0, \infty)$ for any fixed $x_0, x_0'\in\R$.
Thus, we obtain that integral~\eqref{eqA55-11} equals
%
%e166 #&#
\begin{equation}
\label{eqA55-12} \frac{1}{2t^{2/3}} \int_0^{\infty}
g_1 \bigl(x_t+\eta, x_t'+\eta
\bigr) \,d\eta+ \mathcal{O} \bigl( t^{-4/3} \bigr).
\end{equation}
The integral involving $g_2$ in~\eqref{eq140} equals the same integral
with $g_1$ replaced by $g_2$.
On the other hand, it is easy to see that the error term in~\eqref
{eq140} is
%
%e167 #&#
\begin{eqnarray}
\label{eqA55-13} \mathcal{O} \biggl( t^{1/3} \int_0^{\infty}
t^{-4/3}\bigl(1 + |x_t+\eta|\bigr) e^{-c_0|x_t+\eta|^{3/2}} \,d\eta \biggr) =
\mathcal{O} \bigl( t^{-1} \bigr).
\end{eqnarray}
Thus, replacing $x_t$ and $x_t'$ by $x$ and $x'$, which incurs an error
of order $\mathcal{O}  ( t^{-1/3}  )$, \eqref{eq140} equals
%
%e168 #&#
\begin{eqnarray}
\label{eqA55-14} %
\frac{1}{2t^{2/3}} \int_0^{\infty}
\bigl[ g_1 \bigl(x+\eta, x'+\eta \bigr)
+g_2 \bigl(x+\eta, x'+\eta \bigr) \bigr] \,d\eta+
\mathcal{O} \bigl( t^{-1} \bigr) . %
\end{eqnarray}
Now inserting definition~\eqref{eqregime2-3}, we can perform the
integration, and we find that~\eqref{eq140} equals
%
%e169 #&#
\begin{equation}
\label{eqA55-15} %
\frac{-1}{4t^{2/3}} \bigl[ u(x)q
\bigl(x' \bigr) + q(x)u \bigl(x' \bigr) \bigr] +
\mathcal{O} \bigl( t^{-1} \bigr) . %
\end{equation}

We now consider the part of~\eqref{eqA55-5} that comes from the
three double integrals.
We need to evaluate $\QQ_j^m(\tau) - \QQ_j^\infty(\tau) - \QQ_k^{\infty
}(\tau)$.
%%
% \QQ_j^m(\tau) - \QQ_j^\infty(\tau) - \QQ_k^{\infty}(\tau)
%%
Setting
%
%e170 #&#
\begin{equation}
\label{eqA55-21} %
\gamma^{\pm}(\tau):= \frac{2j+1\pm1}{2\tau}
= \gamma(\tau)\pm\frac{1}{2\tau}, %
\end{equation}
we see from~\eqref{sexpansion} that
%
%e171 #&#
\begin{equation}
\label{eqA55-22} %
s \bigl(\gamma^{\pm}(\tau) \bigr)= s
\bigl(\gamma(\tau) \bigr) \pm\frac{1}{\tau^{1/3}} + O \bigl(t^{-1/3} \bigl(
\gamma(\tau)-1 \bigr) \bigr). %
\end{equation}
Let us set
%
%e172 #&#
\begin{equation}
\label{eqA55-22-1} %
\xi:= s \bigl(\gamma(\tau) \bigr),\qquad \tilde{
\xi}:= s \bigl(\tg(\tau) \bigr) %
\end{equation}
to ease the notational burden.
Then,~\eqref{eqregime2-6} implies, using~\eqref{eqpsiallsmall}, that
%
%e173 #&#
\begin{eqnarray}
\label{eqA55-23} %
\pi_{2j+1\pm1, \infty}(0;\tau) &=& -
\pi_{2j+1, \infty}(0; \tau) \pm q'(\xi) \frac{1}{\tau^{2/3}} +
\frac12 q''(\xi) \frac{1}{\tau}
\nonumber
\\[-8pt]
\\[-8pt]
\nonumber
&&{}+ \tau^{-4/3} \mathrm{Error} , %
\end{eqnarray}
where throughout the rest of this section we use the notation $\mathrm{Error}$
to denote any term satisfying
%
%e174 #&#
\begin{eqnarray}
\label{eqA55-23-1} %
\mathrm{Error}&=&\mathcal{O} \bigl( \bigl( 1+
\tau^{2/3}\bigl| \gamma(\tau)-1\bigr| \bigr)e^{-c_0|s(\gamma(\tau))|^{3/2}} \bigr)
\nonumber
\\[-8pt]
\\[-8pt]
\nonumber
&&{} + \mathcal{O} \bigl( \bigl( 1+\tau^{2/3}\bigl|\tg(\tau)-1\bigr|
\bigr)e^{-c_0|s(\tg (\tau))|^{3/2}} \bigr) . %
\end{eqnarray}
Note that
%
%e175 #&#
\begin{equation}
\label{eqA55-23-1-1} %
\int_{t(1-\varepsilon)}^t
\int_{t(1-\varepsilon)}^t \mathrm{Error} \,d\tau \,ds = O
\bigl(t^{2/3} \bigr). %
\end{equation}
%
%the integral of $\mathrm{Error}$ from $t(1-\varepsilon)$ to $t$ in $\tau$ is
%$O(t^{1/3})$.
Also, note that from~\eqref{eqregime2-6}, \eqref{eqA55-23}
implies, in particular, that
%
%e176 #&#
\begin{equation}
\label{eqA55-23-2} %
\pi_{2j+1\pm1, \infty}(0;\tau) = q(\xi)+
\tau^{-2/3} \mathrm{Error}, %
\end{equation}
and clearly asymptotics~\eqref{eqA55-23} and~\eqref{eqA55-23-2}
also hold when $j$ is replaced by $k$ and $\xi$ is replaced by $\tilde
\xi$.

From~\eqref{eqregime2-2},
%
%e177 #&#
\begin{eqnarray}
\label{eqA55-31} %
% &\SSS_j^m(\tau)-\SSS_j^\infty(\tau)- \SSS_k^\infty(\tau)\\
&&\bigl|\pi_{2j+1, m}(0;
\tau)\bigr|^2 - \bigl|\pi_{2j+1, \infty}(0;\tau)\bigr|^2 - \bigl|
\pi_{2k+1, \infty}(0;\tau)\bigr|^2
\nonumber\\
&&\qquad= 2 \pi_{2j+1, \infty}(0;\tau)\pi_{2k+1, \infty}(0;\tau)
\nonumber
\\[-8pt]
\\[-8pt]
\nonumber
&&\qquad\quad{} + \frac2{\tau} \bigl[ g_1(\xi, \tilde{\xi}) + g_2(
\xi, \tilde{\xi}) \bigr] \bigl[ \pi_{2j+1, \infty}(0;\tau)+ \pi_{2k+1, \infty}(0;
\tau) \bigr]
\\
&&\qquad\quad{} + \tau^{-5/3} \mathrm{Error} .\nonumber %
\end{eqnarray}
Thus, from~\eqref{eqregime2-6},
%
%e178 #&#
\begin{eqnarray}
\label{eqA55-31-1} %
&&\SSS_j^m(\tau)-
\SSS_j^\infty(\tau)- \SSS_k^\infty(\tau)\nonumber
\\
&&\qquad= 2 \pi_{2j+1, \infty}(0;\tau)\pi_{2k+1, \infty}(0;\tau)
\\
&&\qquad\quad{} - \frac{2}{\tau^{4/3}} \bigl[ g_1(\xi, \tilde{\xi}) +
g_2(\xi, \tilde{\xi}) \bigr] \bigl[ q(\xi) + q(\tilde{\xi}) \bigr] +
\tau^{-5/3} \mathrm{Error}.\nonumber %
\end{eqnarray}
Similarly, using~\eqref{eqregime2-2} and~\eqref{eqA55-23-2}, we obtain
%
%e179 #&#
\begin{eqnarray}
\label{eqA55-30-1} %
&&\RR_j^m(\tau)-
\RR_j^\infty(\tau)-\RR_k^\infty(\tau)
\nonumber\\
% &=\pi_{2j, m}(0;\tau)\pi_{2j+2, m}(0;\tau) - \pi_{2j, \infty}(0;\tau)
% &= - \pi_{2j, \infty}(0;\tau)\pi_{2k, \infty}(0;\tau)
% - \pi_{2j+2, \infty}(0;\tau)\pi_{2k+2, \infty}(0;\tau) \\
% &+ \frac1{\tau} \big[ g_1(\xi, \tilde{\xi}) - g_2(\xi, \tilde{\xi})
% &+ \bigo{ \big( t^{-5/3}+t^{-1}|\gamma-1|\big)e^{-c_0(|s(
&&\qquad= - \pi_{2j, \infty}(0;
\tau)\pi_{2k, \infty}(0;\tau) - \pi_{2j+2,
\infty}(0; \tau)
\pi_{2k+2, \infty}(0;\tau)
\\
&&\qquad\quad + \frac{2}{\tau^{4/3}} \bigl[ g_1(\xi, \tilde{\xi}) -
g_2(\xi, \tilde{\xi}) \bigr] \bigl[ q(\xi) - q(\tilde{\xi}) \bigr] +
\tau^{-5/3}\mathrm{Error}\nonumber %
\end{eqnarray}
and
%Finally,
%
%e180 #&#
\begin{eqnarray}
\label{eqA55-32} %
&&\RR_j^m(\tau)
\SSS_j^m(\tau)-\RR_j^\infty(\tau)
\SSS_j^\infty(\tau) - \RR_k^\infty(\tau)
\SSS_k^\infty(\tau)
\nonumber
\\[-8pt]
\\[-8pt]
\nonumber
% &\pi_{2j, m}(0;\tau)\pi_{2j+2, m}(0;\tau)|\pi_{2j+1, m}(0;\tau)|^2 \\
% &- \pi_{2j, \infty}(0;\tau)\pi_{2j+2, \infty}(0;\tau) |\pi_{2j+1,
% - \pi_{2k, \infty}(0;\tau)\pi_{2j+2, \infty}(0;\tau)|\pi_{2k+1,
% &=\frac1{\tau^{4/3}}\bigg\{ \big[ q(\xi) - q(\tilde{\xi}) \big]\big[
%q(\xi) - q(\tilde{\xi}) \big]\big[ q(\xi) + q(\tilde{\xi}) \big]^2- q(
% &+ \bigo{ \big( t^{-5/3}+t^{-1}|\gamma-1|\big)e^{-c_0(|s(
&&\qquad = \frac{-2}{\tau^{4/3}} q(
\xi) q(\tilde{\xi}) + \tau^{-5/3} \mathrm{Error}. %
\end{eqnarray}
Therefore,
since
%
%e181 #&#
\begin{eqnarray}
\label{eqA55-33} %
&&\pi_{2j, \infty}(0; \tau)
\pi_{2k, \infty}(0; \tau) + \pi_{2j+2, \infty}(0; \tau) \pi_{2k+2,
\infty}(0;
\tau)
\nonumber\\
&&\quad{}- 2\pi_{2j+1, \infty}(0; \tau) \pi_{2k+1, \infty}(0; \tau)
\\
&&\qquad= \frac1{\tau^{4/3}} \bigl[ q(\xi)q''(
\tilde\xi) + q''(\xi)q(\tilde\xi)+2q'(
\xi)q'(\tilde\xi) \bigr] + \tau^{-5/3}\mathrm{Error},\nonumber %
\end{eqnarray}
we obtain, by using the definition of $g_1, g_2$ and by using the fact
that $q^2=u'$ and $2qq'=u''$, that
%
%e182 #&#
\begin{equation}
\label{eqA55-34} %
\QQ_j^m (\tau) -
\QQ_j^\infty(\tau) - \QQ_k^\infty(\tau)
= \frac{1}{\tau^{4/3}} \UU(\xi, \tilde{\xi}) + \tau^{-5/3} \mathrm{Error},
\end{equation}
where $\xi:= s(\gamma(\tau))$, $\tilde{\xi}:= s(\tg(\tau))$ are defined
in~\eqref{eqA55-22-1},
and we have set
%
%e183 #&#
\begin{eqnarray}
\label{eqA55-35} %
\UU(\xi, \tilde{\xi})& := & u''(
\xi) u(\tilde\xi) + 2u'(\xi) u'(\tilde\xi) +u(
\xi)u''(\tilde\xi)
\nonumber
\\[-8pt]
\\[-8pt]
\nonumber
&&{}+ q''(\xi) q(\tilde\xi) + 2 q'(\xi)
q'(\tilde\xi) + q(\xi) q''(\tilde \xi).
\end{eqnarray}

We insert~\eqref{eqA55-34} into the integral
%
%e184 #&#
\begin{equation}
\label{eqA55-40} %
\int_{t(1-\varepsilon)}^t \int
_{t(1-\varepsilon)}^t \bigl[ \QQ_j^m (
\tau) - \QQ_j^\infty(\tau) - \QQ_k^\infty(
\tau) \bigr] \,d\tau \,ds, %
\end{equation}
and evaluate it by changing variables $\tau\mapsto\eta$, $\tau=
t-2^{-1} \eta t^{1/3}$ and
$s\mapsto\zeta$, $s= t - 2^{-1} \zeta t^{1/3}$, as was done for the
single ingtegrals.
Noting that
%
%e185 #&#
\begin{eqnarray}
\label{eqA55-36} %
\UU(\xi+\eta, \tilde{\xi}+\eta) :=
\frac{d^2}{d\eta^2} \bigl[ u(\xi+\eta) u(\tilde\xi+\eta) +q(\xi +\eta) q(\tilde
\xi+\eta) \bigr], %
\end{eqnarray}
the integral can be evaluated, and we find that~\eqref{eqA55-40} equals
%
%e186 #&#
\begin{equation}
\label{eqA55-41} %
\frac1{4t^{2/3}} \bigl[ u(x) u
\bigl(x' \bigr) +q(x) q \bigl(x' \bigr) \bigr] +
\mathcal{O} \bigl( t^{-1} \bigr). %
\end{equation}
The error term $\mathcal{O}  ( t^{-1}  )$ follows
from~\eqref{eqA55-23-1-1}.

Combining~\eqref{eqA55-15} and~\eqref{eqA55-41}, we obtain
%
%e187 #&#
\begin{eqnarray}
&&\log \biggl[ \frac{ \Prob\{\tilde\Cro_t \leq x, \tilde\Nes_t
\leq
x'\} } {
\Prob\{\tilde\Cro_t \leq x\} \Prob\{ \tilde\Nes_t \leq x'\} } \biggr]
\nonumber
\\[-8pt]
\\[-8pt]
\nonumber
&&\qquad = \frac{ [ q(x) - u(x) ] [ q(x') - u(x') ]}{4t^{2/3}} + \mathcal{O}
\bigl( t^{-1} \bigr).
\end{eqnarray}
This completes the proof of Theorem~\ref{thm3}.
We note that here the error term is uniform for $x, x'$ in a compact
subset of $\R$
(actually in any semi-infinite interval $[x_0, \infty)$.)

Corollary~\ref{cor1} follows if we show that $\Cov(\tilde\Cro_t,
\tilde\Nes_t) = t^{-2/3}+\mathcal{O}  ( t^{-1}  )$.
This is obtained from
Theorem~\ref{thm3} by using the dominated convergence theorem
if we have tail estimates of $\Prob\{ \tilde\Cro_t \leq x, \tilde
\Nes_t \leq x' \}
-\Prob\{ \tilde\Cro_t < x \} \Prob\{ \tilde\Nes_t < x' \}$ as $|x|,
|x'|\to\infty$ since $\int_{-\infty}^\infty x\, dF'(x)=-1$.
The tail as $x, x'\to+\infty$ can be obtained from the analysis of
this paper.
For the other limits, we need an extension of the analysis of this
paper, but we skip the details in this paper.
See \cite{BDJ,BDR} for a similar question about the convergence of
moments using Toeplitz determinant.

%%%%%%%%%%%%%%%%%%%%%%%%%%
%
% Proofs of Theorems 1.3 and 1.4
%
%%%%%%%%%%%%%%%%%%%%%%%%%%
%s8 #&#
\section{\texorpdfstring{Proof of Theorems~\protect\ref{thm4} and \protect\ref{thm5}}
{Proof of Theorems 1.2 and 1.3}}\label{secmarginalproofs}

Here we evaluate the asymptotics of the marginal distributions $\Prob
\{ \Cro_t \leq j \}$ for $j$ as given by \eqref{eqA55-6}. We reuse
as much as possible the calculations in the previous section. Note that
by symmetry we have $\Prob\{ \Nes_t \leq j \} = \Prob\{ \Cro_t \leq j
\}$. In the process of computing the marginal we will compute as a
by-product asymptotics for $\Prob\{ L_t \le\ell\}$ along the way.

Our starting point is to introduce the change of variables
%
%e188 #&#
\begin{equation}
\label{eq81} \tau= t - 2^{-1} (\eta- x_t) t^{1/3},\qquad
s = t - 2^{-1} (\zeta- x_t) t^{1/3}
\end{equation}
into \eqref{eqA55-3} where, as in the previous section, $x_t$ is
given by \eqref{eqA55-8}. Note that this change of variables differs
from \eqref{eqA55-10} by a shift. Making the substitution we have,
with $j$ and $k$ defined by~\eqref{eqA55-6} [recall~\eqref{eqA55-8}],
%
%e189 #&#
\begin{eqnarray}
\label{eq82} %
\log\Prob\{ \Nes_t \leq j \} &=&
\Ical_1 + \Ical_2 + \mathcal{O} \bigl( e^{-ct}
\bigr),
\nonumber
\\[-8pt]
\\[-8pt]
\nonumber
\log\Prob\{ L_t \leq2j+1 \} &=& 2 \Ical_1 + \mathcal{O}
\bigl( e^{-ct} \bigr), %
\end{eqnarray}
where
%
%e190 #&#
\begin{eqnarray}
\label{eq83} %
\Ical_1 &=& \frac{t^{2/3}}{4} \int
_{x_t}^{x_t+ 2\eps t^{2/3}} \int_{\zeta}^{x_t+2\eps t^{2/3}}
\QQ_j^{\infty}(\tau) \,d \eta \,d\zeta,
\nonumber
\\[-8pt]
\\[-8pt]
\nonumber
\Ical_2 &=& \frac{t^{1/3}}{2} \int_{x_t}^{x_t+2\eps t^{2/3}}
\pi_{2j+1,\infty}(0;\tau) \,d \eta. %
\end{eqnarray}
From~\eqref{eqellt-1}, there is an analogous formula for $\log\Prob
\{ L_t \leq2j \}$, and the analysis below applies to this case too
without many changes. We skip the details for this case.

In order to compute expansions of the above integrals, we need more
detailed calculations than the previous section.
%asymptotics for $\pi_{2j+1,\infty}(0;\tau)$.
Inserting \eqref{eq81} into \eqref{eqA55-6-1}, we have
%
%e191 #&#
\begin{equation}
\label{eq84-1} \gamma(\tau) = 1+ \tfrac{1}{2} \eta t^{-2/3} +
\tfrac{1}{4} \bigl(\eta^2 -\eta x_t \bigr)
t^{-4/3} + \mathcal{O} \bigl( \eta^3 t^{-2} \bigr).
\end{equation}
Then \eqref{sexpansion-1}, with $t$ replaced by $\tau$, becomes % we
%have \marginpar{Is $s(\gamma(\tau)$ correct? I get $\eta+\left(
%
%e192 #&#
\begin{equation}
\label{eq84} s \bigl(\gamma(\tau) \bigr) = \eta+ \bigl( \tfrac{3}{20}
\eta^2 -\tfrac{1}{6} \eta x_t \bigr)
t^{-2/3} + \mathcal{O} \bigl( \eta^3 t^{-4/3} \bigr).
\end{equation}
Inserting these into \eqref{eqregime2-6} we have
%
%e193 #&#
\begin{eqnarray}
\label{eq85}
&& -\pi_{2j+1, \infty}(0;\tau)
\nonumber
\\
&&\qquad= \frac{1}{t^{1/3}} q(\eta) + \frac{1}{t} \biggl[ h(\eta) + \biggl(
\frac{3}{20}\eta-\frac{1}{6}x_t \biggr) \bigl(q(\eta) +
\eta q'(\eta) \bigr) \biggr]\\
&&\qquad\quad{}+ \mathcal{O} \bigl( t^{-4/3}\mathrm{Error}
\bigr),\nonumber
\end{eqnarray}
and it follows from \eqref{eqT-1-1}, \eqref{eqT-1-1-1} (when
$m=\infty$),
and (a slight improvement of) \eqref{eqA55-23} that %
%
%e194 #&#
\begin{eqnarray}
\label{eq86}
&&\QQ_j^\infty(\tau) \nonumber\\
&&\qquad= -2t^{-2/3} q(
\eta)^2
\nonumber
\\
&&\quad\qquad{}- t^{-4/3} \bigl[ 4 q(\eta) h(\eta) + \bigl( \tfrac{3}{5} \eta
- \tfrac{2}{3} x_t \bigr) \bigl( \eta q'(\eta)
q(\eta) + q(\eta)^2 \bigr)
+ q(\eta)q''(\eta)
\\
&&\hspace*{256pt}{} - q'(
\eta)^2 - q(\eta)^4 \bigr] \nonumber\\
&&\qquad\quad{}+ \mathcal{O} \bigl(
t^{-5/3} \mathrm{Error} \bigr).\nonumber
\end{eqnarray}
Here $h$ is as given in~\eqref{eqh1}.
In both the above formulas the $\mathrm{Error}$ term is as defined in \eqref
{eqA55-23-1}, and we recall that its integral introduces terms of
order $\mathcal{O}  ( t^{2/3}  )$.
Now using the identity $q^4 = u +  (q'  )^2 - \eta q^2$ and using
the fact that $q^2 = u'$, $2qq' = u''$ and $q'' = \eta q+2q^3$, it is
direct to check that the terms in square brackets in~\eqref{eq85}
and~\eqref{eq86} can be expressed as perfect derivatives. We find
that
%
%e195 #&#
\begin{eqnarray}
\label{eq87001} %
-\pi_{2j+1, \infty}(0;\tau) &=& \frac{1}{t^{1/3}}
q(\eta) + \frac{1}{t} \U_1(\eta) + \mathcal{O} \bigl(
t^{-4/3}\mathrm{Error} \bigr),
\nonumber
\\[-8pt]
\\[-8pt]
\nonumber
\QQ_j^\infty(\tau) &=& - \frac{2}{t^{2/3}}
u'(\eta) - \frac{1}{t^{4/3} }\U_2(\eta) + \mathcal{O}
\bigl( t^{-5/3} \mathrm{Error} \bigr), %
\end{eqnarray}
where
%
%e196 #&#
\begin{eqnarray}
\label{eq87-1} %
\U_1(\eta)& :=& \frac{1}{5}
\frac{\mathrmm{d}^{ } }{\mathrmm{d} \eta^{ } } \biggl[ u(\eta )q(\eta) - q'(\eta) +
\frac{1}{12} (9 \eta- 10 x_t )\eta q(\eta) \biggr],
\nonumber
\\[-8pt]
\\[-8pt]
\nonumber
\U_2(\eta)&:=& \frac{1}{5}
\frac{\mathrmm{d}^{2} }{\mathrmm{d} \eta^{2} } \biggl[ u(\eta)^2 - q(\eta)^2 +
\frac{1}{6} (9 \eta- 10 x_t )\eta u(\eta) \biggr]. %
\end{eqnarray}
%
%Theorems~\ref{thm4} and~\ref{thm5} follow from inserting this
%formula into~\eqref{eq83} and~\eqref{eq82}.
Inserting this formula into~\eqref{eq83} and~\eqref{eq82}, we
obtain with $x^{(t)}$ and $x_t$ defined by~\eqref{eqtm5-2-1}
and~\eqref{eqm-1}, respectively,
%
%e197 #&#
\begin{eqnarray}
\label{eqthm5-2}
&&\log\Prob \bigl\{ L_t \le2t + t^{1/3} x
\bigr\} \nonumber\\%\\
&&\qquad= \log\FGUE \bigl(x^{(t)} \bigr)
\\
&&\qquad\quad{}- \frac{1}{10t^{2/3}} \biggl[ u(x)^2 - q(x)^2 -
\frac{1}{6} x^2 u(x) \biggr] +\mathcal{O} \bigl(
t^{-1} \bigr)\nonumber
\end{eqnarray}
and
%
%e198 #&#
\begin{equation}
\log\Prob \bigl\{ \Nes_t
\leq t + 2^{-1}t^{1/3}x \bigr\} = \log\FGOE(x_t) +
\frac{ \Next(x)}{t^{2/3}} + \mathcal{O} \bigl( t^{-1} \bigr), %
\end{equation}
where $\Next=\Next(x)$ equals
%
%e199 #&#
\begin{eqnarray}
\label{eq12} \qquad\Next:= \tfrac{1}{20} \bigl[ - \bigl(u(x)-q(x)
\bigr)^2 +2 \bigl(u'(x) - q'(x) \bigr)+
\tfrac{1}{6} x^2 \bigl(u(x) - q(x) \bigr) \bigr] .
\end{eqnarray}
It is easy to check that $20\Next(x)F(x) = - 4F''(x)-\frac13 x^2
F'(x)$ and
$(u(x)^2 - q(x)^2 - \frac{1}{6} x^2 u(x))\FGUE(x)= \FGUE''(x)+ \frac16
x^2\FGUE'(x)$.\footnote{We would like to thank Craig Tracy for pointing
out these relations. Relations like these and many others can be found
in \cite{STracy}.}
Theorems~\ref{thm4} and~\ref{thm5} follow immediately.

%%%%%%%%%%%%%%%%%%%%%%%%%%%%%%%%%%%%%%%
% %
% de-Poissonization %
% %
%%%%%%%%%%%%%%%%%%%%%%%%%%%%%%%%%%%%%%%
%s9 #&#
\section{\texorpdfstring{Proof of Corollary~\protect\ref{thm1}}{Proof of Corollary 1.2}}\label{secde-Poissonization}

%Theorem~\ref{thm1},~\eqref{eq5}, and~\eqref{eqthm5-1} follow from
%the results on their Poissonized counter-part by the following
%de-Poissonization argument \cite{10}.

For a sequence $\{ a_n\}_{n=0}^\infty$, consider its Poissonization
%%(exponential generating function times $e^{-t^2}$):
%
%e200 #&#
\begin{equation}
\phi(t) := e^{-t^2} \sum_{n=0}^\infty
\frac{ (t^2)^n}{n!} a_n.
\end{equation}
A de-Poissonization lemma is that if (a) $0\le a_n \leq1$ and (b)
$a_{n+1} \leq a_n$ for all $n$, then we have for $s \geq1$ and $n \geq2$,
%
%e201 #&#
\begin{equation}
\phi(\sqrt{\mu_n}) - \frac{1}{n^s} \leq a_n \leq
\phi(\sqrt{\nu_n}) + \frac{1}{n^s},
\end{equation}
where
%
%e202 #&#
\begin{equation}
\label{dePoissonmu} \mu_n := n + 2\sqrt{sn\log n},\qquad
\nu_n = n - 2\sqrt{sn \log n}.
\end{equation}
%
%Hence $a_n$ is close to $\phi(\sqrt{n})$ as $n$ tends to infinity.
Lemma 2.5 of \cite{10} is stated for the case when $s=2$, but the proof
can be modified in a straightforward way to obtain the above estimates.
%This is a slight modification of Lemma 5 of \cite{10}.

The de-Poissonization lemma can be applied to
$a_n:= \Prob\{ \mathrm{cr}_n \leq k, \mathrm{ne}_n \leq j \}$ due to
the following lemma.

%
%le9.1 #&#
\begin{lemma}
For each $n \geq0$, and $k,j\ge0$,
%
%e203 #&#
\begin{equation}
\Prob\{
\cro_{n+1} \leq k, \mathrm{ne}_{n+1} \leq j \} \leq
\Prob\{
\cro_{n} \leq k, \mathrm{ne}_{n} \leq j \}.
\end{equation}
\end{lemma}

\begin{pf}
Since $\Prob\{ \mathrm{cr}_{n} \leq k, \mathrm{ne}_{n} \leq j \} =
\frac
{g_{k,j}(n)}{(2n-1)!!}$, where
%
%e204 #&#
\begin{eqnarray}
g_{k,j}(n) := \# \bigl\{ M \in\M_n \dvtx
\cro_n(M) \leq k, \mathrm{ne}_n(M) \leq j \bigr\},
\end{eqnarray}
we need to show that $g_{k,j}(n+1) \leq(2n+1)g_{k,j}(n)$.
% g_{k,j}(n+1) \leq(2n+1)g_{k,j}(n).
The set $\M_{n+1}$ of complete matchings of $[2(n+1)]$ is the union of $(2n+1)$
disjoint subsets $\M_{n+1}^\ell, \ell= 1,\ldots,2n+1$, where
$M_{n+1}^\ell$ is the set of complete matchings of $[2(n+1)]$ such that
$1$ is paired with $\ell$ [i.e., $(1,\ell)$ is an element of the
matching]. By removing the two vertices $1$ and $\ell$, and then
relabeling the vertices, there is a trivial bijection $f_\ell: \M_{n+1}^\ell\mapsto\M_n$. % which preserves $\cro$ and $\nes$.
Clearly, $\mathrm{cr}_{n+1}(M)\ge\mathrm{cr}_n(f_\ell(M))$ and
$\mathrm{ne}_{n+1}(M)\ge\mathrm{ne}_n (f_\ell(M))$ for $M\in\M_{n+1}^\ell$.
This implies that $g_{k,j}(n+1) \leq(2n+1) g_{k,j}(n)$.
%Since $\cro$ and $\nes$ of the $M \in\M_{n+1}^\ell$ is always larger
%than or equal to those of $f_\ell(M)$ with $1$ and $\ell$ removed, we
%find that $g_{k,j}(n+1) \leq(2n+1) g_{k,j}(n)$.
\end{pf}

Hence, since [see~\eqref{eq8}]
%
%e205 #&#
\begin{eqnarray}
\Prob\{ \Cro_t \leq k, \Nes_t \leq j \} &=&
e^{-t^2/2} \sum_{n=0}^\infty
\frac{(t^2/2)^n}{n!} \Prob\{
\cro_n \leq k, \mathrm{ne}_n
\leq j \},
\nonumber
\\[-8pt]
\\[-8pt]
\nonumber
\Prob\{ \Cro_t \leq k \} &=& e^{-t^2/2} \sum
_{n=0}^\infty\frac{(t^2/2)^n}{n!} \Prob\{
\cro_n \leq k \}, % & = \sum_{n=0}^\infty\frac{e^{-t^2/2}(t^2/2)^n}{n!}
\end{eqnarray}
we find that
for each $s \geq1$, $n \geq2$ and $j,k\ge0$,
%
%e206 #&#
\begin{eqnarray}
\label{eqqu1} %
&&\Prob\{
\cro_n \leq k,
\mathrm{ne}_n \leq j \} - \Prob\{
\cro_n \leq k \}
\Prob\{ \mathrm{ne}_n \leq j \}
\nonumber\\
&&\qquad\le\Prob\{ \Cro_{\sqrt{2\nu_n}} \leq k, \Nes_{\sqrt{2\nu_n}} \leq j \}\\
&&\qquad\quad{} - \Prob
\{ \Cro_{\sqrt{2\mu_n}} \leq k \} \Prob\{ \Nes_{\sqrt{2\mu_n}} \leq j \} +
4n^{-s}. \nonumber %
\end{eqnarray}
When $k=\sqrt{2n}+2^{-1} x(2n)^{1/6}$ and $j=\sqrt{2n}+2^{-1}
x'(2n)^{1/6}$, from Theorem~\ref{thm3}, the right-hand side of~\eqref
{eqqu1} is less than or equal to
%
%e207 #&#
\begin{eqnarray}
\label{eqqu2} \quad %
&&\Prob\{ \Cro_{\sqrt{2\nu_n}} \leq k \} \Prob\{
\Nes_{\sqrt{2\nu_n}} \leq j \} - \Prob\{ \Cro_{\sqrt{2\mu_n}} \leq k \} \Prob\{
\Nes_{\sqrt{2\mu_n}} \leq j \}
\nonumber
\\[-8pt]
\\[-8pt]
\nonumber
&&\qquad{}+ 4n^{-s} + \mathcal{O} \bigl( n^{-1/3} \bigr). %
\end{eqnarray}
Now we use Theorem~\ref{thm4} to estimate each of the above probabilities.
Note that
%
%e208 #&#
\begin{eqnarray}
\label{eqqu3} %
\frac{\sqrt{2n}+2^{-1}x(2n)^{1/6}- \sqrt{2\nu_n}}{2^{-1}(2\nu_n)^{1/6}} = x + \frac{4\sqrt{sn\log n}}{(2n)^{1/6}} +
\mathcal{O} \biggl( \frac{\sqrt{\log n}}{n^{1/2}} \biggr). %
\end{eqnarray}
When $\nu_n$ is replaced by $\mu_n$, then the first plus sign on the
right-hand side is changed to the minus sign.
From this, it follows that~\eqref{eqqu2} is bounded above by
$\mathcal{O}  ( \frac{\sqrt{\log n}}{n^{1/6}}  )+4n^{-s}$.
The lower bound is similar. Thus we obtain Corollary~\ref{thm1}.

%%%%%%%%%%%%%%%%%%%%%%%%%%%%%%%%%%%%%%%
% %
% Painlev\'e II model RHP %
% %
%%%%%%%%%%%%%%%%%%%%%%%%%%%%%%%%%%%%%%%
%s10 #&#
\section{A model RHP: Painlev\'e II}\label{secpainleve}
Consider the coupled pair of differential equations for $2\times2$
matrix $\Psi(\zeta, s)$,
\begin{subequations}\label{PIIlaxpair}
%
%e209 #&#
%e210 #&#
\begin{eqnarray}
i \frac{\mathrmm{d}^{ } \Psi}{\mathrmm{d} \zeta^{ } } &=& \bigl(4\zeta^2+s \bigr) [
\sigma_3, \Psi] + \pmatrix{ 2q^2 & 4i\zeta q - 2r
\cr
4i
\zeta q +2r & -2q^2 } \Psi, \label{PIIlax1}
\\
i \frac{\mathrmm{d}^{ } \Psi}{\mathrmm{d} s^{ } } &=& - \zeta[ \sigma_3, \Psi] + \offdiag{iq}
{iq} \Psi, \label{PIIlax2}
\end{eqnarray}
\end{subequations}
where $\sigma_3$ denotes the Pauli matrix $ \bigl({1\atop 0}\enskip {0 \atop -1}\bigr)$ and $[*,*]$ is
the commutator $[A,B] =AB-BA$. The compatibility condition for this
overdetermined system is that $q =q(s)$ satisfy Painlev\'e II $q'' =
sq+2q^3$ and $r =q'(s)$.
This is a representation of the Lax-pair for Painlev\'e II equation
introduced by Flaschka and Newell~\cite{FN76}.

%
%f9 #&#
\begin{figure}

\includegraphics{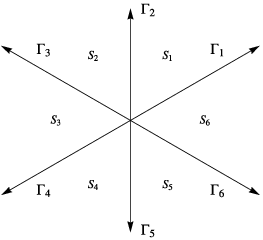}

\caption{The contours $\Gamma_j$ and regions $S_j$ defining $\Psi
(\zeta
,s)$.}\label{figgenericPIIcontours}
\end{figure}

Any solution of \eqref{PIIlax1} is an entire function of $\zeta$. Let
$S_j, j=1,\ldots,6$ denote the sectors
%
%e211 #&#
\begin{equation}
S_j = \biggl\{ \zeta\in\C \dvtx \frac{2j-3}{6}\pi< \arg( \zeta)
< \frac{2j-1}{6} \pi \biggr\},
\end{equation}
and let $\Gamma_j$ denote the outwardly oriented boundary rays (see Figure \ref{figgenericPIIcontours})
%
%e212 #&#
\begin{equation}
\label{eq73} \Gamma_j = \biggl\{ \zeta\in\C \dvtx \arg(\zeta) =
\frac{2j-1}{6} \pi \biggr\}.
\end{equation}
There exists a unique solution $\Psi_j$ of \eqref{PIIlax1} such that
%
%e213 #&#
\begin{equation}
\Psi_j =I + \mathcal{O} \bigl( \zeta^{-1} \bigr)\qquad \mbox{as
} \zeta\to \infty\mbox{ in } S_j,
\end{equation}
and constants $a_j,  j=1,\ldots,6$ such that for $\zeta\in\Gamma_j$
%
%e214 #&#
\begin{eqnarray}
\Psi_{j+1}(\zeta) &=& \Psi_j(\zeta)
\tril{a_j e^{-2i(({4}/{3}) \zeta^3 + s \zeta)}},\qquad j \mbox{ odd, }
\nonumber
\\[-8pt]
\\[-8pt]
\nonumber
\Psi_{j+1}(\zeta) &=& \Psi_j(\zeta) \triu{a_j
e^{2i(({4}/{3}) \zeta^3 + s \zeta)}},\qquad  j \mbox{ even.} %
\end{eqnarray}
Additionally, the constants $a_j$ satisfy
%
%e215 #&#
\begin{equation}
a_{j+3} = a_j, \qquad a_1 a_2
a_3 +a_1 + a_2 +a_3 = 0.
\end{equation}
The parameters $a_j$ depend parametrically on $s,q$ and $r$; in \cite
{FN76} Flaschka and Newell showed that the isomonodromic deformations,
that is, the variations of these parameters that keep the Stokes
multipliers $a_j$ constant, are given by solutions of the Painlev\'e II
equation $q''(s) = s q + 2q(s)^3$ and $r(s) = q'(s)$.

Our particular interest is in the Hastings--McLeod solution of Painlev\'
e II \cite{HM}, which is the unique solution such that
%
%e216 #&#
\begin{eqnarray}
\label{HMasymp} q(s) &= &\Ai(s) \bigl( 1 + {o} ( 1 ) \bigr)\qquad \mbox {as } s \to
\infty,
\nonumber
\\[-8pt]
\\[-8pt]
\nonumber
 q(s) &\sim&\sqrt{-\frac{s}{2}} \qquad\mbox{as } s \to-\infty.
\end{eqnarray}
Let $\Psi(\zeta; s)$ be the solution of \eqref{PIIlax1} with
parameters $s, q=q(s)$ and $r=q'(s)$, where $q(s)$ is the
Hastings--McLeod solution, and let $\mathcal{P}$ denote the set of
poles of $q$ (of which there are infinitely many). Then $\bpsi(\zeta
,s)$ is defined and analytic for $\zeta\in\C\setminus(C_1 \cup
C_2)$ and $s \in\C\setminus\mathcal{P}$. It is known that there are
no poles of $q$ on the real line \cite{HM}. The Stokes multiplier for
the Hastings--McLeod solution are
%
%e217 #&#
\begin{equation}
a_1 = 1,\qquad a_2=0, \qquad a_3 = -1.
\end{equation}
If we reverse the orientation of $\Gamma_3$ and $\Gamma_4$ and define
$C_1 = \Gamma_1 \cup\Gamma_3$ and $C_2 = \Gamma_4 \cup\Gamma_6$ (see Figure \ref{figPIIcontours}), then
$\Psi(\zeta;s)$ solves the following RHP:
%
%f10 #&#
\begin{figure}

\includegraphics{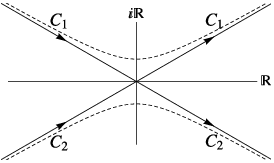}

\caption{The contours defining RHP \protect\ref{rhpPII} related to
the Hastings--McLeod solution of Painlev\'e II.
The contours can be deformed to the dashed
lines without changing the problem statement.} \label{figPIIcontours}
\end{figure}
%
%
%ri10.1 #&#
\begin{rhp}[\textsc{(PII model RHP)}]\label{rhpPII}
Find a $2 \times2$ matrix $\Psi(\zeta;s)$ with the following properties:
\begin{longlist}[(1)]
\item[(1)]$\Psi(\zeta;s)$ is an analytic function of $\zeta$ for
$\zeta\in\C\setminus( C_1 \cup C_2 )$.
\item[(2)] $\Psi(\zeta;s)= I + \mathcal{O}  ( \zeta^{-1}
)$ as $\zeta\to
\infty$
and bounded as $\zeta\to0$.
\item[(3)] The boundary values $\Psi_\pm(\zeta;s)$ satisfy the jump
conditions
%
%e218 #&#
\begin{eqnarray}
\label{PIIjump} \cases{\displaystyle \Psi_+(\zeta;s) = \Psi_-(\zeta;s) \tril{
e^{2i\theta_{\mathit{PII}}}}, &\quad  $\zeta\in C_1,$\vspace*{3pt}
\cr
\displaystyle \Psi_+(\zeta;s) = \Psi_-(
\zeta;s) \triu{-e^{-2i\theta_{\mathit{PII}}}}, & \quad $\zeta\in C_2,$}
\end{eqnarray}
where
%
%e219 #&#
\begin{equation}
\theta_{\mathit{PII}} = \theta_{\mathit{PII}}(\zeta,s) = \tfrac{4}{3}
\zeta^3 +s \zeta\label{PIIphase}.
\end{equation}
\end{longlist}
\end{rhp}

We make two observations which we will need later. First, the
symmetries $-C_1 = C_2$ and $\theta_{\mathit{PII}}(-\zeta,s) = -\theta_{\mathit{PII}}(\zeta,s)$ imply that the solution $\Psi(\zeta,s)$ of RHP~\ref
{rhpPII} satisfies the symmetry
%
%e220 #&#
\begin{eqnarray}
\label{psisymmetry} \Psi(-\zeta,s) = \sigma_1 \Psi(\zeta,s)
\sigma_1,\qquad  \sigma_1 := \lleft(\matrix{ 0& 1
\cr
1& 0}
\rright).
\end{eqnarray}
The second fact is that $\Psi$ admits a uniformly expansion in the
limit as $\zeta\to\infty$ as described in \cite{DZ95}. Specifically,
we have
%
%e221 #&#
\begin{equation}
\Psi(\zeta;s) = I + \frac{\psi_1(s)}{\zeta}+ \frac{\psi_2(s)}{\zeta^2} +
\frac{\psi_3(s)}{\zeta^3} + \mathcal{O} \bigl( \zeta^{-4} \bigr).
\end{equation}
The error term $\mathcal{O}  ( \zeta^{-4}  )$ here depends
on $s$. For our
purpose, we need the dependence on $s$ for $s$ bounded below. An
analysis similar to Section~6 of \cite{DZ95} shows that given $s_0>0$,
there exists a constant $c_0>0$ such that
\begin{subequations}\label{PIIexpansion}
%
%e222 #&#
\begin{equation}
\label{PIIexpansion-a} \Psi(\zeta;s) = I + \frac{\psi_1(s)}{\zeta}+
\frac{\psi_2(s)}{\zeta^2} + \frac{\psi_3(s)}{\zeta^3} + \mathcal {O} \biggl(
\frac {e^{-c_0|s|^{3/2}}}{\zeta^{4}} \biggr).
\end{equation}
The moments $\psi_j(s)$ can be calculated recursively from inserting
the expansion into \eqref{PIIlax2}.
The first three moments are
%
%e223 #&#
\begin{eqnarray}
\label{PIIexpansion-b} %
\psi_1(s) &=&
\frac{1}{2i} \left[\matrix{ -u(s) & q(s)
\cr
-q(s) & u(s) } \right],
\nonumber\\
\psi_2(s) &=& \frac{1}{(2i)^2} \left[\matrix{
\frac{1}{2}u(s)^2 - \frac{1}{2}q(s)^2 & q(s)
u(s) - q'(s)
\cr
q(s)u(s) -q'(s) &
\frac{1}{2}u(s)^2 - \frac{1}{2}q(s)^2 }
\right],
\\
\psi_3(s) &= &\frac{1}{ (2i)^3} \left[\matrix{ \alpha(s) & -
\beta(s)
\cr
\beta(s) & -\alpha(s) } \right], \nonumber %
\end{eqnarray}
where
%
%e224 #&#
%e225 #&#
%e226 #&#
\begin{eqnarray}
u(s) &=& \int_\infty^s q(\xi)^2 \,d
\xi, \label{PIIudef}
\\
\alpha(s) &=& \frac{q(s)^2 u(s)}{2} -\frac{u(s)^3}{6} + \log F(s)^2
- \int_\infty^s q'(
\xi)^2 \,d \xi, \label{PIIexpansion-c}
\\
\beta(s) &=& q'(s)u(s) - q(s) \biggl(s+ \frac{q(s)^2}{2} +
\frac{u(s)^2}{2} \biggr). \label{PIIexpansion-d}
\end{eqnarray}
\end{subequations}
We note that the asymptotic analysis of the RHP for the Painlev\'e
equation implies that for a given $s_0>0$,
%
%e227 #&#
\begin{eqnarray}
\label{eqpsiallsmall} \psi_j(s)= \mathcal{O} \bigl(
e^{-c_0|s|^{3/2}} \bigr),\qquad  j=1,2,3,
\end{eqnarray}
where $c_0$ can be taken as the same constant in the error term
of~\eqref{PIIexpansion-a}.

\section*{Acknowledgments}
J.~B. would like to thank Mishko Mitkovshi for bringing his attention to
the paper of Chen, Deng, Du, Stanley and Yan during the summer graduate
workshop at MSRI. This initiated this paper.
We would also like the thank Christian Krattenthaler, Peter Miller,
Richard Stanley, Craig Tracy and Catherine Yan for helpful communications.

% imsref loaded by akundreckaite, 2012-12-28 15:51:39
% imsref loaded by akundreckaite, 2013-01-03 08:47:41
%

%suskaldyti doi

\printaddresses

\end{document}